\documentclass[a4paper,british]{article}
\usepackage[T1]{fontenc}
\usepackage[latin9]{inputenc}
\usepackage[active]{srcltx}
\usepackage{color}
\usepackage{amsmath}
\usepackage{amsthm}
\usepackage{amssymb}
\usepackage{stmaryrd}
\usepackage{setspace}
\usepackage[numbers]{natbib}

\makeatletter

\theoremstyle{plain}
\newtheorem{thm}{\protect\theoremname}[section]
\theoremstyle{plain}
\newtheorem{lem}[thm]{\protect\lemmaname}
\theoremstyle{plain}
\newtheorem{prop}[thm]{\protect\propositionname}
\theoremstyle{remark}
\newtheorem{rem}[thm]{\protect\remarkname}
\theoremstyle{definition}
\newtheorem{defn}[thm]{\protect\definitionname}
\theoremstyle{definition}
\newtheorem{example}[thm]{\protect\examplename}
\theoremstyle{plain}
\newtheorem{cor}[thm]{\protect\corollaryname}

\usepackage{bbm}
\usepackage{a4wide}
\usepackage{amsmath}
\usepackage{amsfonts}
\usepackage{amssymb}
\usepackage{amsthm}

\usepackage[hyperfootnotes=false]{hyperref}

\allowdisplaybreaks

\numberwithin{equation}{section}

\renewcommand{\tilde}{\widetilde}
\renewcommand{\hat}{\widehat}
\makeatother

\usepackage{babel}
\providecommand{\corollaryname}{Corollary}
\providecommand{\definitionname}{Definition}
\providecommand{\examplename}{Example}
\providecommand{\lemmaname}{Lemma}
\providecommand{\propositionname}{Proposition}
\providecommand{\remarkname}{Remark}
\providecommand{\theoremname}{Theorem}

\begin{document}
\global\long\def\phi{\varphi}%
\global\long\def\epsilon{\varepsilon}%
\global\long\def\theta{\vartheta}%
\global\long\def\E{\mathbb{E}}%
\global\long\def\P{\mathbb{P}}%
\global\long\def\N{\mathbb{N}}%
\global\long\def\Z{\mathbb{Z}}%
\global\long\def\R{\mathbb{R}}%
\global\long\def\F{\mathcal{F}}%
\global\long\def\le{\leqslant}%
\global\long\def\ge{\geqslant}%
\global\long\def\MT{\ensuremath{\clubsuit}}%
\global\long\def\ind{\mathbbm1}%
\global\long\def\d{\,\mathrm{d}}%
\global\long\def\dd{\mathrm{d}}%
 
\global\long\def\subset{\subseteq}%
\global\long\def\supset{\supseteq}%
\global\long\def\argmin{\arg\,\min}%
\global\long\def\bull{{\scriptstyle \bullet}}%
\global\long\def\supp{\operatorname{supp}}%
\global\long\def\sgn{\operatorname{sign}}%
\global\long\def\norm{\interleave}%

\title{Paracontrolled distribution approach\\
to stochastic Volterra equations}
\author{David J. Prömel\thanks{Universität Mannheim; email: proemel@uni-mannheim.de}
~and Mathias Trabs\thanks{Universität Hamburg; email: mathias.trabs@uni-hamburg.de} }
\maketitle
\begin{abstract}
Based on the notion of paracontrolled distributions, we provide existence
and uniqueness results for rough Volterra equations of convolution
type with potentially singular kernels and driven by the newly introduced
class of convolutional rough paths. The existence of such rough paths
above a wide class of stochastic processes including the fractional
Brownian motion is shown. As applications we consider various types
of rough and stochastic (partial) differential equations such as rough
differential equations with delay, stochastic Volterra equations driven
by Gaussian processes and moving average equations driven by Lévy
processes. 
\end{abstract}


\section{Introduction }

Stochastic Volterra equations serve as mathematical models for numerous
random phenomena appearing in various areas such as biology, physics
and mathematical finance. In the present work we consider Volterra
equations of convolution type, which in their simplest form are given
by
\begin{equation}
u(t)=u_{0}+\int_{-\infty}^{t}\phi(t-s)\sigma(u(s))\d\theta(s),\quad t\in\R,\label{eq:intro equation}
\end{equation}
where $\theta\colon\R\to\R^{m}$ is a (random) input signal, e.g.,
an $m$-dimensional (fractional) Brownian motion, $u_{0}\in\R^{n}$,
$\phi\colon\R\to\mathbb{R}$ is the so-called kernel and $\sigma\colon\mathbb{R}^{n}\to\mathcal{L}(\mathbb{R}^{m},\mathbb{R}^{n})$
is a vector field. Since the pioneering works of \citet{Berger1980a,Berger1980b},
stochastic Volterra equations have been studied in different settings
and generality by a vast number of authors, see e.g. \citep{Protter1985,PardouxProtter1990,Cochran1995,Zhang2010}. 

This wide class of equations covers many stochastic differential and
integral equations as special cases such as ordinary stochastic differential
equations, classical stochastic Volterra integral equations, stochastic
equations involving fractional derivatives (noting that singular kernels
correspond to Fourier multipliers) and moving average equations driven
by Lévy processes. Recently, Volterra equations attracted additional
attention from the mathematical finance community because stochastic
Volterra equations with singular kernels $\phi$ constitute very suitable
models for the unpredictable and rough behaviour of volatility in
financial markets, cf. \citep{AbiJaber2017,Thibault2016,ElEuch2016}.

Rough path theory initiated by \citet{Lyons1998} provides an innovative
approach to the theory of stochastic differential equations leading
to many novel insights. One of the fundamental results of rough path
theory is the continuity of the solution map $\theta\mapsto u$, known
as the Itô-Lyons map, for controlled differential equations driven
by rough paths. This continuity statement had significant impact over
the past decades and found many applications, see \citep{Lyons2007,Friz2014}.

The main goal of this article is to develop a pathwise approach to
and a solution theory for Volterra equations driven by rough paths,
which allow for regular as well as singular kernels. In particular,
we prove the local Lipschitz continuity of the Itô-Lyons map for Volterra
equations generalising (in some directions) the above mentioned fundamental
result. Many implications of the rough path theory seem thus to be
feasible for Volterra equations.

For this purpose we first establish the existence of a unique solution
to the Volterra equation~(\ref{eq:intro equation}) driven by singals~$\theta$
with sufficient regularity, based on Littlewood-Paley theory and Bony's
paraproduct. In order to extend the existence and uniqueness results
to a rough path setting, we rely on the notion of paracontrolled distributions,
which was introduced by \citet{Gubinelli2015}. The paracontrolled
distribution approach is particularly suitable for the pathwise analysis
of Volterra equations of convolution type because of the following
two key observations: Firstly, the convolution operator appearing
in~(\ref{eq:intro equation}) fits nicely together with the underlying
Fourier and Littlewood-Paley analysis since the convolution operator
is, for instance, a local operation in the Fourier domain. The second
advantage of paracontrolled distributions is that the driving rough
path~$\theta$ (or the underlying model using the language of regularity
structures \citep{Hairer2014}) can be chosen adapted to the specific
equation which turns out to be essential for the solution theory involving
singular kernels.

Volterra equations driven by rough paths have so far only been studied
by \citet{Deya2009,Deya2011}. They have demonstrated that classical
rough path theory can be utilised to handle Volterra equations driven
by rough paths\textcolor{blue}{.} The approach in \citep{Deya2009,Deya2011}
requires a deep and heavy analysis leading to strong regularity assumptions
on the kernel~$\phi$, namely $\phi\in C^{3}$, and thus excluding
singular kernels. This is mainly caused by relying on the classical
space of (geometric) rough paths, which have been designed to treat
ordinary rough differential equations. Adapting the classical notion
of rough paths, \citet{Gubinelli2010} dealt with the mild formulation
of rough evolution equations associated to analytic semigroups, which
corresponds to infinite dimensional Volterra equations with kernels
given by the semigroups. More recently, \citet{Bayer2017} showed
the existence of a solution to a specific rough Volterra equation
modelling the `rough' volatility appearing on financial markets,
using Hairer's theory of regularity structures~\citep{Hairer2014}. 

Using the flexibility of the paracontrolled distribution approach,
we introduce the notion of convolutional rough paths by including
the convolution kernel~$\phi$ in the definition of the so-called
resonant term. The later notion can be seen as the analogue to geometric
rough paths in the paracontrolled distribution stetting. We prove
that the Itô-Lyons map has a locally Lipschitz continuous extension
from the space of smooth paths to the space of convolutional rough
paths. Hence, the Volterra equation~(\ref{eq:intro equation}) driven
by a level-$2$ convolutional rough path possesses a unique solution.
This ansatz leads to rather weak regularity assumptions on the kernel~$\phi$
requiring less than Lipschitz continuity and thus allowing especially
for singular kernels. 

In addition to the above mentioned modelling, there is a particular
interest in singular kernels, e.g. \citep{Cochran1995,Coutin2001,Wang2008},
because of their links to stochastic differential equations with fractional
derivatives \citep{Lototsky2018} and to a large class of semilinear
stochastic partial differential equations \citep{Zhang2010}. The
here developed paracontrolled distribution approach to Volterra equations
can thus also be viewed as a step towards these applications. However,
exploiting these directions more comprehensively would require extensions
based on higher order paracontrolled calculus, see \citep{Bailleul2016},
or to infinite dimensional spaces, cf. \citep{Martin2018}, which
is beyond the scope of the present article.

While it is necessary for singular kernels to be included in the definition
of the rough path, in the case of regular kernels, say $\phi$ is
at least Lipschitz continuous, the existence of the convolutional
rough path can be reduced to the existence of a generic rough path,
i.e., independent of the specific kernel. Moreover, considering the
regularity of the driving signal in Besov spaces, our analysis builds
on \citep{Promel2015} and interestingly the continuity results hold
for some Volterra equations driven by convolutional rough paths with
jumps, contributing to the recent extension of rough path theory to
càdlàg paths, cf. \citep{Chevyrev2017,Friz2012}.

In order to apply the pathwise solution theory for Volterra equations
driven by convolutional rough paths to stochastic Volterra equtaions,
we construct convolutional rough paths for a large class of stochastic
processes satisfying a hypercontractivity property. Examples include
many Gaussian processes such as fractional Brownian motion with Hurst
index~$H>1/3$. As a consequence, we obtain unique solutions to stochastic
Volterra equations driven by Gaussian processes, extending most literature
which focuses on driving signals given by semi-martinagles. Furthermore,
the approach developed here based on paracontrolled distributions
constitutes a solution theory of stochastic differential equations
in the sense of Stratonovich integration. This complements the related
literature about stochastic Volterra equations, which focuses on (generalization
of) It\^o integration, except, of course, the works \citep{Deya2009,Deya2011,Gubinelli2010}
relying on rough path theory. Another advantage of the pathwise approach
is that it can immediately deal with stochastic Volterra equations
with anticipating coefficients, cf. the seminal work of \citep{PardouxProtter1990}. 

\medskip{}

\textit{Plan of the paper:} In Section~\ref{sec:setting} the functional
analytic foundation is provided. Section~\ref{sec:convolution} establishes
the existence and uniqueness results for Volterra equations. The connection
to the classical rough path theory and the probabilistic construction
of the resonant term for suitable stochastic processes can be found
in Section~\ref{sec:resonant term}. Applications of the pathwise
results to various types of stochastic Volterra equations are presented
in Section~\ref{sec:examples}. Appendix~\ref{sec:appendix} collects
several auxiliary lemmas concerning Besov spaces. 

\subsection{Setting up the Volterra equation}

In the rest of the paper, we study the following class of Volterra
equations of convolution type

\begin{equation}
u(t)=u_{0}(t)+\big(\phi_{1}*(\sigma_{1}(u)\xi_{1})\big)(t)+\big(\phi_{2}*(\sigma_{2}(u)\xi_{2})\big)(t),\quad t\in\R,\label{eq:convol}
\end{equation}
where
\begin{itemize}
\item the convolution operator $*$ is defined by 
\[
(f\ast g)(y):=\int_{\R}f(y-x)g(x)\d x,\quad y\in\R,
\]
with the usual generalization for distributions $f$ and $g$,
\item $u_{0}\colon\mathbb{R}\to\mathbb{R}^{n}$ is the initial condition,
\item $\phi_{j}\colon\R\to\R$ are the kernels (or kernel functions) for
$j=1,2$,
\item $\sigma_{j}\colon\mathbb{R}^{n}\to\mathcal{L}(\mathbb{R}^{m},\mathbb{R}^{n})$
are vector fields for $j=1,2$,
\item $\xi_{1}\colon\R\to\R^{m}$ and $\xi_{2}\colon\R\to\R^{m}$ is a possibly
rough and smoother signal, respectively. 
\end{itemize}
Comparing (\ref{eq:intro equation}) and (\ref{eq:convol}), the signal
$\xi_{1}$ corresponds to the (distributional) derivative of $\theta$
and the integral boundaries $(-\infty,t]$ are included via kernel
functions of the form $\phi_{1}=\ind_{[0,\infty)}\phi$. Throughout
the paper we refer to $\phi_{1}*(\sigma_{1}(u)\xi_{1})$ as the \emph{rough
term} and to $\phi_{2}*(\sigma_{2}(u)\xi_{2})$ as the \emph{drift
term}, making for simplicity the assumption that also $\phi_{1},\xi_{1}$
are less regular than $\phi_{2},\xi_{2}$, respectively. Let us remark
that distinguishing between a rough and a drift term allows for sharper
regularity conditions on the respective vector fields $\sigma_{1}$
and $\sigma_{2}$, cf. the notion of $(p,q)$-rough paths \citep{Lejay2006}.

\section{Bony's paraproduct and Besov spaces\label{sec:setting} }

Let us briefly set up the functional analytic framework. We begin
by recalling the notion of Besov spaces in terms of the Littlewood-Paley
decomposition. For a more general introduction we refer to \citet{Bahouri2011},
\citet{Sawano2018} and \citet{triebel2010}. \\

For the sake of clarification let us mention that $L^{p}(\R^{d},\R^{m\times n})$
denotes the space of Lebesgue $p$-integrable functions with norm
$\|\cdot\|_{L^{p}}$ for $p\in[1,\infty)$ and $L^{\infty}(\R^{d},\R^{m\times n})$
denotes the space of bounded functions with corresponding norm $\|\cdot\|_{\infty}$.
The space of Schwartz functions on $\R^{d}$ is denoted by $\mathcal{S}(\R^{d}):=\mathcal{S}(\R^{d},\R^{m\times n})$
and its dual by $\mathcal{S}^{\prime}(\R^{d}):=\mathcal{S}^{\prime}(\mathbb{R}^{d},\mathbb{R}^{m\times n})$,
which is the space of tempered distributions.

For a function $f\in L^{1}(\R^{d},\R^{m\times n})$ the Fourier transform
and its inverse are defined by

\[
\mathcal{F}f(z):=\int_{\mathbb{R}^{d}}e^{-i\langle z,x\rangle}f(x)\d x\quad\text{and}\quad\mathcal{F}^{-1}f(z):=(2\pi)^{-d}\mathcal{F}f(-z).
\]
If $f\in\mathcal{S}^{\prime}(\R^{d})$, then the usual generalization
of the Fourier transform is considered. 

The Littlewood-Paley theory is based on a localization in the frequency
domain by a \textit{dyadic partition of unity}~$(\chi,\rho)$, that
is, $\chi$ and $\rho$ are non-negative infinitely differentiable
radial functions on $\mathbb{R}^{d}$ such that $\supp\chi\subset\mathcal{B}$
and $\supp\rho\subset\mathcal{A}$ for a ball $\mathcal{B}\subset\R^{d}$
and an annulus $\mathcal{A}\subset\R^{d}$, $\chi(z)+\sum_{j\geq0}\rho(2^{-j}z)=1$
for all $z\in\mathbb{R}^{d}$, $\textup{supp}(\chi)\cap\textup{supp}(\rho(2^{-j}\cdot))=\emptyset$
for $j\geq1$, and $\textup{supp}(\rho(2^{-i}\cdot))\cap\textup{supp}(\rho(2^{-j}\cdot))=\emptyset$
for $\vert i-j\vert>1$. We set throughout 
\[
\rho_{-1}:=\chi\quad\text{and}\quad\rho_{j}:=\rho(2^{-j}\cdot)\quad\text{for }j\ge0.
\]
Given a dyadic partition of unity $(\chi,\rho),$ the \emph{Littlewood-Paley
blocks} are defined by 
\[
\Delta_{-1}f:=\mathcal{F}^{-1}(\rho_{-1}\mathcal{F}f)\quad\text{and}\quad\Delta_{j}f:=\mathcal{F}^{-1}(\rho_{j}\mathcal{F}f)\quad\text{for }j\geq0.
\]
Note that $\Delta_{j}f$ is a smooth function for every $j\geq-1$
and for every $f\in\mathcal{S}^{\prime}(\R^{d})$ one has $f=\sum_{j\geq-1}\Delta_{j}f.$
For $\alpha\in\mathbb{R}$ and $p,q\in[1,\infty]$ the \emph{Besov
space} $\mathcal{B}_{p,q}^{\alpha}(\mathbb{R}^{d},\mathbb{R}^{m\times n})$
is given by 
\begin{align*}
\mathcal{B}_{p,q}^{\alpha}(\mathbb{R}^{d},\mathbb{R}^{m\times n}) & :=\big\{ f\in\mathcal{S}^{\prime}(\mathbb{R}^{d},\mathbb{R}^{m\times n})\ :\ \|f\|_{\alpha,p,q}<\infty\big\}\\
 & \text{with}\quad\|f\|_{\alpha,p,q}:=\Big\|\big(2^{j\alpha}\|\Delta_{j}f\|_{L^{p}}\big)_{j\ge-1}\Big\|_{\ell^{q}}.
\end{align*}
Although the norm $\|\cdot\|_{\alpha,p,q}$ depends on the dyadic
partition $(\chi,\rho)$, different dyadic partitions of unity lead
to equivalent norms (see \citep[Corollary~2.70]{Bahouri2011}). Whenever
the dimension of the image space is clear from the context, we write
$\mathcal{B}_{p,q}^{\alpha}(\R{}^{d}):=\mathcal{B}_{p,q}^{\alpha}(\mathbb{R}^{d},\mathbb{R}^{m\times n})$
and $\mathcal{B}_{p,q}^{\alpha}:=\mathcal{B}_{p,q}^{\alpha}(\mathbb{R},\mathbb{R}^{m\times n})$
and analogous abbreviations for $L^{p}(\R^{d},\R^{m\times n})$. The
special case of Hölder-Zygmund spaces is denoted by $\mathcal{C}^{\alpha}:=\mathcal{B}_{\infty,\infty}^{\alpha}$
with corresponding norms $\|\cdot\|_{\mathcal{C}^{\alpha}}:=\|\cdot\|_{\alpha,\infty,\infty}$
for $\alpha>0$. In the following we will frequently apply embedding
results for Besov spaces, which can be found for example in \citep[Proposition~2.5.7 and Theorem~2.7.1]{triebel2010}.\\

Let us fix the notation $A_{\theta}\lesssim B_{\theta}$, for a generic
parameter $\theta$, meaning that $A_{\theta}\le CB_{\theta}$ for
some constant $C>0$ independent of $\theta$. We write $A_{\theta}\sim B_{\theta}$
if $A_{\theta}\lesssim B_{\theta}$ and $B_{\theta}\lesssim A_{\theta}$.
For integers $j_{\theta},k_{\theta}\in\Z$ we write $j_{\theta}\lesssim k_{\theta}$
if there is some $N\in\N$ such that $j_{\theta}\le k_{\theta}+N$,
and $j_{\theta}\sim k_{\theta}$ if $j_{\theta}\lesssim k_{\theta}$
and $k_{\theta}\lesssim j_{\theta}$.\\

Given $f\in\mathcal{B}_{p_{1,}q_{1}}^{\alpha}(\R^{d})$ and $g\in\mathcal{B}_{p_{2},q_{2}}^{\beta}(\R^{d})$,
we can formally decompose the product $fg$ in terms of Littlewood-Paley
blocks as
\begin{equation}
fg=\sum_{j\geq-1}\sum_{i\geq-1}\Delta_{i}f\Delta_{j}g=T_{f}g+T_{g}f+\pi(f,g)\label{eq:bony decomposition}
\end{equation}
where 
\[
T_{f}g:=\sum_{j\geq-1}\bigg(\sum_{i\leq j-2}\Delta_{i}f\bigg)\Delta_{j}g\quad\text{and}\quad\pi(f,g):=\sum_{\vert i-j\vert\leq1}\Delta_{i}f\Delta_{j}g.
\]
This decomposition was originally introduced by \citet{Bony1981}
and $\pi(f,g)$ is usually called \textit{resonant term}. The following
paraproduct estimates verify the importance of \emph{Bony's decomposition}.
For the proof of this lemma, we refer to \citep[Theorem~2.82 and~2.85]{Bahouri2011}
and \citep[Lemma~2.1]{Promel2015}.
\begin{lem}[Bony's paraproduct estimates]
\label{lem:paraproduct} Let $\alpha,\beta\in\mathbb{R}$ and $p_{1},p_{2},q_{1},q_{2}\in[1,\infty]$
and suppose that 
\[
\frac{1}{p}:=\frac{1}{p_{1}}+\frac{1}{p_{2}}\le1\quad\text{and}\quad\frac{1}{q}:=\frac{1}{q_{1}}+\frac{1}{q_{2}}\le1.
\]

\begin{enumerate}
\item If $(f,g)\in L^{p_{1}}(\R^{d})\times B_{p_{2},q}^{\beta}(\R^{d})$,
then $\|T_{f}g\|_{\beta,p,q}\lesssim\|f\|_{L^{p_{1}}}\|g\|_{\beta,p_{2},q}$.
\item If $\alpha<0$ and $(f,g)\in B_{p_{1},q_{1}}^{\alpha}(\R^{d})\times B_{p_{2},q_{2}}^{\beta}(\R^{d})$,
then $\|T_{f}g\|_{\alpha+\beta,p,q}\lesssim\|f\|_{\alpha,p_{1},q_{1}}\|g\|_{\beta,p_{2},q_{2}}$.
\item If $\alpha+\beta>0$ and $(f,g)\in B_{p_{1},q_{1}}^{\alpha}(\R^{d})\times B_{p_{2},q_{2}}^{\beta}(\R^{d})$,
then $\|\pi(f,g)\|_{\alpha+\beta,p,q}\lesssim\|f\|_{\alpha,p_{1},q_{1}}\|g\|_{\beta,p_{2},q_{2}}$.
\end{enumerate}
\end{lem}

In order to analyse the smoothing property of the convolution operator
$*$ appearing in the Volterra equation~(\ref{eq:convol}), we 
provide the following Young inequality and its proof since the authors
are not aware of a reference for this result in the stated generality.
\begin{lem}[Generalized Young's inequality]
\label{lem:Convolution} Let $\alpha,\beta\in\R$, $d\in\N$ and
$p_{1},p_{2},q_{1},q_{2}\in[1,\infty]$ satisfying
\[
0\le\frac{1}{p}:=\frac{1}{p_{1}}+\frac{1}{p_{2}}-1\le1\quad\text{and}\quad0\le\frac{1}{q}:=\frac{1}{q_{1}}+\frac{1}{q_{2}}\le1.
\]
Then, for any \textup{$f\in\mathcal{B}_{p_{1},q_{1}}^{\alpha}(\R^{d})$
and $g\in\mathcal{B}_{p_{2},q_{2}}^{\beta}(\R^{d})$ we have} $f\ast g\in\mathcal{B}_{p,q}^{\alpha+\beta}(\R^{d})$\textup{
with} 
\[
\|f\ast g\|_{\alpha+\beta,p,q}\lesssim\|f\|_{\alpha,p_{1},q_{1}}\|g\|_{\beta,p_{2},q_{2}}.
\]
\end{lem}

\begin{proof}
The Littlewood-Paley blocks of the convolution satisfy
\[
\Delta_{j}(f\ast g)=\F^{-1}\big[\rho_{j}\F f\F g\big]=\F^{-1}[\rho_{j}^{1/2}\F f]\ast\F^{-1}[\rho_{j}^{1/2}\F g],\quad j\ge-1.
\]
Using Young's inequality for $L^{p}$-spaces, we bound
\[
2^{j(\alpha+\beta)}\|\Delta_{j}(f\ast g)\|_{L^{p}}\le\big(2^{j\alpha}\|\F^{-1}[\rho_{j}^{1/2}\F f]\|_{L^{p_{1}}}\big)\big(2^{j\beta}\|\F^{-1}[\rho_{j}^{1/2}\F g]\|_{L^{p_{2}}}\big).
\]
Hence, by the Cauchy-Schwarz inequality it suffices to show
\begin{equation}
\big\|\big(2^{j\alpha}\|\F^{-1}[\rho_{j}^{1/2}\F f]\|_{L^{p_{1}}}\big)_{j\ge-1}\big\|_{\ell^{q_{1}}}\lesssim\|f\|_{\alpha,p_{1},q_{1}}\label{eq:PrConv}
\end{equation}
(and consequently the analogous estimate holds true for $g$). To
verify (\ref{eq:PrConv}), we decompose $f=\sum_{j}\Delta_{j}f$.
Due to the compact support of $\rho_{j}$ and the classical Young
inequality, we obtain
\begin{align*}
2^{j\alpha}\|\F^{-1}[\rho_{j}^{1/2}\F f]\|_{L^{p_{1}}}\le & 2^{j\alpha}\sum_{j'}\big\|\F^{-1}\big[\rho_{j}^{1/2}\F[\Delta_{j'}f]\big]\big\|_{L^{p_{1}}}\\
\le & 2^{j\alpha}\sum_{|j-j'|\le1}\|\F^{-1}[\rho_{j}^{1/2}]\|_{L^{1}}\|\Delta_{j}f\|_{L^{p_{1}}}\\
\lesssim & \sum_{j'}\big(2^{-(j'-j)\alpha}\ind_{[-1,1]}(j'-j)\big)\big(2^{j'\alpha}\|\Delta_{j'}f\|_{L^{p_{1}}}\big).
\end{align*}
Again by Young's inequality (applied to $\ell^{q_{1}}$) we conclude
\[
\big\|\big(2^{j\alpha}\|\F^{-1}[\rho_{j}^{1/2}\F f]\|_{L^{p_{1}}}\big)_{j\ge-1}\big\|_{\ell^{q_{1}}}\le\big\|\big(2^{-j\alpha}\ind_{[-1,1]}(j)\big)_{j\ge-1}\big\|_{\ell^{1}}\|f\|_{\alpha,p_{1},q_{1}}\le(2^{|\alpha|}+2)\|f\|_{\alpha,p_{1},q_{1}}.
\]
\end{proof}
In order to quantify the regularity of the vector fields appearing
in the Volterra equation~(\ref{eq:convol}) we follow the convention
by Stein (cf.\texttt{~}\citep[Definition~3.1]{Friz2010}): For operator-valued
functions $F\colon\mathbb{R}^{m}\to\mathcal{L}(\mathbb{R}^{n},\mathbb{R}^{m})$
we write $F\in C^{k}$ for $k\in\mathbb{N}$, if $F$ is bounded,
continuous and $k$-times differentiable with bounded and continuous
derivatives. The first and second derivative are denoted by $F^{\prime}$
and $F^{\prime\prime}$, respectively, and higher derivatives by $F^{(k)}$.
The space $C^{k}$ is equipped with the norms 
\[
\|F\|_{\infty}:=\sup_{x\in\R^{m}}\|F(x)\|\quad\mbox{and}\quad\|F\|_{C^{k}}:=\|F\|_{\infty}+\sum_{j=1}^{k}\|F^{(k)}\|_{\infty},
\]
where $\|\cdot\|$ denotes the corresponding operator norms.

\section{Existence and uniqueness results for Volterra equations\label{sec:convolution}}

Let us briefly recall the Volterra equation~(\ref{eq:convol}), which
was given by

\[
u(t)=u_{0}(t)+\big(\phi_{1}*(\sigma_{1}(u)\xi_{1})\big)(t)+\big(\phi_{2}*(\sigma_{2}(u)\xi_{2})\big)(t),\quad t\in\R,
\]
for an initial condition $u_{0}\colon\mathbb{R}\to\mathbb{R}^{n}$,
vector fields $\sigma_{j}\colon\mathbb{R}^{n}\to\mathcal{L}(\mathbb{R}^{m},\mathbb{R}^{n})$,
kernel functions $\phi_{j}\colon\R\to\R$ and driving signals $\xi_{j}\colon\R\to\R^{m}$,
for $j=1,2$. While the convolution is always well-defined for any
function or distribution in a Besov space (cf. Lemma~\ref{lem:Convolution}),
the product $\sigma_{j}(u)\xi_{j}$ requires sufficient Besov regularity
of the involved functions (cf. Lemma~\ref{lem:paraproduct}). This
statement will be made precise in the next subsection.

\subsection{Regular driving signals\label{subsec: young case}}

To analyse the product $\sigma_{j}(u)\xi_{j}$ more carefully, we
suppose that the driving signals satisfy $\xi_{j}\in\mathcal{B}_{p,\infty}^{\beta_{j}-1}$
with $\beta_{j}>0$ and $p\geq2$, for $j=1,2$. We further assume
that the corresponding solution~$u$ of the Volterra equation~(\ref{eq:convol})
fulfills $u\in\mathcal{B}_{p,\infty}^{\alpha}$ for some regularity
$\alpha\ge\frac{1}{p}$. In this case Bony's decomposition~(\ref{eq:bony decomposition}),
the paraproduct estimates (Lemma~\ref{lem:paraproduct}) and the
Besov embedding $B_{p/2,\infty}^{\alpha+\beta_{j}-1}\subset B_{p,\infty}^{\beta_{j}-1}$
applied to the problematic product yields 
\begin{equation}
\sigma_{j}(u)\xi_{j}=\underbrace{T_{\sigma_{j}(u)}\xi_{j}}_{\in\mathcal{B}_{p,\infty}^{\beta_{j}-1}}+\underbrace{\pi(\sigma_{j}(u),\xi_{j})}_{\in\mathcal{B}_{p/2,\infty}^{\alpha+\beta_{j}-1}}+\underbrace{T_{\xi_{j}}\sigma_{j}(u)}_{\mathcal{B}_{p/2,\infty}^{\alpha+\beta_{j}-1}}\in\mathcal{B}_{p,\infty}^{\beta_{j}-1}\quad\mbox{if }\alpha+\beta_{j}>1,\,p\ge2.\label{eq:bonyDecompYoung}
\end{equation}
Notice that the Young type condition $\alpha+\beta_{j}>1$ is crucial
for the regularity estimate of the resonant term $\pi(\sigma_{j}(u),\xi_{j})$.
If $\phi_{j}\in\mathcal{B}_{1,\infty}^{\gamma_{j}}$ for some $\gamma_{j}\ge0$,
then Young's inequality (Lemma~\ref{lem:Convolution}) combined with
(\ref{eq:bonyDecompYoung}) yields
\[
\phi_{j}\ast(\sigma_{j}(u)\xi_{j})\in\mathcal{B}_{p,\infty}^{\beta_{j}+\gamma_{j}-1}.
\]
In view of the Volterra equation~(\ref{eq:convol}) we obtain the
relationship
\[
\alpha=\min\{\beta_{1}+\gamma_{1}-1,\beta_{2}+\gamma_{2}-1\}.
\]
In the following we associate the ``rougher'' signal with the first
convolution term and thus assume $\beta_{1}\le\beta_{2}$ and $\gamma_{1}\le\gamma_{2}$.
The Young type condition $\alpha+\beta_{j}>1$ is then equivalent
to $2\beta_{1}+\gamma_{1}>2$. \\

Applying a fixed point argument, we first prove the existence of a
unique solution to the Volterra equation~(\ref{eq:convol}) in this
Young setting. Afterwards we will relax the regularity assumptions
allowing for a more irregular driving signal~$\xi_{1}$ in (\ref{eq:convol}),
see Subsection~\ref{subsec:rough driving singals}.
\begin{prop}
\label{prop:YoungLip} Let $p\ge2$, $0<\beta_{1}\le\beta_{2}$ and
$0<\gamma_{1}\le\gamma_{2}$ such that
\[
\alpha:=\beta_{1}+\gamma_{1}-1\in(1/p,1]\quad\text{and}\quad2\beta_{1}+\gamma_{1}>2.
\]
Suppose $u_{0}\in\mathcal{B}_{p,\infty}^{\alpha}$, $\xi_{j}\in\mathcal{B}_{p,\infty}^{\beta_{j}-1}$,
$\phi_{j}\in\mathcal{B}_{1,\infty}^{\gamma_{j}}$ and $\sigma_{j}\in C^{2}$
with $\sigma_{j}(0)=0$, for $j=1,2$. If $\max_{j=1,2}\|\sigma_{j}\|_{C^{2}}$
is sufficiently small depending on $\max_{j=1,2}\|\phi_{j}\|_{\gamma_{j},1,\infty}$,
$\max_{j=1,2}\|\xi_{j}\|_{\beta_{j}-1,p,\infty}$ and $\|u_{0}\|_{\alpha,p,\infty}$,
then the Volterra equation~(\ref{eq:convol}) has a unique solution
$u\in\mathcal{B}_{p,\infty}^{\alpha}$.
\end{prop}

Let us remark that the assumption $\alpha>\frac{1}{p}$ in Proposition~\ref{prop:YoungLip}
is only used for the embedding $\mathcal{B}_{p,\infty}^{\alpha}\subset L^{\infty}$.
If we separately control the norms $\|\cdot\|_{\alpha,p.\infty}$
and $\|\cdot\|_{\infty}$ of the solution $u$, we may allow for $u\in\mathcal{B}_{p,\infty}^{1/p}$.
This implies that the solution~$u$ of the Volterra equation~(\ref{eq:convol})
may have jumps but these jumps can only come from the initial condition
$u_{0}$. This observation leads to the next proposition. 
\begin{prop}
\label{prop:YoungLip with jumps} Let $p\ge2$, $0<\beta_{1}\le\beta_{2}$,
$0<\gamma_{1}\le\gamma_{2}$ such that
\[
\beta_{1}+\gamma_{1}-1\in(1/p,1]\quad\text{and}\quad\beta_{1}+1/p>1.
\]
Suppose $u_{0}\in\mathcal{B}_{p,\infty}^{1/p}\cap L^{\infty}$, $\xi_{j}\in\mathcal{B}_{p,\infty}^{\beta_{j}-1}$,
$\phi_{j}\in\mathcal{B}_{1,\infty}^{\gamma_{j}}$ and $\sigma_{j}\in C^{2}$
with $\sigma_{j}(0)=0$, for $j=1,2$. If $\max_{j=1,2}\|\sigma_{j}\|_{C^{2}}$
is sufficiently small depending on $\|u_{0}\|_{\frac{1}{p},p,\infty}+\|u_{0}\|_{\infty}$,
$\max_{j=1,2}\|\phi_{j}\|_{\gamma_{j},1,\infty}$ and $\max_{j=1,2}\|\xi_{j}\|_{\beta_{j}-1,p,\infty}$,
then the Volterra equation~(\ref{eq:convol}) has a unique solution
$u\in\mathcal{B}_{p,\infty}^{1/p}\cap L^{\infty}$.
\end{prop}

Let us remark that Proposition~\ref{prop:YoungLip} is not a corollary
of Proposition~\ref{prop:YoungLip with jumps}. However, since the
corresponding proofs work analogously, we present here only the proof
of Proposition~\ref{prop:YoungLip with jumps} in order to avoid
redundance.
\begin{proof}[Proof of Proposition~\ref{prop:YoungLip with jumps}]
We study the solution map
\[
\Phi\colon\mathcal{B}_{p,\infty}^{\frac{1}{p}}\cap L^{\infty}\to\mathcal{B}_{p,\infty}^{\frac{1}{p}}\cap L^{\infty},\quad v\mapsto u:=u_{0}+\phi_{1}*\big(\sigma_{1}(v)\xi_{1}\big)+\phi_{2}*\big(\sigma_{2}(v)\xi_{2}\big).
\]
If $\Phi$ is a well-defined map and a contraction, then the assertion
follows from Banach's fixed point theorem.

\emph{Step 1:} \emph{The map $\Phi$ is well-defined.} Indeed, by
Young's inequality (Lemma~\ref{lem:Convolution}), the Besov embeddings
$B_{p/2,\infty}^{1/p+\beta_{j}-1}\subset B_{p,\infty}^{\beta_{j}-1}$,
$B_{p,\infty}^{\gamma_{j}+\beta_{j}-1}\subset B_{p,\infty}^{1/p}$
and $B_{p,\infty}^{\beta_{j}+\gamma_{j}-1}\subset L^{\infty}$ for
$j=1,2$ and Bony's decomposition we have
\begin{align*}
 & \|\Phi(v)\|_{\frac{1}{p},p,\infty}+\|\Phi(v)\|_{\infty}\\
 & \quad\lesssim\|u_{0}\|_{\frac{1}{p},p,\infty}+\|u_{0}\|_{\infty}+\sum_{j=1,2}\|\phi_{j}\|_{\gamma_{j},1,\infty}\|\sigma_{j}(v)\xi_{j}\|_{\beta_{j}-1,p,\infty}\\
 & \quad\lesssim\|u_{0}\|_{\frac{1}{p},p,\infty}+\|u_{0}\|_{\infty}+\sum_{j=1,2}\big(\|T_{\sigma_{j}(v)}\xi_{j}\|_{\beta_{j}-1,p,\infty}+\|\pi(\sigma_{j}(v),\xi_{j})\|_{\frac{1}{p}+\beta_{j}-1,p/2,\infty}\\
 & \qquad\quad\qquad\qquad\qquad\hspace{6em}+\|T_{\xi_{j}}\sigma_{j}(v)\|_{\frac{1}{p}+\beta_{j}-1,p/2,\infty}\big)\|\phi_{j}\|_{\gamma_{j},1,\infty}.
\end{align*}
The paraproduct estimates (Lemma~\ref{lem:paraproduct}) and Lemma~\ref{lem:differenceBesovNorm}
yield 
\begin{align}
\begin{split} & \|\Phi(v)\|_{\frac{1}{p},p,\infty}+\|\Phi(v)\|_{\infty}\\
 & \quad\lesssim\|u_{0}\|_{\frac{1}{p},p,\infty}+\|u_{0}\|_{\infty}+\sum_{j=1,2}\|\phi_{j}\|_{\gamma_{j},1,\infty}\big(\|\sigma_{j}(v)\|_{\frac{1}{p},p,\infty}+\|\sigma_{j}(v)\|_{\infty}\big)\|\xi_{j}\|_{\beta_{j}-1,p,\infty}\\
 & \quad\lesssim\|u_{0}\|_{\frac{1}{p},p,\infty}+\|u_{0}\|_{\infty}+\big(\|v\|_{\frac{1}{p},p,\infty}+\|v\|_{\infty}\big)\sum_{j=1,2}\|\phi_{j}\|_{\gamma_{j},1,\infty}\|\sigma_{j}\|_{C^{1}}\|\xi_{j}\|_{\beta_{j}-1,p,\infty}.
\end{split}
\label{eq:PhiBounded}
\end{align}
Hence, $\Phi(v)\in\mathcal{B}_{p,\infty}^{\frac{1}{p}}\cap L^{\infty}$
for every $v\in\mathcal{B}_{p,\infty}^{\frac{1}{p}}\cap L^{\infty}.$

\emph{Step 2: Invariance of $\Phi$}. We now verify that $\Phi$ maps
the ball 
\[
\mathcal{B}_{K}:=\left\{ v\in\mathcal{B}_{p,\infty}^{\frac{1}{p}}\cap L^{\infty}\::\:\|v\|_{\frac{1}{p},p,\infty}+\|v\|_{\infty}\leq2K^{2}\right\} \subset\mathcal{B}_{p,\infty}^{\frac{1}{p}}\cap L^{\infty}
\]
into itself for some suitable constant $K\in\mathbb{R}$. Due to (\ref{eq:PhiBounded}),
there exists some $K\ge1$ such that $\|u_{0}\|_{\frac{1}{p},p,\infty}+\|u_{0}\|_{\infty}\le K$
and 
\begin{align*}
 & \|\Phi(v)\|_{\frac{1}{p},p,\infty}+\|\Phi(v)\|_{\infty}\\
 & \quad\le K\bigg(\|u_{0}\|_{\frac{1}{p},p,\infty}+\|u_{0}\|_{\infty}+\big(\|v\|_{\frac{1}{p},p,\infty}+\|v\|_{\infty}\big)\sum_{j=1,2}\|\phi_{j}\|_{\gamma_{j},1,\infty}\|\sigma_{j}\|_{C^{1}}\|\xi_{j}\|_{\beta_{j}-1,p,\infty}\bigg).
\end{align*}
If $\max_{j=1,2}\|\sigma_{j}\|_{C^{1}}$ is sufficiently small such
that 
\[
\max_{j=1,2}\|\phi_{j}\|_{\gamma_{j},1,\infty}\|\sigma_{j}\|_{C^{1}}\|\xi_{j}\|_{\beta_{j}-1,p,\infty}\leq\frac{1}{4K},
\]
then for any $v\in\mathcal{B}_{K}$ we obtain $\|\Phi(v)\|_{\frac{1}{p},p,\infty}+\|\Phi(v)\|_{\infty}\leq K^{2}+K^{2}\leq2K^{2}.$

\emph{Step 3: $\Phi$ is a contraction.} To deduce the Lipschitz continuity
of $\Phi$ on $\mathcal{B}_{K}$, let $v_{1},v_{2}\in\mathcal{B}_{K}$.
By Young's inequality (Lemma~\ref{lem:Convolution}) and the auxiliary
Lemmas~\ref{lem:product} and \ref{lem:differenceBesovNorm} we deduce
\begin{align*}
 & \|\Phi(v_{1})-\Phi(v_{2})\|_{\frac{1}{p},p,\infty}+\|\Phi(v_{1})-\Phi(v_{2})\|_{\infty}\\
 & \quad\lesssim\sum_{j=1,2}\|\phi_{j}\|_{\gamma_{j},1,\infty}\big(\|\sigma_{j}(v_{1})-\sigma_{j}(v_{2})\|_{\frac{1}{p},p,\infty}+\|\sigma_{j}(v_{1})-\sigma_{j}(v_{2})\|_{\infty}\big)\|\xi_{j}\|{}_{\beta_{j}-1,p,\infty}\\
 & \quad\lesssim\Big(\sum_{j=1,2}\|\sigma_{j}\|_{C^{2}}\|\phi_{j}\|_{\gamma_{j},1,\infty}\|\xi_{j}\|{}_{\beta_{j}-1,p,\infty}\Big)\big(1+\|v_{1}\|_{\frac{1}{p},p,\infty}+\|v_{1}\|_{\infty}+\|v_{2}\|_{\frac{1}{p},p,\infty}+\|v_{2}\|_{\infty}\big)\\
 & \qquad\times\big(\|v_{1}-v_{2}\|_{\frac{1}{p},p,\infty}+\|v_{1}-v_{2}\|_{\infty}\big)\\
 & \quad\lesssim\max_{j=1,2}\|\sigma_{j}\|_{C^{2}}\Big(\sum_{j=1,2}\|\phi_{j}\|_{\gamma_{j},1,\infty}\|\xi_{j}\|{}_{\beta_{j}-1,p,\infty}\Big)(1+4K^{2})\big(\|v_{1}-v_{2}\|_{\frac{1}{p},p,\infty}+\|v_{1}-v_{2}\|_{\infty}\big).
\end{align*}
In conclusion, $\Phi\colon\mathcal{B}_{K}\to\mathcal{B}_{K}$ is Lipschitz
continuous and it is a contraction for sufficiently small $\max_{j=1,2}\|\sigma_{j}\|_{C^{2}}$
depending on $\|u_{0}\|_{\frac{1}{p},p,\infty}+\|u_{0}\|_{\infty}$,
$\max_{j=1,2}\|\phi_{j}\|_{\gamma_{j},1,\infty}$ and $\max_{j=1,2}\|\xi_{j}\|_{\beta_{j}-1,p,\infty}$. 
\end{proof}
\begin{rem}
One can bypass the flatness condition on the vector fields $\sigma_{1}$
and $\sigma_{2}$, that is, $\max_{j=1,2}\|\sigma_{j}\|_{C^{2}}$
was assumed to be sufficiently small, by assuming that the kernel
functions $\phi_{1},\phi_{2}$ as well as the driving signals $\xi_{1},\xi_{2}$
are supported on the positive real line, cf. Subsection~\ref{subsec:global solution}.
\end{rem}

\subsection{Rough driving signals\label{subsec:rough driving singals}}

The regularity assumptions on the driving signals proposed in Proposition~\ref{prop:YoungLip},
for obtaining a unique solution to the Volterra equation~(\ref{eq:convol}),
are usually too strong for applications in probability theory. Namely,
we have imposed the smoothness condition $\alpha+\beta_{1}>1$, which
means $\beta_{1}>\frac{2-\gamma_{1}}{2}$. For instance, for ordinary
differential equations we have $\gamma_{1}=1$ and thus $\beta_{1}>\frac{1}{2}$
excluding stochastic differential equations driven by the Brownian
motion. In the sequel we will generalise this condition for the first
convolution term in (\ref{eq:convol}) to $2\alpha+\beta_{1}>1$ being
equivalent to $\beta_{1}>\frac{3-2\gamma_{1}}{3}$. In the case $\gamma_{1}=1$
we then require $\beta_{1}>\frac{1}{3}$ which is in line with the
classical rough path theory with one iterated integral. This paves
the way for a wide range of applications of our results to, e.g.,
fractional Brownian motion, martingales and Lévy processes, see Section~\ref{sec:examples}.

As discussed before, under the weaker regularity condition $\beta_{1}>\frac{3-2\gamma_{1}}{3}$
one main difficulty is to give a rigorous meaning to the product $\sigma_{1}(u)\xi_{1}$,
cf. (\ref{eq:bonyDecompYoung}). To overcome this issue, we adapt
the paracontrolled approach introduced by \citet{Gubinelli2015}.
In order to profit from the smoothing effect of the convolution with
$\phi_{1}$, we choose a paracontrolled ansatz that reflects the convolution
structure of equation~(\ref{eq:convol}). \\

Abbreviating the regular terms by $u_{0,2}:=u_{0}+\phi_{2}\ast(\sigma_{2}(u)\xi_{2})$
and using Bony's decomposition~(\ref{eq:bony decomposition}), we
may write
\[
u=u_{0,2}+\phi_{1}\ast\big(T_{\sigma_{1}(u)}\xi_{1}+\pi(\sigma_{1}(u),\xi_{1})+T_{\xi_{1}}\sigma_{1}(u)\big).
\]
Since the term $\phi_{1}\ast T_{\sigma_{1}(u)}\xi_{1}$ is the least
regular one, we choose the ansatz:
\[
u=u_{0,2}+\phi_{1}\ast T_{\sigma_{1}(u)}\xi_{1}+u^{*}
\]
with remainder 
\begin{equation}
u^{*}:=\phi_{1}\ast\big(\pi(\sigma_{1}(u),\xi_{1})+T_{\xi_{1}}\sigma_{1}(u)\big),\label{eq:u star}
\end{equation}
which is of regularity $\alpha+\beta_{1}-1+\gamma_{1}=2\alpha$ assuming
everything is well-defined. However, this is a priori not true due
to the resonant term $\pi(\sigma_{1}(u),\xi_{1})$. To analyze this
term, we use a linearization of $\sigma_{1}(u)$ (see \citep[Proposition~4.1]{Promel2015})
and again the ansatz for $u$ to decompose the critical term $\pi(\sigma_{1}(u),\xi_{1})$
into
\begin{align}
\begin{split}\pi(\sigma_{1}(u),\xi_{1}) & =\sigma_{1}'(u)\pi(u,\xi_{1})+\Pi_{\sigma_{1}}(u,\xi_{1})\\
 & =\sigma_{1}'(u)\big(\pi(u_{0,2},\xi_{1})+\pi(\phi_{1}\ast T_{\sigma_{1}(u)}\xi_{1},\xi_{1})+\pi(u^{*},\xi_{1})\big)+\Pi_{\sigma_{1}}(u,\xi_{1}),
\end{split}
\label{eq:criticalTerm}
\end{align}
where 
\begin{equation}
\Pi_{\sigma_{1}}(u,\xi_{1}):=\pi(\sigma_{1}(u),\xi_{1})-\sigma_{1}'(u)\pi(u,\xi_{1})\in\mathcal{B}_{p/3,\infty}^{2\alpha+\beta_{1}-1}.\label{eq:Pi}
\end{equation}
At this point the resonant term $\pi(\phi_{1}\ast T_{\sigma_{1}(u)}\xi_{1},\xi_{1})$
is not yet well-defined. In order to continue our analysis, we need
to compare 
\[
\pi(\phi_{1}\ast T_{\sigma_{1}(u)}\xi_{1},\xi_{1})\quad\mbox{and}\quad\pi(\phi_{1}\ast\xi_{1},\xi_{1}),
\]
which is indeed possible thanks to the following lemma.
\begin{lem}
\label{lem:convolution paracontrolled} Suppose there exists a constant
$r\in\R$ such that for some $\gamma\ge0$
\[
\phi\in\mathcal{B}_{1,\infty}^{\gamma}\quad\mbox{and}\quad(\cdot-r)\phi\in\mathcal{B}_{1,\infty}^{\gamma+1}.
\]
If $f\in\mathcal{B}_{p_{1},\infty}^{\alpha}$ and $g\in\mathcal{B}_{p_{2},\infty}^{\beta}$
with $\alpha\in(0,1)$, $\beta\in\R$ and $p_{1},p_{2}\in[1,\infty]$
such that $\frac{1}{p}:=\frac{1}{p_{1}}+\frac{1}{p_{2}}\le1$, then
\begin{align*}
R_{\phi}(f,g) & :=\phi\ast T_{f}g-T_{f(\cdot-r)}(\phi\ast g)\in\mathcal{B}_{p,\infty}^{\alpha+\beta+\gamma}
\end{align*}
with 
\[
\big\| R_{\phi}(f,g)\big\|_{\alpha+\beta+\gamma,p,\infty}\,\lesssim\big(\|\phi\|_{\gamma,1,\infty}+\|(\cdot+r)\phi\|_{\gamma+1,1,\infty}\big)\|f\|_{\alpha,p_{1},\infty}\|g\|_{\beta,p_{2},\infty}.
\]
\end{lem}

Applying Lemma~\ref{lem:convolution paracontrolled}, we can write
the ``undefined'' resonant term $\pi(\phi_{1}\ast T_{\sigma_{1}(u)}\xi_{1},\xi_{1})$
in (\ref{eq:criticalTerm}) as
\begin{align}
\begin{split} & \pi(\phi_{1}\ast T_{\sigma_{1}(u)}\xi_{1},\xi_{1})\\
 & \quad=\pi(T_{\sigma_{1}(u(\cdot-r_{1}))}(\phi_{1}\ast\xi_{1}),\xi_{1})+\pi(R_{\phi_{1}}(\sigma_{1}(u),\xi_{1}),\xi_{1})\\
 & \quad=\sigma_{1}(u(\cdot-r_{1}))\pi(\phi_{1}\ast\xi_{1},\xi_{1})+\Gamma(\sigma_{1}(u(\cdot-r_{1})),\phi_{1}\ast\xi_{1},\xi_{1})+\pi(R_{\phi_{1}}(\sigma_{1}(u),\xi_{1}),\xi_{1}),
\end{split}
\label{eq:ansatzCirtTerm}
\end{align}
for some $r_{1}\in\R$ and where we used the commutator
\[
\Gamma(f,g,h):=\pi(T_{f}g,h)-f\pi(g,h),\quad f,g,h\in\mathcal{S}^{\prime}(\R),
\]
 satisfying
\begin{equation}
\|\Gamma(f,g,h)\|_{a+b+c,p/3,q}\lesssim\|f\|_{a,p,\infty}\|g\|_{b,p,\infty}\|h\|_{c,p,\infty},\label{eq:commutator}
\end{equation}
for $p\ge3$, $a\in(0,1)$ and $b,c\in\R$ with $a+b+c>0$ and $b+c<0$,
see the so-called commutator lemma~\citep[Lemma~4.4]{Promel2015}.
We thus have reduced the critical term $\pi(\sigma_{1}(u),\xi_{1})$
to the resonant term $\pi(\phi_{1}\ast\xi_{1},\xi_{1})$. The latter
one does not depend on the particular equation (\ref{eq:convol})
in the sense that it neither depends on $u$ nor on $\sigma_{1}$,
but only on the signal $\xi_{1}$ and the convolution kernel $\phi_{1}$.
\begin{proof}[Proof of Lemma~\ref{lem:convolution paracontrolled}]
\noindent  We may assume that $f\in\mathcal{B}_{p_{1},\infty}^{\alpha+1/p_{1}+1}$
being a dense subset of $\mathcal{B}_{p_{1},\infty}^{\alpha}$ such
that the result will follow by continuity. We will use the notation
$S_{j}f:=\sum_{k<j-1}\Delta_{j}f$ for $f\in\mathcal{S}'(\R)$. Noting
that 
\begin{equation}
\Delta_{j}(\phi\ast g)=\F^{-1}\rho_{j}\ast\phi\ast g=\phi\ast(\Delta_{j}g),\label{eq:DeltaPhi}
\end{equation}
and since $\sum_{j}\Delta_{j}(\phi\ast g)$ converges if $\sum_{j}\Delta_{j}g$
converges by Lemma~\ref{lem:Convolution}, we have
\[
\phi\ast T_{f}g(x)-T_{f(\cdot-r)}(\phi\ast g)(x)=\sum_{j\geq-1}R_{j}(x)
\]
with
\begin{align}
\begin{split}R_{j}(x):= & \phi\ast\big(S_{j-1}f\Delta_{j}g\big)(x)-S_{j-1}f(x-r)(\phi\ast\Delta_{j}g)(x)\\
= & \int_{\R}\phi(x-z)\big(S_{j-1}f(z)-S_{j-1}f(x-r)\big)\Delta_{j}g(z)\d z\\
= & \int_{0}^{1}\int_{\R}(z-x+r)\phi(x-z)S_{j-1}f'(x-r+t(z-x+r))\Delta_{j}g(z)\d z\,\dd t
\end{split}
\label{eq:proofConvPhi}
\end{align}
where we apply Fubini's theorem in the last line using that $f'$
is bounded. Since $(y-r)\phi(y)\in L^{1}(\R)\subset\mathcal{B}_{1,\infty}^{1}$,
the Fourier transform of $R_{j}$ is well-defined and we have
\begin{align*}
\F R_{j}(\xi) & =\int_{0}^{1}\int_{\R^{2}}e^{i\xi(x-z)}(z-x+r)\phi(x-z)e^{i\xi z}S_{j-1}f'(x-r+t(z-x+r))\Delta_{j}g(z)\d z\,\dd x\,\dd t\\
 & =-\int_{0}^{1}\int_{\R^{2}}e^{i\xi x}(x-r)\phi(x)e^{i\xi z}S_{j-1}f'(x+z-r-tx+tr))\Delta_{j}g(z)\d z\,\dd x\,\dd t\\
 & =-\int_{0}^{1}\int_{\R}e^{i\xi x}(x-r)\phi(x)\F\big[S_{j-1}\big(f'(\cdot+(1-t)(x-r))\big)\Delta_{j}g\big](\xi)\d x\,\dd t.
\end{align*}
Since $\F\big[\Delta_{j}gS_{j-1}\big(f'(\cdot+(1-t)(x-r))\big)\big]$
is supported on an annulus with radius of order $2^{j}$ (uniformly
in $t$ and $x$), we conclude that $\Delta_{k}R_{j}=0$ if $|k-j|\gtrsim1$.
Consequently, 
\begin{equation}
\|\phi\ast T_{f}g(x)-T_{f(\cdot-r)}(\phi\ast g)(x)\|_{\alpha+\beta+\gamma,p,\infty}\lesssim\sup_{k\geq-1}2^{k(\alpha+\beta+\gamma)}\sum_{j\sim k}\|\Delta_{k}R_{j}\|_{L^{p}}.\label{eq:estimatebesov}
\end{equation}
Let us introduce $\phi_{k}:=\F^{-1}\rho_{k}\ast\phi$ and recall the
operator $[\Delta_{k},f]g:=\Delta_{k}(fg)-f\Delta_{k}g$. We have
\begin{align}
\begin{split}\Delta_{k}R_{j} & =\phi_{k}\ast\big(S_{j-1}f\Delta_{j}g\big)-\Delta_{k}\big(S_{j-1}f(\cdot-r)(\phi\ast\Delta_{j}g)\big)\\
 & =\phi_{k}\ast\big(S_{j-1}f\Delta_{j}g\big)-S_{j-1}f(\cdot-r)(\phi_{k}\ast\Delta_{j}g)-[\Delta_{k},S_{j-1}f(\cdot-r)](\phi\ast\Delta_{j}g).
\end{split}
\label{eq:DeltaR}
\end{align}
For the third term in the above display \citep[Lemma~4.3]{Promel2015},
(\ref{eq:DeltaPhi}) and Lemma~\ref{lem:Convolution} yield
\begin{align}
\begin{split}\sum_{j\sim k}\big\|[\Delta_{k},S_{j-1}f(\cdot-r)](\phi\ast\Delta_{j}g)\big\|_{L^{p}} & \leq\sum_{j\sim k}2^{-k\alpha}\|S_{j-1}f\|_{\alpha,p_{1},\infty}\|\Delta_{j}(\phi\ast g)\|_{L^{p_{2}}}\\
 & \lesssim\sum_{j\sim k}2^{-k\alpha}\|f\|_{\alpha,p_{1},\infty}2^{-j(\beta+\gamma)}\|\phi\ast g\|_{\beta+\gamma,p_{2},\infty}\\
 & \lesssim2^{-k(\alpha+\beta+\gamma)}\|\phi\|_{\gamma,1,\infty}\|f\|_{\alpha,p_{1},\infty}\|g\|_{\beta,p_{2},\infty}.
\end{split}
\label{eq:estimate1}
\end{align}
Exactly as in (\ref{eq:proofConvPhi}) the first two terms in (\ref{eq:DeltaR})
can be written as
\begin{align*}
 & \phi_{k}\ast\big(S_{j-1}f\Delta_{j}g\big)(x)-S_{j-1}f(x-r)(\phi_{k}\ast\Delta_{j}g)(x)\\
 & \qquad=\int_{0}^{1}\int_{\R}(z-x+r)\phi_{k}(x-z)\Delta_{j}g(z)S_{j-1}f'(x+t(z-x)+(t-1)r)\d z\,\dd t.
\end{align*}
Abbreviating $\tilde{\phi}_{k}(x)=(x-r)\phi_{k}(x)$, Hölder's inequality
gives
\begin{align*}
 & \big|\phi_{k}\ast\big(S_{j-1}f\Delta_{j}g\big)(x)-S_{j-1}f(x-r)(\phi_{k}\ast\Delta_{j}g)(x)\big|\\
 & \qquad\le\int_{0}^{1}\int_{\R}|\tilde{\phi}_{k}(x-z)|^{1-1/p}|\tilde{\phi}_{k}(x-z)|^{1/p}\big|S_{j-1}f'\big(x+t(z-x)+(t-1)r\big)\Delta_{j}g(z)\big|\d z\,\dd t\\
 & \qquad\le\|\tilde{\phi}_{k}\|_{L^{1}}^{1-1/p}\int_{0}^{1}\Big(\int_{\R}|\tilde{\phi}_{k}(x-z)|\big|S_{j-1}f'\big(x+t(z-x)+(t-1)r\big)\Delta_{j}g(z)\big|^{p}\d z\Big)^{1/p}\dd t.
\end{align*}
Using that $\alpha-1<0$ and \citep[Proposition~2.79]{Bahouri2011},
we obtain by 
\begin{align}
\begin{split} & \big\|\phi_{k}\ast\big(S_{j-1}f\Delta_{j}g\big)-S_{j-1}f(\phi_{k}\ast\Delta_{j}g)\big\|_{L^{p}}\\
 & \qquad\lesssim\|\tilde{\phi}_{k}\|_{L^{1}}^{1-1/p}\Big(\int_{0}^{1}\int_{\R^{2}}|\tilde{\phi}_{k}(x)|\big|\Delta_{j}g(z)S_{j-1}f'(x+z-tx+(t-1)r)\big|^{p}\d x\,\dd z\,\dd t\Big)^{1/p}\\
 & \qquad\le\|\tilde{\phi}_{k}\|_{L^{1}}^{1-1/p}\Big(\int_{0}^{1}\int_{\R}|\tilde{\phi}_{k}(x)|\|\Delta_{j}g\|_{L^{p_{2}}}^{p}\|S_{j-1}f'\|_{L^{p_{1}}}^{p}\d x\,\dd t\Big)^{1/p}\\
 & \qquad\lesssim2^{-j(\alpha+\beta+\gamma)}\|g\|_{\beta,p_{2},\infty}\|f'\|_{\alpha-1,p_{1},\infty}2^{j(\gamma+1)}\|\tilde{\phi}_{k}(x)\|_{L^{1}}.
\end{split}
\label{eq:estimate2}
\end{align}
Since $\|f'\|_{\alpha-1,p_{1},\infty}\lesssim\|f\|_{\alpha,p_{1},\infty}$,
it suffices to show $\|\tilde{\phi}_{k}\|_{L^{1}}=\|(x-r)\phi_{k}(x)\|_{L^{1}}\lesssim2^{-k(\gamma+1)}.$
Note that
\begin{align*}
(x-r)\phi_{k}(x) & =\int_{\R}(x-z+z-r)\F^{-1}\rho_{k}(x-z)\phi(z)\d z\\
 & =\int_{\R}(x-z)\F^{-1}\rho_{k}(x-z)\phi(z)\d z+\int_{\R}\F^{-1}\rho_{k}(x-z)(z-r)\phi(z)\d z\\
 & =\big((y\F^{-1}\rho_{k}(y))\ast\phi\big)(x)+\big(\F^{-1}\rho_{k}\ast((y-r)\phi(y)\big)(x)\\
 & =-i\big(\F^{-1}[\rho'_{k}]\ast\phi\big)(x)+\Delta_{k}\big((y-r)\phi(y)\big)(x)\\
 & =-i\sum_{j\sim k}\F^{-1}[\rho_{k}']\ast\Delta_{j}\phi(x)+\Delta_{k}\big((y-r)\phi(y)\big)(x),
\end{align*}
which by Young's inequality implies
\begin{align}
\begin{split}\|(x-r)\phi_{k}(x)\|_{L^{1}} & \le\sum_{j\sim k}\|\F^{-1}[\rho_{k}']\|_{L^{1}}\|\Delta_{j}\phi\|_{L^{1}}+\|\Delta_{k}\big((y-r)\phi(y)\big)\|_{L^{1}}\\
 & =\sum_{j\sim k}\|\F^{-1}[\rho'](2^{k}\cdot)\|_{L^{1}}\|\Delta_{j}\phi\|_{L^{1}}+\|\Delta_{k}\big((y-r)\phi(y)\big)\|_{L^{1}}\\
 & =2^{-k(\gamma+1)}\|\F^{-1}[\rho']\|_{L^{1}}\|\phi\|_{\gamma,1,\infty}+2^{-k(\gamma+1)}\|(\cdot-r)\phi\|_{\gamma+1,1,\infty}.
\end{split}
\label{eq:estimate3}
\end{align}
Finally, we combine the estimates~(\ref{eq:estimate1}), (\ref{eq:estimate2})
and (\ref{eq:estimate3}) to get 
\[
\sum_{j\sim k}\|\Delta_{k}R_{j}\|_{L^{p}}\lesssim2^{-k(\alpha+\beta+\gamma)}\big(\|\phi\|_{\gamma,1,\infty}+\|(\cdot-r)\phi\|_{\gamma+1,1,\infty}\big)\|f\|_{\alpha,p_{1},\infty}\|g\|_{\beta,p_{2},\infty}.
\]
In view of (\ref{eq:estimatebesov}) we have proven the asserted bound
for $\big\| R_{\phi}(f,g)\big\|_{\alpha+\beta+\gamma,p,\infty}$ and
in particular $R_{\phi}(f,g)\in\mathcal{B}_{p,\infty}^{\alpha+\beta+\gamma}$.
\end{proof}
\begin{rem}
Lemma~\ref{lem:convolution paracontrolled} can be seen as a counterpart
to the integration by parts formula as used in the context of classical
rough differential equations of the form $Du=F(u)\xi$ with the differential
operator $D$ and a signal $\xi\in\mathcal{B}_{p,\infty}^{\beta-1}$,
see for example \citep{Gubinelli2015,Promel2015}. Defining the integration
operator $I:=D^{-1}$ to be the inverse of $D$ and denoting by $\theta$
the solution of $D\theta=\xi$, one has 
\[
T_{F(u)}\theta=I(DT_{F(u)}\theta)=IT_{F(u)}\xi+IT_{DF(u)}\theta,
\]
where the second term is of regularity $2\alpha$. Heuristically speaking,
for Volterra equations we replace the integration operator $I\colon f\mapsto I(f)$
by the convolution operator $f\mapsto\phi\ast f$ and set $\theta:=\phi\ast\xi$.
\end{rem}

The resonant term $\pi(\phi_{1}\ast\xi_{1},\xi_{1})$ appearing in
(\ref{eq:ansatzCirtTerm}) turns out to be the necessary ``additional
information'' one needs to postulate in order to give a meaning to
the Volterra equation~(\ref{eq:convol}) with rough driving signals
$\xi_{1}$. It corresponds to the iterated integrals in rough path
theory (cf. \citep{Lyons1998,Lyons2007,Friz2014}) or the models in
Hairer's theory of regularity structures (cf. \citep{Hairer2014,Hairer2015}).
For the construction of $\pi(\phi_{1}\ast\xi_{1},\xi_{1})$ for certain
stochastic processes we refer to Section~\ref{sec:resonant term}. 

In the present context we introduce the notion of convolutional rough
paths. 
\begin{defn}
\label{def:convRoughPath}Let $\beta,\gamma>0$, $p\in[2,\infty]$
and set $\alpha:=\beta+\gamma-1$. The space of smooth functions $\xi\colon\R\to\R^{n}$
with compact support is denoted by $\mathcal{C}_{c}^{\infty}$. Given
a function $\phi\in\mathcal{B}_{1,\infty}^{\gamma}$, the closure
of the set 
\[
\big\{\big(\xi,\pi(\phi*\xi,\xi)\big)\,:\,\xi\in\mathcal{C}_{c}^{\infty}\big\}\subset\mathcal{B}_{p,\infty}^{\beta-1}\times\mathcal{B}_{p/2,\infty}^{\alpha+\beta-1}
\]
with respect to the norm $\|\xi\|_{\beta-1,p,\infty}+\|\pi(\phi*\xi,\xi)\|_{\alpha+\beta-1,p/2,\infty}$
is denoted by $\mathcal{\mathcal{B}}_{p}^{\beta,\gamma}(\phi)$ and
$(\xi,\mu)\in\mathcal{B}_{p}^{\beta,\gamma}(\phi)$ is called \textit{convolutional
rough path}. 
\end{defn}

Assuming $\pi(\phi_{1}\ast\xi_{1},\xi_{1})$ is well-defined, by the
previous analysis we know that $u^{*}$ from (\ref{eq:u star}) is
also well-defined. Hence, Bony's decomposition and Lemma~\ref{lem:convolution paracontrolled}
allow to rewrite the rough term $\phi_{1}\ast(\sigma_{1}(u)\xi_{1})$
as 
\begin{align*}
\phi_{1}\ast(\sigma_{1}(u)\xi_{1}) & =\phi_{1}*\big(T_{\sigma_{1}(u)}\xi_{1}+\pi(\sigma_{1}(u),\xi_{1})+T_{\xi_{1}}\sigma_{1}(u)\big)\\
 & =T_{\sigma_{1}(u(\cdot-r_{1}))}(\phi_{1}\ast\xi_{1})+\underbrace{u^{*}+R_{\phi_{1}}(\sigma_{1}(u),\xi_{1})}_{\in\mathcal{B}_{p/2,\infty}^{2\alpha}}.
\end{align*}
For the more regular drift term $\phi_{2}\ast(\sigma_{2}(u)\xi_{2})$
we observe (using similar calculations as in the Young setting and
Lemma~\ref{lem:convolution paracontrolled}) a control structure
with respect to $\phi_{2}\ast\xi_{2}$: 
\begin{align*}
\phi_{2}\ast(\sigma_{2}(u)\xi_{2}) & =\phi_{2}\ast T_{\sigma_{2}(u)}\xi_{2}+\phi_{2}\ast\big(\underbrace{\pi(\sigma_{2}(u),\xi_{2})+T_{\xi_{2}}(\sigma_{2}(u))}_{\in\mathcal{B}_{p/2,\infty}^{\alpha+\beta_{2}-1}}\big)\\
 & =T_{\sigma_{2}(u(\cdot-r_{2}))}(\phi_{2}\ast\xi_{2})+\underbrace{\phi_{2}\ast\big(\pi(\sigma_{2}(u),\xi_{2})+T_{\xi_{2}}(\sigma_{2}(u))\big)+R_{\phi_{2}}(\sigma_{2}(u),\xi_{2})}_{\in\mathcal{B}_{p/2,\infty}^{2\alpha}},
\end{align*}
for some $r_{2}\in\R$. Therefore, the ansatz for a solution $u$
to the Volterra equation~(\ref{eq:convol}) leads to the following
``paracontrolled'' structure:
\begin{defn}
Let $p\ge1$ and $\alpha>1/p$. A function $v\in\mathcal{B}_{p,\infty}^{\alpha}$
is called \emph{paracontrolled }by $w_{1},w_{2}\in\mathcal{B}_{p,\infty}^{\alpha}$
if there are $v^{(1)},v^{(2)}\in\mathcal{B}_{p,\infty}^{\alpha}$
such that $v^{\#}:=v-T_{v^{(1)}}w_{1}-T_{v^{(2)}}w_{2}\in\mathcal{B}_{p/2,\infty}^{2\alpha}$.
The space of all such triples $(v,v^{(1)},v^{(2)})\in(\mathcal{B}_{p,\infty}^{\alpha})^{3}$
where $v$ paracontrolled by $w_{1},w_{2}\in\mathcal{B}_{p,\infty}^{\alpha}$
is denoted by $\mathcal{D}_{p}^{\alpha}(w_{1},w_{2})$ equipped with
the norm
\[
\|v^{(1)}\|_{\alpha,p,\infty}+\|v^{(2)}\|_{\alpha,p,\infty}+\|v-T_{v^{(1)}}w_{1}-T_{v^{(2)}}w_{2}\|_{2\alpha,p/2,\infty}.
\]
\end{defn}

\begin{rem}
Note that for any $v^{(1)},v^{(2)},w_{1},w_{2}\in\mathcal{B}_{p,\infty}^{\alpha}$
and $v^{\#}\in\mathcal{B}_{p,\infty}^{2\alpha}$ the function $v:=T_{v^{(1)}}w_{1}-T_{v^{(2)}}w_{2}+v^{\#}$
is paracontrolled by $w_{1},w_{2}$ and, in particular, $v$ is an
element of $\mathcal{B}_{p,\infty}^{\alpha}$. Indeed, Lemma~\ref{lem:paraproduct}
and the embeddings $\mathcal{B}_{p/2,\infty}^{2\alpha}\subset\mathcal{B}_{p,\infty}^{\alpha}\subset L^{\infty}$
imply 
\begin{align*}
\|v\|_{\alpha,p,\infty} & \lesssim\sum_{j=1,2}\|T_{v^{(j)}}w_{j}\|_{\alpha,p,\infty}+\|v{}^{\#}\|_{\alpha,p,\infty}\lesssim\sum_{j=1,2}\|v^{(j)}\|_{\alpha,p,\infty}\|w_{j}\|_{\alpha,p,\infty}+\|v^{\#}\|_{2\alpha,p/2,\infty}.
\end{align*}
\end{rem}

It is natural to require the same paracontrolled structure for the
initial condition $u_{0}$ as for the solution for $u$. In other
words, $u_{0}$ is assumed to be of the form
\[
u_{0}=T_{u_{0}^{(1)}}(\phi_{1}\ast\xi_{1})+T_{u_{0}^{(2)}}(\phi_{2}\ast\xi_{2})+u_{0}^{\#}\quad\mbox{for some }u_{0}^{(1)},u_{0}^{(2)}\in\mathcal{B}_{p,\infty}^{\alpha},\,u_{0}^{\#}\in\mathcal{B}_{p/2,\infty}^{2\alpha}.
\]

\begin{rem}
A similar requirement for initial conditions $u_{0}$ appears in the
context of delay differential equations driven by rough paths where
$u_{0}$ is usually a path and not only a constant, cf. \citet[Theorem 1.1]{Neuenkirch2008}.
Hence, in order to ensure that the rough path integral is well-defined,
\citet{Neuenkirch2008} suppose the initial condition to be a controlled
path in the sense of \citet{Gubinelli2004}. 
\end{rem}

To sum up, the ansatz reads as
\begin{align}
u & =T_{u_{0}^{(1)}+\sigma_{1}(u(\cdot-r_{1}))}(\phi_{1}\ast\xi_{1})+T_{u_{0}^{(2)}+\sigma_{2}(u(\cdot-r_{2}))}(\phi_{2}\ast\xi_{2})+u^{\#}\label{eq:finalAnsatz}
\end{align}
with 
\begin{equation}
u^{\#}:=u_{0}^{\#}+\sum_{j=1,2}\big(\phi_{j}\ast\big(\pi(\sigma_{j}(u),\xi_{j})+T_{\xi_{j}}\sigma_{j}(u)\big)+R_{\phi_{j}}(\sigma_{j}(u),\xi_{j})\big)\in\mathcal{B}_{p/2,\infty}^{2\alpha}.\label{eq:u raute}
\end{equation}
Note that, imposing the Young type condition $\alpha+\beta_{2}>1$
on the regularity of the drift term $\phi_{2}*(\sigma_{2}(u)\xi_{2})$
ensures especially that the cross terms $\pi(\phi_{1}\ast\xi_{1},\xi_{2})$
and $\pi(\phi_{2}\ast\xi_{2},\xi_{1})$ are well-defined. \\

Postulating the paracontrolled structure for the initial condition,
we show in the following that the \emph{Itô-Lyons map} $\hat{S}:=\hat{S}_{\phi_{1},\phi_{2}}$
given by 
\begin{equation}
\begin{split}\hat{S}\colon(\mathcal{B}_{p,\infty}^{\alpha})^{2}\times\mathcal{B}_{p/2,\infty}^{2\alpha}\times\mathcal{B}_{p}^{\beta_{1},\gamma_{1}}(\phi_{1})\times\mathcal{B}_{p,\infty}^{\beta_{2}} & \to\mathcal{B}_{p,\infty}^{\alpha},\\
\big(u_{0}^{(1)},u_{0}^{(2)},u_{0}^{\#},(\xi_{1},\mu),\xi_{2}\big) & \mapsto u,
\end{split}
\label{eq:itomap}
\end{equation}
where $u$ is the solution to the Volterra equation~(\ref{eq:convol})
given the initial condition $u_{0}:=T_{u_{0}^{(1)}}(\phi_{1}\ast\xi_{1})+T_{u_{0}^{(2)}}(\phi_{2}\ast\xi_{2})+u_{0}^{\#}$
and the inputs $(\xi_{1},\mu)$, $\xi_{2}$, has indeed a unique locally
Lipschitz continuous extension from smooth driving signals $(\xi_{1},\pi(\phi_{1}*\xi_{1},\xi_{1}))$
to the space of convolutional rough paths~$(\xi_{1},\mu)$. For fixed
signals $((\xi_{1},\mu),\xi_{2})$ the ansatz from above and the proof
of the following theorem reveals that the Itô-Lyons maps, more precisely,
$\mathcal{D}_{p}^{\alpha}(\phi_{1}\ast\xi_{1},\phi_{2}\ast\xi_{2})$
into $\mathcal{D}_{p}^{\alpha}(\phi_{1}\ast\xi_{1},\phi_{2}\ast\xi_{2})$.
\begin{thm}
\label{thm:solve} Let $p\in[3,\infty]$, $0<\beta_{1}\le\beta_{2}\le1$
and $0<\gamma_{1}\le\gamma_{2}$ satisfy \textup{$\alpha:=\beta_{1}+\gamma_{1}-1\in(\frac{1}{3},1)$,
$2\alpha+\beta_{1}>1$ and} $\alpha+\beta_{2}>1$. For 

\begin{enumerate}
\item $\sigma_{1}\in C^{3}$ and $\sigma_{2}\in C^{2}$ with $\sigma_{1}(0)=\sigma_{2}(0)=0$,
\item $\phi_{j}\in\mathcal{B}_{1,\infty}^{\gamma_{j}}$ such that there
exists $r_{j}\in\R$ with $\|(\cdot-r_{j})\phi_{j}\|_{\gamma_{j}+1,1,\infty}<\infty,$
for $j=1,2$, 
\item $(\xi_{1},\mu)\in\mathcal{B}_{p}^{\beta_{1}-1,\gamma_{1}}(\phi_{1})$
and $\xi_{2}\in\mathcal{B}_{p,\infty}^{\beta_{2}-1}$,
\item $(u_{0}^{(1)},u_{0}^{(2)},u_{0}^{\#})\in(\mathcal{B}_{p,\infty}^{\alpha})^{2}\times\mathcal{B}_{p/2,\infty}^{2\alpha}$,
\end{enumerate}
the Volterra equation~(\ref{eq:convol}) with initial condition $u_{0}=T_{u_{0}^{(1)}}(\phi_{1}\ast\xi_{1})+T_{u_{0}^{(2)}}(\phi_{2}\ast\xi_{2})+u_{0}^{\#}$
has a unique solution if $\Delta:=\|\sigma_{1}\|_{C^{3}}\|\phi_{1}\|_{\gamma_{1},1,\infty}+\|\sigma_{2}\|_{C^{2}}\|\phi_{2}\|_{\gamma_{2},1,\infty}$
is sufficiently small depending on $((u_{0}^{(1)},u_{0}^{(2)},u_{0}^{\#}),(\xi_{1},\mu),\xi_{2})$
and $\phi_{1},\phi_{2}$. Moreover, the Itô-Lyons map $\hat{S}$ from
(\ref{eq:itomap}) is locally Lipschitz continuous around $((u_{0}^{(1)},u_{0}^{(2)},u_{0}^{\#}),(\xi_{1},\mu),\xi_{2})$.
\end{thm}

\begin{rem}
\label{rem:jumps} Theorem~\ref{thm:solve} provides the local Lipschitz
continuity of the Itô-Lyons map $\hat{S}$ on the rough path space
$\mathcal{B}_{p}^{\beta_{1},\gamma_{1}}(\phi_{1})$, which contains
(convolutional) geometric rough paths with jumps. Indeed, considering
$\gamma_{1}>1$ and $p=3$, the parameter assumptions in Theorem~\ref{thm:solve}
only require 
\[
\beta_{1}>1-\frac{2}{3}\gamma_{1}\quad\mbox{and}\quad\beta_{1}>\frac{1}{3}+1-\gamma_{1}.
\]
Hence, we can choose $\beta_{1}<1/p$, which implies that $\mathcal{B}_{p}^{\beta_{1},\gamma_{1}}(\phi_{1})$
contains discontinuous paths, and Theorem~\ref{thm:solve} is still
applicable.

The existence of a continuous extension of the Itô-Lyons map from
the space of smooth paths to a space of geometric rough paths containing
discontinuous paths seems to be due to the smoothing property of the
kernel function~$\phi_{1}\in\mathcal{B}_{1,\infty}^{\gamma_{1}}$
for $\gamma_{1}>1$, cf. \citep[Remark~5.11]{Promel2015}. Note that
even if the driving rough path may possess jumps, the solution of
the Volterra equation is still a continuous functions as we require
$\alpha>1/3$. In order to obtain a continuous extension of the Itô-Lyons
map acting on smooth paths to discontinuous rough paths in the case
of classical rough differential equations (corresponding to $\gamma_{1}=1$)
requires to consider the rough paths enhanced with an additional information
given by the so-called path functionals, see the work of \citet{Chevyrev2017}.
\end{rem}

\subsection{Proof of Theorem~\ref{thm:solve}}

Most objects appearing the paracontrolled approach to the Volterra
equation~(\ref{eq:convol}) come only with local Lipschitz estimates
of their Besov norms. Therefore, as a first step towards a proof of
Theorem~\ref{thm:solve} we provide the a priori bounds for solutions
of the Volterra equation~(\ref{eq:convol}).
\begin{prop}
\label{prop:bound ito map} Let $p\in[3,\infty]$, $0<\beta_{1}\le\beta_{2}\le1$
and $0<\gamma_{1}\le\gamma_{2}$ satisfy \textup{$\alpha:=\beta_{1}+\gamma_{1}-1\in(\frac{1}{3},1)$,
$2\alpha+\beta_{1}>1$ and} $\alpha+\beta_{2}>1$. Suppose that 

\begin{enumerate}
\item $\sigma_{1}\in C^{2}$ and $\sigma_{2}\in C^{1}$ with $\sigma_{1}(0)=\sigma_{2}(0)=0$,
\item $\phi_{j}\in\mathcal{B}_{1,\infty}^{\gamma_{j}}$ such that there
exists $r_{j}\in\R$ with $\|(\cdot-r_{j})\phi_{j}\|_{\gamma_{j}+1,1,\infty}<\infty,$
for $j=1,2$, 
\item $\xi_{1}\in\mathcal{C}_{c}^{\infty}$ and $\xi_{2}\in\mathcal{B}_{p,\infty}^{\beta_{2}-1}$,
\item $(u_{0}^{(1)},u_{0}^{(2)},u_{0}^{\#})\in(\mathcal{B}_{p,\infty}^{\alpha})^{2}\times\mathcal{B}_{p/2,\infty}^{2\alpha}$.
\end{enumerate}
Let $u_{0}:=u_{0}^{\#}+T_{u_{0}^{(1)}}(\phi_{1}\ast\xi_{1})+T_{u_{0}^{(2)}}(\phi_{2}\ast\xi_{2})$
be the paracontrolled initial condition. Setting
\begin{align*}
 & \Delta:=\|\sigma_{1}\|_{C^{2}}\|\phi_{1}\|_{\gamma_{1},1,\infty}+\|\sigma_{2}\|_{C^{1}}\|\phi_{2}\|_{\gamma_{2},1,\infty}, &  & C_{\sigma}:=\|\sigma_{1}\|_{C^{1}}+\|\sigma_{2}\|_{C^{1}}+1,\\
 & C_{\phi}:=\sum_{j=1,2}\big(\|\phi_{j}\|_{\gamma_{j},1,\infty}+\|(\cdot-r_{j})\phi_{j}\|_{\gamma_{j}+1,1,\infty}\big)+1, &  & C_{u_{0}}:=\|u_{0}^{(1)}\|_{\alpha,p,\infty}+\|u_{0}^{(2)}\|_{\alpha,p,\infty}+1
\end{align*}
 and 
\[
C_{\xi}:=\|\pi(\phi_{1}\ast\xi_{1},\xi_{1})\|_{\alpha+\beta_{1}-1,p/2,\infty}+\sum_{j=1,2}\bigg(1+\sum_{k=1,2}\|\phi_{k}\|_{\gamma_{k},1,\infty}\|\xi_{k}\|_{\beta_{k}-1,p,\infty}\bigg)\|\xi_{j}\|_{\beta_{j}-1,p,\infty},
\]
there is a constant $c>0$ depending only on $\alpha$ and $p$ such
that, if $\Delta C_{\sigma}C_{\phi}C_{\xi}C_{u_{0}}\le c$, then 
\begin{align*}
\|u\|_{\alpha,p,\infty}\le & 2\|u_{0}\|_{\alpha,p,\infty}+\|u_{0}^{\#}\|_{2\alpha,p/2,\infty}+1\\
\lesssim & \sum_{j=1,2}\|\phi_{j}\|_{\gamma_{j},1,\infty}\|\xi_{j}\|_{\beta_{j}-1,p,\infty}\|u_{0}^{(j)}\|_{\alpha,p,\infty}+\|u_{0}^{\#}\|_{2\alpha,p/2,\infty}+1.
\end{align*}
\end{prop}

\begin{proof}
Using 
\[
u=u_{0}+\sum_{j=1,2}\phi_{j}\ast\big(T_{\sigma_{j}(u)}\xi_{j}+\pi(\sigma_{j}(u),\xi_{j})+T_{\xi_{j}}\sigma_{j}(u)\big),
\]
$\alpha\le\beta_{j}+\gamma_{j}-1$, the generalized Young inequality
(Lemma~\ref{lem:Convolution}) and Besov embeddings (as $\alpha>1/p$),
we have
\begin{align*}
\|u\|_{\alpha,p,\infty} & \lesssim\|u_{0}\|_{\alpha,p,\infty}+\sum_{j=1,2}\|\phi_{j}\|_{\gamma_{j},1,\infty}\big(\|T_{\sigma_{j}(u)}\xi_{j}\|_{\beta_{j}-1,p,\infty}\\
 & \qquad\hspace{6em}\quad+\|\pi(\sigma_{j}(u),\xi_{j})\|_{\alpha+\beta_{j}-1,p/2,\infty}+\|T_{\xi_{j}}\sigma_{j}(u)\|_{\beta_{j}-1,p,\infty}\big).
\end{align*}
By the paraproduct estimates (Lemma~\ref{lem:paraproduct}) and Lemma~\ref{lem:differenceBesovNorm}
we obtain for $j=1,2$
\begin{align*}
\|T_{\sigma_{j}(u)}\xi_{j}\|_{\beta_{j}-1,p,\infty} & \lesssim\|\sigma_{j}(u)\|_{\infty}\|\xi_{j}\|_{\beta_{j}-1,p,\infty}\lesssim\|\sigma_{j}\|_{\infty}\|\xi_{j}\|_{\beta_{j}-1,p,\infty}
\end{align*}
and
\begin{align}
\|T_{\xi_{j}}\sigma_{j}(u)\|_{\beta_{j}-1,p,\infty} & \lesssim\|\xi_{j}\|_{\beta_{j}-1,p,\infty}\|\sigma_{j}(u)\|_{\alpha,p,\infty}\lesssim\|\xi_{j}\|_{\beta_{j}-1,p,\infty}\|\sigma_{j}\|_{C^{1}}\|u\|_{\alpha,p,\infty}.\label{eq:estimate T_xi}
\end{align}

We now need to bound the resonant terms $\|\pi(\sigma_{j}(u),\xi_{j})\|_{\alpha+\beta_{j}-1,p/2,\infty}$
for $j=1,2$. For $j=2$ we apply again the paraproduct estimates
(Lemma~\ref{lem:paraproduct}) and Lemma~\ref{lem:differenceBesovNorm}
to get
\[
\|\pi(\sigma_{2}(u),\xi_{2})\|_{\alpha+\beta_{2}-1,p/2,\infty}\lesssim\|\sigma_{2}(u)\|_{\alpha,p,\infty}\|\xi_{2}\|_{\beta_{2}-1,p,\infty}\lesssim\|\sigma_{2}\|_{C^{1}}\|u\|_{\alpha,p,\infty}\|\xi_{2}\|_{\beta_{2}-1,p,\infty}
\]
using the assumption $\alpha+\beta_{2}-1>0$.

For $j=1$, in order to avoid a quadratic bound of $\Pi_{\sigma}(u,\xi)$
(cf.~(\ref{eq:Pi}) and \citep[Proposition~4.1]{Promel2015}), we
apply the linearization from Lemma~\ref{lem:linearization}, which
provides a function $S_{\sigma_{1}}(u)\in\mathcal{B}_{p/2,\infty}^{2\alpha}$
such that
\[
\pi(\sigma_{1}(u),\xi_{1})=\pi(T_{\sigma_{1}'(u)}u,\xi_{1})+\pi(S_{\sigma_{1}}(u),\xi_{1}).
\]
Writing the ansatz~(\ref{eq:finalAnsatz}) as 
\begin{equation}
u=\sum_{k=1,2}T_{\tilde{u}_{k}}(\phi_{k}\ast\xi_{k})+u^{\#}\quad\mbox{with}\quad\tilde{u}_{k}:=u_{0}^{(k)}+\sigma_{k}(u(\cdot-r_{k})),\quad k=1,2,\label{eq:ansatz rewritten}
\end{equation}
and in combination with the commutator estimate~(\ref{eq:commutator}),
we find that
\begin{align*}
\pi(\sigma_{1}(u),\xi_{1}) & =\sum_{k=1,2}\pi\big(T_{\sigma_{1}'(u)}\big(T_{\tilde{u}_{k}}(\phi_{k}\ast\xi_{k})\big),\xi_{1}\big)+\pi(T_{\sigma_{1}'(u)}u^{\#},\xi_{1})+\pi(S_{\sigma_{1}}(u),\xi_{j})\\
 & =\sum_{k=1,2}\Big(\sigma_{1}'(u)\pi(T_{\tilde{u}_{k}}(\phi_{k}\ast\xi_{k}),\xi_{1})+\Gamma(\sigma_{1}'(u),T_{\tilde{u}_{k}}(\phi_{k}\ast\xi_{k}),\xi_{1})\Big)\\
 & \qquad\qquad+\pi(T_{\sigma_{1}'(u)}u^{\#},\xi_{1})+\pi(S_{\sigma_{1}}(u),\xi_{1})\\
 & =\sum_{k=1,2}\Big(\sigma_{1}'(u)\tilde{u}_{k}\pi(\phi_{k}\ast\xi_{k},\xi_{1})+\sigma_{1}'(u)\Gamma(\tilde{u}_{k},\phi_{k}\ast\xi_{k},\xi_{1})\\
 & \qquad\qquad+\Gamma(\sigma_{1}'(u),T_{\tilde{u}_{k}}(\phi_{k}\ast\xi_{k}),\xi_{1})\Big)+\pi(T_{\sigma_{1}'(u)}u^{\#},\xi_{1})+\pi(S_{\sigma_{1}}(u),\xi_{1}).
\end{align*}
In the following we estimate these five terms, with $k=1,2$, frequently
using Besov embeddings ($\alpha>1/p$), the paraproduct estimates
(Lemma~\ref{lem:paraproduct}) and the auxiliary Besov estimates
(Lemma~\ref{lem:shiftBesovNorm}, \ref{lem:product} and \ref{lem:differenceBesovNorm}).

Defining $F(x,y):=\sigma_{1}'(x)\sigma_{k}(y)$ and owing to $2\alpha+\beta_{1}>1$,
we have 
\begin{align*}
 & \|\sigma_{1}'(u)\tilde{u}_{k}\pi(\phi_{k}\ast\xi_{k},\xi_{1})\|_{\alpha+\beta_{1}-1,p/2,\infty}\\
 & \quad\lesssim\|\sigma_{1}^{\prime}(u)\tilde{u}_{k}\|_{\alpha,p,\infty}\|\pi(\phi_{k}\ast\xi_{k},\xi_{1})\|_{\alpha+\beta_{1}-1,p/2,\infty}\\
 & \quad\lesssim\big(\|\sigma_{1}'(u)\|_{\alpha,p,\infty}\|u_{0}^{(k)}\|_{\alpha,p,\infty}+\|\sigma_{1}'(u)\sigma_{k}(u(\cdot-r_{k}))\|_{\alpha,p,\infty}\big)\|\pi(\phi_{k}\ast\xi_{k},\xi_{1})\|_{\alpha+\beta_{1}-1,p/2,\infty}\\
 & \quad\lesssim\big(\|\sigma_{1}'\|_{C^{1}}\|u\|_{\alpha,p,\infty}\|u_{0}^{(k)}\|_{\alpha,p,\infty}+\|F(u,u(\cdot-r_{k}))\|_{\alpha,p,\infty}\big)\|\pi(\phi_{k}\ast\xi_{k},\xi_{1})\|_{\alpha+\beta_{1}-1,p/2,\infty}\\
 & \quad\lesssim\big(\|\sigma_{1}'\|_{C^{1}}\|u\|_{\alpha,p,\infty}\|u_{0}^{(k)}\|_{\alpha,p,\infty}+\|\sigma_{k}\|_{C^{1}}\|\sigma_{1}\|_{C^{2}}(\|u\|_{\alpha,p,\infty}+\|u(\cdot-r_{k}))\|_{\alpha,p,\infty})\big)\\
 & \qquad\times\|\pi(\phi_{k}\ast\xi_{k},\xi_{1})\|_{\alpha+\beta_{1}-1,p/2,\infty}\\
 & \quad\lesssim\|\sigma_{1}\|_{C^{2}}\big(\|u_{0}^{(k)}\|_{\alpha,p,\infty}+\|\sigma_{k}\|_{C^{1}}\big)\|\pi(\phi_{k}\ast\xi_{k},\xi_{1})\|_{\alpha+\beta_{1}-1,p/2,\infty}\|u\|_{\alpha,p,\infty}.
\end{align*}
Applying the commutator estimate~(\ref{eq:commutator}) and Young's
inequality (Lemma~\ref{lem:Convolution}), we obtain 
\begin{align*}
 & \|\sigma_{1}'(u)\Gamma(\tilde{u}_{k},\phi_{k}\ast\xi_{k},\xi_{1})\|_{\alpha+\beta_{1}-1,p/2,\infty}\\
 & \quad\lesssim\|\sigma_{1}'(u)\|_{\infty}\|\Gamma(\tilde{u}_{k},\phi_{k}\ast\xi_{k},\xi_{1})\|_{2\alpha+\beta_{1}-1,p/3,\infty}\\
 & \quad\lesssim\|\sigma_{1}^{\prime}\|_{\infty}\|u_{0}^{(k)}+\sigma_{k}(u(\cdot-r_{k}))\|_{\alpha,p,\infty}\|\phi_{k}\ast\xi_{k}\|_{\alpha,p,\infty}\|\xi_{1}\|_{\beta_{1}-1,p,\infty}\\
 & \quad\lesssim\|\sigma_{1}\|_{C^{1}}\big(\|u_{0}^{(k)}\|_{\alpha,p,\infty}+\|\sigma_{k}\|_{C^{1}}\|u\|_{\alpha,p,\infty}\big)\|\phi_{k}\|_{\gamma_{k},1,\infty}\|\xi_{k}\|_{\beta_{k}-1,p,\infty}\|\xi_{1}\|_{\beta_{1}-1,p,\infty}
\end{align*}
and similarly
\begin{align*}
 & \|\Gamma(\sigma_{1}'(u),T_{\tilde{u}_{k}}(\phi_{k}\ast\xi_{k}),\xi_{1})\|_{\alpha+\beta_{1}-1,p/2,\infty}\\
 & \quad\lesssim\|\sigma_{1}'(u)\|_{\alpha,p,\infty}\|T_{\tilde{u}_{k}}(\phi_{k}\ast\xi_{k})\|_{\alpha,p,\infty}\|\xi_{1}\|_{\beta_{1}-1,p,\infty}\\
 & \quad\lesssim\|\sigma_{1}\|_{C^{2}}\|u\|_{\alpha,p,\infty}\big(\|u_{0}^{(k)}\|_{\infty}+\|\sigma_{k}(u(\cdot-r_{k}))\|_{\infty}\big)\|\phi_{k}\ast\xi_{k}\|_{\alpha,p,\infty}\|\xi_{1}\|_{\beta_{1}-1,p,\infty}\\
 & \quad\lesssim\|\sigma_{1}\|_{C^{2}}\big(\|u_{0}^{(k)}\|_{\alpha,p,\infty}+\|\sigma_{k}\|_{\infty}\big)\|\phi_{k}\|_{\gamma_{k},1,\infty}\|u\|_{\alpha,p,\infty}\|\xi_{k}\|_{\beta_{k}-1,p,\infty}\|\xi_{1}\|_{\beta_{1}-1,p,\infty}.
\end{align*}
From Bony's estimates (Lemma~\ref{lem:paraproduct}) we deduce that
\begin{align*}
 & \|\pi(T_{\sigma_{1}'(u)}u^{\#},\xi_{1})\|_{\alpha+\beta_{1}-1,p/2,\infty}\\
 & \quad\lesssim\|T_{\sigma_{1}'(u)}u^{\#}\|_{2\alpha,p/2,\infty}\|\xi_{1}\|_{\beta_{1}-1,p,\infty}\lesssim\|\sigma_{1}'\|_{\infty}\|u^{\#}\|_{2\alpha,p/2,\infty}\|\xi_{1}\|_{\beta_{1}-1,p,\infty}.
\end{align*}
Finally, Lemma~\ref{lem:linearization} shows
\begin{align*}
 & \|\pi(S_{\sigma_{1}}(u),\xi_{1})\|_{\alpha+\beta_{1}-1,p/2,\infty}\\
 & \quad\lesssim\|S_{\sigma_{1}}(u)\|_{2\alpha,p/2,\infty}\|\xi_{1}\|_{\beta_{1}-1,p,\infty}\\
 & \quad\lesssim\|\sigma_{1}\|_{C^{2}}\|\xi_{1}\|_{\beta_{1}-1,p,\infty}\bigg(1+\sum_{k=1,2}\|\tilde{u}_{k}\|_{\infty}\|\phi_{k}\ast\xi_{k}\|_{\alpha,p,\infty}\bigg)\big(\|u\|_{\alpha,p,\infty}+\|u^{\#}\|_{2\alpha,p/2,\infty}\big)\\
 & \quad\lesssim\|\sigma_{1}\|_{C^{2}}\|\xi_{1}\|_{\beta_{1}-1,p,\infty}\bigg(1+\sum_{k=1,2}\big(\|u_{0}^{(k)}\|_{\infty}+\|\sigma_{k}(u(\cdot-r_{k}))\|_{\infty}\big)\|\phi_{k}\|_{\gamma_{k},1,\infty}\|\xi_{k}\|_{\beta_{k}-1,p,\infty}\bigg)\\
 & \qquad\times\big(\|u\|_{\alpha,p,\infty}+\|u^{\#}\|_{2\alpha,p/2,\infty}\big)\\
 & \quad\lesssim\|\sigma_{1}\|_{C^{2}}\|\xi_{1}\|_{\beta_{1}-1,p,\infty}\bigg(1+\sum_{k=1,2}\big(\|u_{0}^{(k)}\|_{\infty}+\|\sigma_{k}\|_{\infty}\big)\|\phi_{k}\|_{\gamma_{k},1,\infty}\|\xi_{k}\|_{\beta_{k}-1,p,\infty}\bigg)\\
 & \qquad\times\big(\|u\|_{\alpha,p,\infty}+\|u^{\#}\|_{2\alpha,p/2,\infty}\big).
\end{align*}
Summarizing, we have
\begin{align*}
 & \|\pi(\sigma_{1}(u),\xi_{1})\|_{\alpha+\beta_{1}-1,p/2,\infty}\\
 & \quad\lesssim\|\sigma_{1}\|_{C^{2}}(\|\sigma_{1}\|_{C^{1}}+\|\sigma_{2}\|_{C^{1}}+1)\\
 & \qquad\times\Big(\|\xi_{1}\|_{\beta_{1}-1,p,\infty}+\sum_{k=1,2}\big(\|\phi_{k}\|_{\gamma_{k},1,\infty}\|\xi_{k}\|_{\beta_{k}-1,p,\infty}\|\xi_{1}\|_{\beta_{1}-1,p,\infty}+\|\pi(\phi_{k}\ast\xi_{k},\xi_{1})\|_{\alpha+\beta_{1}-1,p/2,\infty}\big)\Big)\\
 & \qquad\times\big(\|u_{0}^{(1)}\|_{\alpha,p,\infty}+\|u_{0}^{(2)}\|_{\alpha,p,\infty}+1\big)\big(\|u\|_{\alpha,p,\infty}+\|u^{\#}\|_{2\alpha,p/2,\infty}+1\big).
\end{align*}
Since $\gamma_{2}+\beta_{2}+\beta_{1}-2\ge\alpha+\beta_{2}-1>0$,
we can estimate the resonant term for $k=2$ by 
\begin{align*}
\|\pi(\phi_{2}\ast\xi_{2},\xi_{1})\|_{\alpha+\beta_{1}-1,p/2,\infty} & \lesssim\|\pi(\phi_{2}\ast\xi_{2},\xi_{1})\|_{\beta_{2}+\beta_{1}+\gamma_{2}-2,p/2,\infty}\\
 & \lesssim\|\phi_{2}\ast\xi_{2}\|_{\beta_{2}+\gamma_{2}-1,p,\infty}\|\xi_{1}\|_{\beta_{1}-1,p,\infty}\\
 & \lesssim\|\phi_{2}\|_{\gamma_{2},1,\infty}\|\xi_{2}\|_{\beta_{2}-1,p,\infty}\|\xi_{1}\|_{\beta_{1}-1,p,\infty},
\end{align*}
where we used Bony's estimates (Lemma~\ref{lem:paraproduct}) and
Young's inequality (Lemma~\ref{lem:Convolution}).

With the definitions from Proposition~\ref{prop:bound ito map} we
thus obtain
\begin{align}
\begin{split}\| & \phi_{1}\ast\pi(\sigma_{1}(u),\xi_{1})\|_{2\alpha,p/2,\infty}+\|\phi_{2}\ast\pi(\sigma_{2}(u),\xi_{2})\|_{2\alpha,p/2,\infty}\\
 & \quad\lesssim\Delta C_{\sigma}C_{\xi}C_{u_{0}}\big(\|u\|_{\alpha,p,\infty}+\|u^{\#}\|_{2\alpha,p/2,\infty}+1\big)
\end{split}
\label{eq:estimate pi}
\end{align}
and
\begin{align}
\|u\|_{\alpha,p,\infty} & \lesssim\|u_{0}\|_{\alpha,p,\infty}+\Delta C_{\sigma}C_{\xi}C_{u_{0}}\big(\|u\|_{\alpha,p,\infty}+\|u^{\#}\|_{2\alpha,p/2,\infty}+1\big).\label{eq:estnormU}
\end{align}
Moreover, by the formula for $u^{\#}$ as given in (\ref{eq:u raute}),
by the estimates~(\ref{eq:estimate T_xi}), (\ref{eq:estimate pi})
and Lemma~\ref{lem:convolution paracontrolled} we see
\begin{align*}
\|u^{\#}\|_{2\alpha,p,\infty} & \le\|u_{0}^{\#}\|_{2\alpha,p/2,\infty}+\sum_{j=1,2}\big(\|\phi_{j}\ast\pi(\sigma_{j}(u),\xi_{j})\|_{2\alpha,p/2,\infty}\\
 & \quad+\|\phi_{j}\ast(T_{\xi_{j}}\sigma_{j}(u))\|_{2\alpha,p/2,\infty}+\|R_{\phi_{j}}(\sigma_{j}(u),\xi_{j})\|_{2\alpha,p/2,\infty}\big)\\
 & \le\|u_{0}^{\#}\|_{2\alpha,p/2,\infty}+C\Delta C_{\sigma}C_{\phi}C_{\xi}C_{u_{0}}\big(\|u\|_{\alpha,p,\infty}+\|u^{\#}\|_{2\alpha,p/2,\infty}+1\big)
\end{align*}
for a constant $C>0$. Assuming $C\Delta C_{\sigma}C_{\phi}C_{\xi}C_{u_{0}}\le1/2$,
one gets 
\[
\|u^{\#}\|_{2\alpha,p/2,\infty}\le\|u\|_{\alpha,p,\infty}+2\|u_{0}^{\#}\|_{2\alpha,p/2,\infty}+1.
\]
Combining this with (\ref{eq:estnormU}), we have for another constant
$C'>0$
\[
\|u\|_{\alpha,p,\infty}\le\|u_{0}\|_{\alpha,p,\infty}+C'\Delta C_{\sigma}C_{\xi}C_{u_{0}}\big(\|u\|_{\alpha,p,\infty}+\|u_{0}^{\#}\|_{2\alpha,p/2,\infty}+1\big).
\]
Therefore, $\|u\|_{\alpha,p,\infty}\le2\|u_{0}\|_{\alpha,p,\infty}+\|u_{0}^{\#}\|_{2\alpha,p/2,\infty}+1$
provided $C'\Delta C_{\sigma}C_{\xi}C_{u_{0}}\le1/2$.
\end{proof}
Finally, we can establish the existence of a unique local Lipschitz
continuous extension of the Itô-Lyons map $\hat{S}$ from~(\ref{eq:itomap})
and thus conclude the existence of unique solution of the Volterra
equation~(\ref{eq:convol}) for the rough setting by approximating
the \textit{\emph{convolutional rough paths}} with smooth functions.
As the estimates work analogously to the proof of Proposition~\ref{prop:bound ito map},
we present only the key estimates without giving to many details.
\begin{proof}[Proof of Theorem~\ref{thm:solve}]
 For $i=1,2$ let $(\xi_{1}^{i},\xi_{2}^{i})\in\mathcal{C}_{c}^{\infty}\times\mathcal{B}_{p,\infty}^{\beta_{2}-1}$
be two signals and $(u_{0}^{(1),i},u_{0}^{(2),i},u^{\#,i})\in(\mathcal{B}_{p,\infty}^{\alpha})^{2}\times\mathcal{B}_{p/2,\infty}^{2\alpha}$
be two initial conditions. Let $M\geq1$ be a constant such that 
\[
C_{\phi},C_{\xi^{i}},C_{u_{0}^{i}},\|u_{0}^{\#,i}\|_{2\alpha,p/2,\infty}\leq M,\quad\text{for}\quad i=1,2,
\]
using the definitions from Proposition~\ref{prop:bound ito map}.
Assuming that 
\[
L_{\sigma}:=\big(\|\sigma_{1}\|_{C^{3}}+\|\sigma_{2}\|_{C^{2}}\big)\big(1+\|\sigma_{1}\|_{C^{3}}+\|\sigma_{2}\|_{C^{2}}\big)
\]
is sufficiently small depending on $M$, Proposition~\ref{prop:YoungLip}
implies the existence of corresponding unique solutions $u^{1},u^{2}$
to the Volterra equation~(\ref{eq:convol}) and additionally Proposition~\ref{prop:bound ito map}
leads to the bound
\[
\|u^{i}\|_{\alpha,p,\infty}\lesssim M^{2},\quad i=1,2.
\]
Based on the ansatz for $u^{1},u^{2}$ (see (\ref{eq:ansatz rewritten}))
and Young's inequality (Lemma~\ref{lem:Convolution}), we observe
\begin{align}
\begin{split} & \|u^{1}-u^{2}\|_{\alpha,p,\infty}\quad\\
 & \lesssim\|u_{0}^{1}-u_{0}^{2}\|_{\alpha,p,\infty}+\sum_{j=1,2}\|\phi_{j}\|_{\gamma_{j},1,\infty}\big(\|T_{\sigma_{j}(u^{1})}\xi_{j}^{1}-T_{\sigma_{j}(u^{2})}\xi_{j}^{2}\|_{\beta_{j}-1,p,\infty}\\
 & \qquad+\|\pi(\sigma_{j}(u^{1}),\xi_{j}^{1})-\pi(\sigma_{j}(u^{2}),\xi_{j}^{2})\|_{\alpha+\beta_{j}-1,p/2,\infty}+\|T_{\xi_{j}^{1}}\sigma_{j}(u^{1})-T_{\xi_{j}^{2}}\sigma_{j}(u^{2})\|_{\beta_{j}-1,p,\infty}\big).
\end{split}
\label{eq:difference lipschitz}
\end{align}
By the paraproduct estimates (Lemma~\ref{lem:paraproduct}) and Lemma~\ref{lem:differenceBesovNorm}
we obtain
\begin{align}
\begin{split}\|T_{\sigma_{j}(u^{1})}\xi_{j}^{1}-T_{\sigma_{j}(u^{2})}\xi_{j}^{2}\|_{\beta_{j}-1,p,\infty} & \lesssim\|\sigma_{j}\|_{\infty}\|\xi_{j}^{1}-\xi_{j}^{2}\|_{\beta_{j}-1,p,\infty}+\|\sigma_{j}\|_{C^{2}}\|\xi^{2}\|_{\beta_{j}-1,p,\infty}\\
 & \quad\times\big(1+\|u^{1}\|_{\alpha,p,\infty}+\|u^{2}\|_{\alpha,p,\infty}\big)\|u^{1}-u^{2}\|_{\alpha,p,\infty}
\end{split}
\label{eq:lipschitz first term}
\end{align}
and
\begin{align}
\begin{split} & \|T_{\xi_{j}^{1}}\sigma_{j}(u^{1})-T_{\xi_{j}^{2}}\sigma(u^{2})\|_{\beta_{j}-1,p,\infty}\\
 & \quad\lesssim\|\sigma_{j}\|_{C^{2}}\|\xi_{j}^{1}\|_{\beta_{j}-1,p,\infty}\big(1+\|u^{1}\|_{\alpha,p,\infty}+\|u^{2}\|_{\alpha,p,\infty}\big)\|u^{1}-u^{2}\|_{\alpha,p,\infty}\\
 & \qquad+\|\sigma_{j}\|_{C^{1}}\|u^{2}\|_{\alpha,p,\infty}\|\xi_{j}^{1}-\xi_{j}^{2}\|_{\beta_{j}-1,p,\infty}.
\end{split}
\label{eq:lipschitz last term}
\end{align}

It remains to show the local Lipschitz continuity of $\pi(\sigma_{j}(u^{i}),\xi_{j}^{i})$.
For $j=2$, due to $\alpha+\beta_{2}-1>0$, we directly apply the
paraproduct estimates (Lemma~\ref{lem:paraproduct}) and Lemma~\ref{lem:differenceBesovNorm}
to get
\begin{align*}
 & \|\pi(\sigma_{2}(u^{1}),\xi_{2}^{1})-\pi(\sigma_{2}(u^{2}),\xi_{2}^{2})\|_{\alpha+\beta_{2}-1,p/2,\infty}\\
 & \quad\lesssim\|\sigma_{2}\|_{C^{2}}\|\xi_{2}^{1}\|_{\beta_{2}-1,p,\infty}\big(1+\|u^{1}\|_{\alpha,p,\infty}+\|u^{2}\|_{\alpha,p,\infty}\big)\|u^{1}-u^{2}\|_{\alpha,p,\infty}\\
 & \qquad+\|\sigma_{2}\|_{C^{1}}\|u^{2}\|_{\alpha,p,\infty}\|\xi_{2}^{1}-\xi_{2}^{2}\|_{\beta_{2}-1,p,\infty}.
\end{align*}
For $j=1$ we linearise $\pi(\sigma_{1}(u^{i}),\xi_{1}^{i})$ more
carefully using Lemma~\ref{lem:linearization}, the ansatz and the
commutator estimate~(\ref{eq:commutator}). Rewriting the ansatz~(\ref{eq:finalAnsatz})
as 
\begin{align*}
u^{i} & =\sum_{k=1,2}T_{\tilde{u}_{k}^{i}}(\phi_{k}\ast\xi_{k}^{i})+u^{\#,i}
\end{align*}
with 
\begin{align*}
 & u^{\#,i}:=u_{0}^{\#,i}+\sum_{j=1,2}\big(\phi_{j}\ast\big(\pi(\sigma_{j}(u^{i}),\xi_{j}^{i})+T_{\xi_{j}^{i}}\sigma_{j}(u^{i})\big)+R_{\phi_{j}}(\sigma_{j}(u^{i}),\xi_{j}^{i})\big),\\
 & \tilde{u}_{k}^{i}:=u_{0}^{(k),i}+\sigma_{k}(u^{i}(\cdot-r_{k})),\quad k=1,2,
\end{align*}
we find as in the proof of Proposition~\ref{prop:bound ito map}
that
\begin{align*}
\pi(\sigma_{1}(u^{i}),\xi_{1}^{i}) & =\sum_{k=1,2}\bigg(\sigma_{1}'(u^{i})\tilde{u}_{k}^{i}\pi(\phi_{k}\ast\xi_{k}^{i},\xi_{1}^{i})+\sigma_{1}'(u^{i})\Gamma(\tilde{u}_{k}^{i},\phi_{k}\ast\xi_{k}^{i},\xi_{1}^{i})\\
 & \qquad\qquad+\Gamma(\sigma_{1}'(u^{i}),T_{\tilde{u}_{k}^{i}}(\phi_{k}\ast\xi_{k}^{i}),\xi_{1}^{i})\bigg)+\pi(T_{\sigma_{1}'(u^{i})}u^{\#,i},\xi_{1}^{i})+\pi(S_{\sigma_{1}}(u^{i}),\xi_{1}^{i})\\
 & =:\sum_{k=1,2}\big(D_{1}^{k,i}+D_{2}^{k,i}+D_{3}^{k,i}\big)+D_{4}^{i}+D_{5}^{i}.
\end{align*}

We estimate the differences of these five terms, with $k=1,2$, using
again Besov embeddings ($\alpha>1/p$), the paraproduct estimates
(Lemma~\ref{lem:paraproduct}) and the auxiliary Besov estimates
(Lemma~\ref{lem:shiftBesovNorm}, \ref{lem:product} and \ref{lem:differenceBesovNorm}).
In order to abbreviate theses estimates, let us introduce 

\begin{align*}
\tilde{C}_{u} & :=\bigg(1+\sum_{i,k=1}^{2}\big(\|u^{i}\|_{\alpha,p,\infty}+\|u^{\#,i}\|_{2\alpha,p/2,\infty}+\|u_{0}^{(k),i}\|_{\alpha,p,\infty}\big)\bigg)^{2},\\
\tilde{C}_{\xi} & :=1+\|\pi(\phi_{1}*\xi_{1}^{1},\xi_{1}^{1})\|_{\alpha+\beta_{1}-1,p/2,\infty}+\sum_{i,k=1}^{2}\|\xi_{k}^{i}\|_{\beta_{k}-1,p,\infty},\\
\tilde{C}_{\phi} & :=\|\phi_{1}\|_{\gamma_{1},1,\infty}+\|\phi_{2}\|_{\gamma_{2},1,\infty}.
\end{align*}
For the first term we have
\begin{align*}
 & \|D_{1}^{k,1}-D_{1}^{k,2}\|_{\alpha+\beta_{1}-1,p/2,\infty}\\
 & \quad\quad\lesssim L_{\sigma}\tilde{C}_{\xi}\tilde{C}_{u}\big(\|u_{0}^{(k),1}-u_{0}^{(k),2}\|_{\alpha,p,\infty}+\|u^{1}-u^{2}\|_{\alpha,p,\infty}\\
 & \qquad\quad\qquad\qquad\qquad+\|\pi(\phi_{k}\ast\xi_{k}^{1},\xi_{1}^{1})-\pi(\phi_{k}\ast\xi_{k}^{2},\xi_{1}^{2})\|_{\alpha+\beta_{1}-1,p/2,\infty}\big).
\end{align*}
Applying the commutator estimate~(\ref{eq:commutator}) and Young's
inequality (Lemma~\ref{lem:Convolution}), we obtain
\begin{align*}
 & \|D_{2}^{k,1}-D_{2}^{k,2}\|_{\alpha+\beta_{1}-1,p/2,\infty}\\
 & \quad\quad\lesssim L_{\sigma}\tilde{C}_{\phi}\tilde{C}_{\xi}^{2}\tilde{C}_{u}\big(\|u_{0}^{(k),1}-u_{0}^{(k),2}\|_{\alpha,p,\infty}+\|u^{1}-u^{2}\|_{\alpha,p,\infty}\\
 & \qquad\quad\qquad\qquad\qquad\quad+\|\xi_{k}^{1}-\xi_{k}^{2}\|_{\beta_{k}-1,p,\infty}+\|\xi_{1}^{1}-\xi_{1}^{2}\|_{\beta_{1}-1,p,\infty}\big).
\end{align*}
The commutator estimate and Young's inequality moreover yield
\begin{align*}
 & \|D_{3}^{k,1}-D_{3}^{k,2}\|_{\alpha+\beta_{1}-1,p/2,\infty}\\
 & \quad\quad\lesssim L_{\sigma}\tilde{C}_{\phi}\tilde{C}_{\xi}^{2}\tilde{C}_{u}\big(\|u^{1}-u^{2}\|_{\alpha,p,\infty}+\|u_{0}^{(k),1}-u_{0}^{(k),2}\|_{\alpha,p,\infty}\\
 & \qquad\quad\qquad\qquad\qquad\quad+\|\xi_{k}^{1}-\xi_{k}^{2}\|_{\beta_{k}-1,p,\infty}+\|\xi_{1}^{1}-\xi_{1}^{2}\|_{\beta_{1}-1,p,\infty}\big).
\end{align*}
Applying Lemma~\ref{lem:convolution paracontrolled}, we deduce that
\begin{align*}
 & \|D_{4}^{1}-D_{4}^{2}\|_{\alpha+\beta_{1}-1,p/2,\infty}\\
 & \quad\quad\lesssim L_{\sigma}\tilde{C}_{\xi}\tilde{C}_{u}\big(\|u^{1}-u^{2}\|_{\alpha,p,\infty}+\|u^{\#,1}-u^{\#,2}\|_{2\alpha,p/2,\infty}+\|\xi_{1}^{1}-\xi_{1}^{2}\|_{\beta_{1}-1,p,\infty}\big).
\end{align*}
Finally, \citep[Lemna~4.2]{Promel2015} leads to
\begin{align*}
 & \|D_{5}^{1}-D_{5}^{2}\|_{\alpha+\beta_{1}-1,p/2,\infty}\lesssim L_{\sigma}\tilde{C}_{\xi}^{2}\tilde{C}_{u}\big(\|u^{1}-u^{2}\|_{\alpha,p,\infty}+\|\xi_{1}^{1}-\xi_{1}^{2}\|_{\beta_{1}-1,p,\infty}\big).
\end{align*}
Relying additionally on the estimate 
\begin{align*}
 & \|\pi(\phi_{2}\ast\xi_{2}^{1},\xi_{1}^{1})-\pi(\phi_{2}\ast\xi_{2}^{2},\xi_{1}^{2})\|_{\alpha+\beta_{1}-1,p/2,\infty}\\
 & \quad\leq\|\phi_{2}\|_{\gamma_{2}-1,1,\infty}\|\xi_{1}^{1}\|_{\beta_{1}-1,p,\infty}\|\xi_{2}^{1}-\xi_{2}^{2}\|_{\beta_{2}-1,p,\infty}+\|\phi_{2}\|_{\gamma_{2},1,\infty}\|\xi_{2}^{1}\|_{\beta_{2}-1,p,\infty}\|\xi_{1}^{1}-\xi_{1}^{2}\|_{\beta_{1}-1,p/2,\infty},
\end{align*}
we conclude that there exist a constant $C(M)$ such that 
\begin{align*}
 & \|\pi(\sigma_{1}(u^{1}),\xi_{1}^{1})-\pi(\sigma_{1}(u^{2}),\xi_{1}^{2})\|_{2\alpha,p/2,\infty}\\
 & \quad\lesssim L_{\sigma}C(M)\bigg(\|u^{1}-u^{2}\|_{\alpha,p,\infty}+\|\pi(\phi_{1}\ast\xi_{1}^{1},\xi_{1}^{1})-\pi(\phi_{1}\ast\xi_{1}^{2},\xi_{1}^{2})\|_{\alpha+\beta_{1}-1,p/2,\infty}\\
 & \qquad+\sum_{j=1,2}\big(\|\xi_{j}^{1}-\xi_{j}^{2}\|_{\beta_{1}-1,,p,\infty}+\|u_{0}^{(j),1}-u_{0}^{(j),2}\|_{\alpha,p,\infty}\big)+\|u^{\#,1}-u^{\#,2}\|_{2\alpha,p/2,\infty}\bigg).
\end{align*}
The last term can be further estimated by

\begin{align*}
 & \|u^{\#,1}-u^{\#,2}\|_{2\alpha,p/2,\infty}\\
 & \quad\leq\|u_{0}^{\#,1}-u_{0}^{\#,2}\|_{2\alpha,p/2,\infty}+\sum_{j=1,2}\|\phi_{j}\|_{\gamma_{j},1,\infty}\big(\|\pi(\sigma_{j}(u^{1}),\xi_{j}^{1})-\pi(\sigma_{j}(u^{2}),\xi_{j}^{2})\|_{2\alpha,p/2,\infty}\\
 & \qquad+\|T_{\xi_{j}^{1}}\sigma_{j}(u^{1})-T_{\xi_{j}^{2}}\sigma_{j}(u^{2})\|_{2\alpha,p/2,\infty}+\|R_{\phi_{j}}(\sigma_{j}(u^{1}),\xi_{j}^{1})-R_{\phi_{j}}(\sigma_{j}(u^{2}),\xi_{j}^{2})\|_{2\alpha,p/2,\infty}\big)\\
 & \quad\lesssim\|u_{0}^{\#,1}-u_{0}^{\#,2}\|_{2\alpha,p/2,\infty}+\tilde{C}_{\phi}L_{\sigma}C(M)\bigg(\|u^{1}-u^{2}\|_{\alpha,p,\infty}+\|u^{\#,1}-u^{\#,2}\|_{2\alpha,p/2,\infty}\\
 & \qquad+\sum_{j=1,2}\big(\|\xi_{j}^{1}-\xi_{j}^{2}\|_{\beta_{1}-1,,p,\infty}+\|u_{0}^{(j),1}-u_{0}^{(j),2}\|_{\alpha,p,\infty}\big)\\
 & \qquad+\|\pi(\phi_{1}\ast\xi_{1}^{1},\xi_{1}^{1})-\pi(\phi_{1}\ast\xi_{1}^{2},\xi_{1}^{2})\|_{\alpha+\beta_{1}-1,p/2,\infty}\bigg),
\end{align*}
where we used that $R_{\phi}(\cdot,\cdot)$ from Lemma~\ref{lem:convolution paracontrolled}
is a bounded linear operator by its definition. For $L_{\sigma}$
small enough the last inequality in combination with (\ref{eq:difference lipschitz})
(\ref{eq:lipschitz first term}) and (\ref{eq:lipschitz last term})
implies 
\begin{align*}
\|u^{1}-u^{2}\|_{\alpha,p,\infty} & \lesssim\|u_{0}^{\#,1}-u_{0}^{\#,2}\|_{2\alpha,p/2,\infty}+\hat{C}\bigg(\|\pi(\phi_{1}\ast\xi_{1}^{1},\xi_{1}^{1})-\pi(\phi_{1}\ast\xi_{1}^{2},\xi_{1}^{2})\|_{\alpha+\beta_{1}-1,p/2,\infty}\\
 & \qquad+\sum_{j=1,2}\big(\|\xi_{j}^{1}-\xi_{j}^{2}\|_{\beta_{1}-1,,p,\infty}+\|u_{0}^{(j),1}-u_{0}^{(j),2}\|_{\alpha,p,\infty}\bigg),
\end{align*}
for some constant $\hat{C}:=C(L_{\sigma},M)>0$. This Lipschitz estimate
allows to extend the Itô-Lyons map~(\ref{eq:itomap}) from smooth
driving signals $\xi_{1}$ with compact support to the space of convolutional
rough paths. 
\end{proof}

\subsection{Solutions for general vector fields\label{subsec:global solution}}

In Theorem~\ref{thm:solve} we assumed that
\[
\Delta=\|\sigma_{1}\|_{C^{3}}\|\phi_{1}\|_{\gamma_{1},1,\infty}+\|\sigma_{2}\|_{C^{2}}\|\phi_{2}\|_{\gamma_{2},1,\infty}
\]
is sufficiently small, which can be interpreted as a flatness condition
on the vector fields $\sigma_{1},\sigma_{2}$.\textcolor{blue}{{} }In
this subsection we discuss how the existence and uniqueness results
can be extended to general vector fields~$\sigma_{1},\sigma_{2}$
applying a scaling argument in the spirit of \citet{Gubinelli2015}
to a localized version of (\ref{eq:convol}).\textcolor{blue}{{} }Interestingly,
$\Delta$ is small if the (localised) kernels $\phi_{1},\phi_{2}$
are supported on a sufficiently small domain and if $\gamma_{1},\gamma_{2}<1$,
cf. Remark~\ref{rem:smallKernel}. 
\begin{thm}
\label{thm:localSol}Let $p\in[3,\infty]$, $0<\beta_{1}\le\beta_{2}\le1$
and $0<\gamma_{1}\le\gamma_{2}$ satisfy \textup{$\alpha:=\beta_{1}+\gamma_{1}-1\in(\frac{1}{3},1)$,
$\alpha+\beta_{1}<1<2\alpha+\beta_{1}$ and} $\alpha+\beta_{2}>1$.
If $\gamma_{j}>1$ for $j=1,2$, let also $\beta_{j}>1/p$ be fullfilled.
Suppose that 
\begin{enumerate}
\item $\sigma_{1}\in C^{3}$ and $\sigma_{2}\in C^{2}$ with $\sigma_{1}(0)=\sigma_{2}(0)=0$,
\item $\phi_{j}\in\mathcal{B}_{1,\infty}^{\gamma_{j}}$ such that there
exists $r_{j}\in\R$ with $\|(\cdot-r_{j})\phi_{j}\|_{\gamma_{j}+1,1,\infty}<\infty,$
for $j=1,2$, 
\item $(\xi_{1},\mu)\in\mathcal{B}_{p}^{\beta_{1}-1,\gamma_{1}}(\phi_{1})$
and $\xi_{2}\in\mathcal{B}_{p,\infty}^{\beta_{2}-1}$,
\item $u_{0}\in\mathcal{B}_{p/2,\infty}^{2\alpha}$. 
\end{enumerate}
Additionally, we impose the structural assumption on the kernel~$\phi_{1}$:
\begin{enumerate}
\item[(v)]  There is some $\psi\in\mathcal{B}_{1,\infty}^{s+\delta}$\textup{
for $\delta>(2-2\beta_{1})\vee1$ and $s\in[0,1)$ such that $\phi_{1}(x)=(x-r_{1})^{-s}\ind_{(r_{1},\infty)}(x)\psi(x)$.}
\end{enumerate}
Let $\chi$ be a $C^{\infty}$ function with $\supp\chi\subset[-2,2]$
and $\chi(x)=1$ for $x\in[-1,1]$. Then there is some $\lambda\in(0,1)$
depending on $(u_{0},(\xi_{1},\mu),\xi_{2})$, $\phi_{1},\phi_{2}$
and $\sigma_{1},\sigma_{2}$, such that the localized Volterra equation
\begin{equation}
u(t)=u_{0}^{loc,\lambda}(t)+\big(\phi_{1}^{loc,\lambda}*(\sigma_{1}(u)\xi_{1})\big)(t)+\big(\phi_{2}^{loc,\lambda}*(\sigma_{2}(u)\xi_{2})\big)(t),\quad t\in\R,\label{eq:convLoc-1}
\end{equation}
with kernels $\phi_{j}^{loc,\lambda}:=\chi(\lambda^{-1}\cdot)\phi_{j}$,
$j=1,2$, and initial condition $u_{0}^{loc,\lambda}:=\chi(\lambda^{-1}\cdot)u_{0}$
has a unique solution in the space $\mathcal{B}_{p,\infty}^{\alpha-\epsilon}$
for any $\varepsilon>0$.
\end{thm}

\begin{proof}
Let us introduce the dilation operator $\Lambda_{\lambda}f:=f(\lambda\cdot)$
for any $f\in\mathcal{S}^{\prime}$. For $\delta,\lambda>0$, we first
observe that 
\begin{align*}
u & =u_{0}^{loc,\lambda}+\sum_{j=1,2}\phi_{j}^{loc,\lambda}\ast(\sigma_{j}(u)\xi_{j})\\
 & =u_{0}^{loc,\lambda}+\sum_{j=1,2}\int_{\R}\frac{\lambda}{\delta}\phi_{j}^{loc,\lambda}(\cdot-\lambda s)\delta\sigma_{j}\big(u(\lambda s)\big)\Lambda_{\lambda}\xi_{j}(s)\d s.
\end{align*}
Therefore, $u$ solves (\ref{eq:convLoc-1}) if and only if $\tilde{u}:=\Lambda_{\lambda}u$
solves 
\begin{equation}
\tilde{u}=\Lambda_{\lambda}u_{0}^{loc,\lambda}+\sum_{j=1,2}\int_{\R}\frac{\lambda}{\delta}\Lambda_{\lambda}\phi_{j}^{loc,\lambda}(\cdot-s)\delta\sigma_{j}\big(\tilde{u}(s)\big)\Lambda_{\lambda}\xi_{j}(s)\d s.\label{eq:uTilde1-1}
\end{equation}
Applying the dilation estimate from \citep[Lem.~2.3]{Promel2015}
we have
\begin{equation}
\big\|\Lambda_{\lambda}\xi_{j}\big\|_{\beta_{j}-1,p,\infty}\lesssim(1+\lambda^{\beta_{j}-1}|\log\lambda|)\lambda^{-1/p}\|\xi_{j}\|_{\beta_{j}-1,p,\infty}.\label{eq:tildeXi-1}
\end{equation}
The auxiliary Lemma~\ref{lem:dialation} yields
\begin{align}
 & \begin{split} & \|\Lambda_{\lambda}\phi_{j}^{loc,\lambda}\|_{\gamma_{j},1,\infty}=\|\chi\Lambda_{\lambda}\phi_{j}\|_{\gamma_{j},1,\infty}\lesssim\lambda^{(\gamma'\wedge1)-1}|\log\lambda|\|\phi_{j}\|_{\gamma_{j},1,\infty}\qquad\text{for any }\gamma'<\gamma_{j},\\
 & \|\Lambda_{\lambda}u_{0}^{loc,\lambda}\|_{2\alpha,p/2,\infty}=\|\chi\Lambda_{\lambda}u_{0}\|_{2\alpha,p/2,\infty}\lesssim\lambda^{\alpha-1/p}|\log\lambda|\|u_{0}\|_{2\alpha,p/2,\infty}.
\end{split}
\label{eq:tildePhi-1}
\end{align}
We now may choose $\delta$ such that the norms of the scaled noise
and kernels remain bounded while $\|\delta\sigma_{j}\|_{C^{3}}\to0$
for $\delta\to0$. Due to the assumptions on the parameters, we have
$\frac{1}{p}<\beta_{j}+(\gamma_{j}\wedge1)-1$ such that there is
some $0<\tau<(\beta_{j}+(\gamma_{j}\wedge1)-1-1/p)/2$ and we can
choose $\delta=\lambda^{\beta_{j}+(1\wedge\gamma_{j})-1-1/p-2\tau}$.
Setting $\tilde{u}_{0}^{loc}:=\Lambda_{\lambda}u_{0}^{loc}$ and
\begin{align*}
\tilde{\xi}_{j} & :=\lambda^{1+1/p-\beta_{j}+\tau}\Lambda_{\lambda}\xi_{j},\quad\tilde{\phi}_{j}^{loc}:=\lambda^{1-(1\wedge\gamma_{j})+\tau}\Lambda_{\lambda}\phi_{j}^{loc},\quad\tilde{\sigma}_{j}:=\delta\sigma_{j},
\end{align*}
we obtain from (\ref{eq:uTilde1-1}) the dilated representation 
\begin{equation}
\tilde{u}:=\tilde{u}_{0}^{loc}+\sum_{j=1,2}\big(\tilde{\phi}_{j}^{loc}\ast(\tilde{\sigma}_{j}(\tilde{u})\tilde{\xi}_{j}\big).\label{eq:uTilde2-1}
\end{equation}
Owing to (\ref{eq:tildeXi-1}) and (\ref{eq:tildePhi-1}), we have
uniformly in $\lambda>0$
\[
\|\tilde{\xi}_{j}\|_{\beta_{j}-1,p,\infty}\lesssim\|\xi_{j}\|_{\beta_{j}-1,p,\infty},\qquad\|\tilde{\phi}_{j}^{loc}\|_{\gamma_{j},1,\infty}\lesssim\|\phi_{j}\|_{\gamma_{j},1,\infty},\qquad\|\tilde{u}_{0}^{loc}\|_{2\alpha,p/2,\infty}\lesssim\|u_{0}\|_{2\alpha,p/2,\infty}.
\]
We may now choose $\lambda$ and thus $\delta$ sufficiently small
such that Theorem~\ref{thm:solve} applies to (\ref{eq:uTilde2-1})
when $\gamma_{j}$ and $\alpha$ are replaced by $\tilde{\gamma}_{j}:=\gamma_{j}-\epsilon$
and $\tilde{\alpha}:=\alpha-\epsilon=\beta_{1}+\tilde{\gamma}_{1}-1$,
respectively, for some sufficiently small $\epsilon>0$. Since $\|\tilde{\phi}_{j}^{loc}\|_{\tilde{\gamma}_{j},1,\infty}\lesssim\|\tilde{\phi}_{j}^{loc}\|_{\gamma_{j},1,\infty}\lesssim\|\phi_{j}\|_{\gamma_{j},1,\infty}$,
it only remains to verify bounds for $\|(\cdot-r_{j})\tilde{\phi}_{j}^{loc}\|_{\tilde{\gamma}_{j}+1,1,\infty}$
and $\|\pi(\tilde{\phi}_{1}^{loc}\ast\tilde{\xi}_{1},\tilde{\xi}_{1})\|_{\tilde{\alpha}+\beta_{1}-1,p/2,\infty}$
uniformly in $\lambda$. Setting $r_{j}=0$ without loss of generality,
we obtain from Lemma~\ref{lem:dialation} for $\gamma'=1\wedge\gamma_{j}-\tau/2$
\begin{align*}
\|x\tilde{\phi}_{j}^{loc}(x)\|_{\tilde{\gamma}_{j}+1,1,\infty} & =\lambda^{-(1\wedge\gamma_{j})+\tau}\|\chi(x)\chi(x/2)\Lambda_{\lambda}(x\phi_{j}(x))\|_{\tilde{\gamma}_{j}+1,1,\infty}\\
 & \lesssim\lambda^{\gamma'-(1\wedge\gamma_{j})+\tau}|\log\lambda|\big(\|x\phi_{j}(x)\|_{\gamma_{j}+1,1,\infty}+\|\phi_{j}\|_{\gamma_{j},1,\infty}\big).
\end{align*}
Moreover, we have due to \citep[Lem.~2.3]{Promel2015}, Lemma~\ref{lem:scaling resonant term}
and $\alpha+\beta_{1}<1$:
\begin{align}
\begin{split} & \|\pi(\tilde{\phi}_{1}^{loc}\ast\tilde{\xi}_{1},\tilde{\xi}_{1})\|_{\tilde{\alpha}+\beta_{1}-1,p/2,\infty}\\
 & \quad=\lambda^{2+2/p-2\beta_{1}-(\gamma_{1}\wedge1)+3\tau}\|\pi(\Lambda_{\lambda}(\phi_{1}^{loc,\lambda}\ast\xi_{1}),\Lambda_{\lambda}\xi_{1})\|_{\tilde{\alpha}+\beta_{1}-1,p/2,\infty}\\
 & \quad\leq\lambda^{2+2/p-2\beta_{1}-(\gamma_{1}\wedge1)+3\tau}\big(\|\Lambda_{\lambda}(\pi(\phi_{1}^{loc,\lambda}\ast\xi_{1}),\xi_{1}))\|_{\tilde{\alpha}+\beta_{1}-1,p/2,\infty}\\
 & \qquad+\|\pi(\Lambda_{\lambda}(\phi_{1}^{loc,\lambda}\ast\xi_{1}),\Lambda_{\lambda}\xi_{1})-\Lambda_{\lambda}(\pi(\phi_{1}^{loc,\lambda}\ast\xi_{1},\xi_{1}))\|_{\tilde{\alpha}+\beta_{1}-1,p/2,\infty}\big)\\
 & \quad\lesssim\lambda^{\tilde{\alpha}+1-\beta_{1}-(\gamma_{1}\wedge1)+3\tau}|\log\lambda|\|\pi(\phi_{1}^{loc,\lambda}\ast\xi_{1},\xi_{1})\|_{\tilde{\alpha}+\beta_{1}-1,p/2,\infty}\\
 & \qquad+\lambda^{\tilde{\alpha}+1-\beta_{1}-(\gamma_{1}\wedge1)+3\tau}\|\phi_{1}^{loc,\lambda}\ast\xi_{1}\|_{\tilde{\alpha},p,\infty}\|\xi_{1}\|_{\beta_{1}-1,p,\infty}\\
 & \qquad+\lambda^{2+2/p-2\beta_{1}-(\gamma_{1}\wedge1)+3\tau}\|\Lambda_{\lambda}(\phi_{1}^{loc,\lambda}\ast\xi_{1})\|_{\tilde{\alpha},p,\infty}\|\Lambda_{\lambda}\xi_{1}\|_{\beta_{1}-1,p,\infty}.
\end{split}
{\color{red}}\label{eq:locRes}
\end{align}
The last two terms in (\ref{eq:locRes}) can be bounded by Young's
inequality
\begin{align*}
\lambda^{\tilde{\alpha}+1-\beta_{1}-(\gamma_{1}\wedge1)+3\tau}\|\phi_{1}^{loc,\lambda}\ast\xi_{1}\|_{\tilde{\alpha},p,\infty}\|\xi_{1}\|_{\beta_{1}-1,p,\infty} & \lesssim\lambda^{3\tau}\|\phi_{1}^{loc,\lambda}\|_{\tilde{\gamma},p,\infty}\|\xi_{1}\|_{\beta_{1}-1,p,\infty}^{2}
\end{align*}
and, in combination with \citep[Lem.~2.3]{Promel2015} and Lemma~(\ref{lem:dialation})
for $\epsilon<\tau$,
\begin{align*}
 & \lambda^{2+2/p-2\beta_{1}-(\gamma_{1}\wedge1)+3\tau}\|\Lambda_{\lambda}(\phi_{1}^{loc,\lambda}\ast\xi_{1})\|_{\tilde{\alpha},p,\infty}\|\Lambda_{\lambda}\xi_{1}\|_{\beta_{1}-1,p,\infty}\\
 & \quad\lesssim\lambda^{3+2/p-2\beta_{1}-(\gamma_{1}\wedge1)+3\tau}\|\Lambda_{\lambda}\phi_{1}^{loc,\lambda}\|_{\tilde{\gamma},1,\infty}\|\Lambda_{\lambda}\xi_{1}\|_{\beta_{1}-1,p,\infty}^{2}\\
 & \quad\le\lambda^{\tau}|\log\lambda|^{3}\|\phi_{1}^{loc,\lambda}\|_{\tilde{\gamma},1,\infty}\|\xi_{1}\|_{\beta_{1}-1,p,\infty}^{2}.
\end{align*}
Choosing $q,q^{\prime}\in[1,\infty)$ such that $\frac{1}{q^{\prime}}+\frac{1}{q}=1$
and $\gamma_{1}>\frac{1}{q}>\tilde{\gamma}_{1}$, we observe 
\begin{align}
\begin{split} & \|\phi_{1}^{loc,\lambda}\|_{\tilde{\gamma},1,\infty}\\
 & \quad\lesssim\|T_{\phi_{1}}\chi(\lambda^{-1}\cdot)\|_{\tilde{\gamma}_{1},1,\infty}+\|T_{\chi(\lambda^{-1}\cdot)}\phi_{1}\|_{\tilde{\gamma}_{1},1,\infty}+\|\pi(\phi_{1},\chi(\lambda^{-1}\cdot))\|_{\tilde{\gamma}_{1},1,\infty}\\
 & \quad\lesssim\|\phi_{1}\|_{L^{q^{\prime}}}\|\chi(\lambda^{-1}\cdot)\|_{\tilde{\gamma}_{1},q,\infty}+\|\chi(\lambda^{-1}\cdot)\|_{L^{\infty}}\|\phi_{1}\|_{\tilde{\gamma}_{1},1,\infty}+\|\phi_{1}\|_{\tilde{\gamma}_{1},1,\infty}\|\chi(\lambda^{-1}\cdot)\|_{\beta_{1}-1,\infty,\infty}\\
 & \quad\lesssim\|\phi_{1}\|_{\gamma_{1},1,\infty}(1+\lambda^{-\tilde{\gamma}_{1}}|\log\lambda^{-1}|)\lambda^{\frac{1}{q}}\|\chi\|_{\tilde{\gamma}_{1},q,\infty}+\|\chi\|_{L^{\infty}}\|\phi_{1}\|_{\gamma_{1},1,\infty}\\
 & \qquad+\|\phi_{1}\|_{\gamma_{1},1,\infty}(1+\lambda^{1-\beta_{1}}|\log\lambda^{-1}|)\|\chi\|_{\beta_{1}-1,\infty,\infty},
\end{split}
\label{eq:estimate phi local}
\end{align}
where we applied Bony's decomposition, \citep[Lem.~2.3]{Promel2015}
and Besov embeddings. Hence,
\[
\|\phi_{1}^{loc,\lambda}\|_{\tilde{\gamma},1,\infty}\lesssim\|\phi_{1}\|_{\gamma_{1},1,\infty}
\]
and we can estimate (\ref{eq:locRes}) by
\begin{align}
\begin{split} & \|\pi(\tilde{\phi}_{1}^{loc}\ast\tilde{\xi}_{1},\tilde{\xi}_{1})\|_{\tilde{\alpha}+\beta_{1}-1,p/2,\infty}\\
 & \quad\lesssim\lambda^{\tilde{\alpha}+1-\beta_{1}-(\gamma_{1}\wedge1)+3\tau}|\log\lambda|\|\pi(\phi_{1}^{loc,\lambda}\ast\xi_{1},\xi_{1})\|_{\tilde{\alpha}+\beta_{1}-1,p/2,\infty}+\|\phi_{1}\|_{\tilde{\gamma},p,\infty}\|\xi_{1}\|_{\beta_{1}-1,p,\infty}^{2}\\
 & \quad\lesssim\|\pi(\phi_{1}\ast\xi_{1},\xi_{1})\|_{\alpha+\beta_{1}-1,p/2,\infty}+\big\|\pi\big(((1-\Lambda_{1/\lambda}\chi)\phi_{1})\ast\xi_{1},\xi_{1}\big)\big\|_{\tilde{\alpha}+\beta_{1}-1,p/2,\infty}\\
 & \qquad+\|\phi_{1}\|_{\tilde{\gamma},p,\infty}\|\xi_{1}\|_{\beta_{1}-1,p,\infty}^{2}.
\end{split}
\label{eq:locRes 2}
\end{align}
It remains to estimate the term $\pi\big(((1-\Lambda_{1/\lambda}\chi)\phi_{1})\ast\xi_{1},\xi_{1}\big)$
since the other terms can be seen to be uniformly bounded in $\lambda\in(0,1]$
keeping in mind~(\ref{eq:estimate phi local}). We use that the potential
irregularity of $\phi_{1}$ at the origin is smoothed out. Setting
$\epsilon':=(1-\alpha-\beta_{1})+\epsilon$ such that $\epsilon-2(\beta_{1}-1)=\gamma+\epsilon'$,
we can bound
\begin{align*}
\big\|\pi\big(((1-\Lambda_{1/\lambda}\chi)\phi_{1})\ast\xi_{1},\xi_{1}\big)\big\|_{\tilde{\alpha}+\beta_{1}-1,p/2,\infty} & \lesssim\big\|\pi\big(((1-\Lambda_{1/\lambda}\chi)\phi_{1})\ast\xi_{1},\xi_{1}\big)\big\|_{\epsilon,p/2,\infty}\\
 & \lesssim\|((1-\Lambda_{1/\lambda}\chi)\phi_{1})\ast\xi_{1}\|_{\epsilon-\beta_{1}+1,p,\infty}\|\xi_{1}\|_{\beta_{1}-1,p,\infty}\\
 & \lesssim\|((1-\Lambda_{1/\lambda}\chi)\phi_{1})\|_{\gamma+\epsilon',1,\infty}\|\xi_{1}\|_{\beta_{1}-1,p,\infty}^{2}.
\end{align*}
We will now use the kernel assumption $\phi_{1}(x)=x^{-s}\psi(x)\ind_{[0,\infty)}(x)$.
According to \citep[Corollary~2.9.3]{triebel2010} and the proof of
\citep[Theorem~2.9.1]{triebel2010}, the extension operator 
\begin{align*}
S_{0}\colon\{f\in\mathcal{B}_{p,\infty}^{\delta}(\R_{+}):f(0)=0\} & \to\mathcal{B}_{p,\infty}^{\delta}(\R),\quad f\mapsto\tilde{f}(x):=\begin{cases}
f(x), & x\ge0\\
0, & x<0
\end{cases},
\end{align*}
is bounded and linear if $\frac{1}{p}<\delta<\frac{1}{p}+1$. In particular,
for any function $f\in\mathcal{B}_{p,\infty}^{\delta}$ with $f(0)=0$
we conclude with restriction $f|_{\R_{+}}$ to $\R_{+}$ that
\begin{align}
\begin{split}\|f\ind_{[0,\infty)}\|_{\delta,p,q} & \lesssim\|f|_{\R_{+}}\|_{\mathcal{B}_{p,\infty}^{\delta}(\R_{+})}\\
 & =\inf\big\{\|g\|_{\delta,p,\infty}:g\in\mathcal{B}_{p,\infty}^{\delta},g(x)=f(x)\,\forall x\ge0\big\}\le\|f\|_{\delta,p,\infty}.
\end{split}
\label{eq:extension}
\end{align}
Since $\chi$ is constant one in a neighbourhood of the origin, we
may apply (\ref{eq:extension}) to $f(x)=(1-\chi(\lambda^{-1}x))x^{-s}\psi(x)$
and any $\delta\in((\gamma+\epsilon')\vee1,2)$. Together with Lemma~\ref{lem:dialation}
we obtain for $\epsilon''\in(s,1)$
\begin{align*}
\|(1-\chi(\lambda^{-1}\cdot))\phi\|_{\gamma+\epsilon',1,\infty} & \lesssim\Big\|\frac{1-\chi(x/\lambda)}{x^{s}}\psi\Big\|_{\delta,1,\infty}=\lambda^{-s}\Big\|\frac{1-\chi(\lambda^{-1}x)}{(x/\lambda)^{s}}\psi\Big\|_{\delta,1,\infty}\\
 & \lesssim\lambda^{\epsilon''-s}|\log\lambda|\|x^{-s}(1-\chi(x))\|_{\delta,1,\infty}\|\psi\|_{\delta,1/(1-\epsilon''),\infty}\\
 & \lesssim\|x^{-s}(1-\chi(x))\|_{\delta,1,\infty}\|\psi\|_{\delta+\epsilon'',1,\infty}.
\end{align*}
In combination with (\ref{eq:locRes 2}), we observe a uniform bound
for $\|\pi(\tilde{\phi}_{1}^{loc}\ast\tilde{\xi}_{1},\tilde{\xi}_{1})\|_{\tilde{\alpha}+\beta_{1}-1,p/2,\infty}$
which concludes the proof. 
\end{proof}
\begin{rem}
\label{rem:local soluton} Note that under the support assumptions
$\supp\phi_{i}\subset[0,\infty)$ and $\supp\xi_{i}\subset[0,\infty)$
for $i=1,2$ the solution of the localized equation~(\ref{eq:convLoc-1})
coincide with the solution of the original Volterra equation~(\ref{eq:convol})
on a small time horizon, provided the initial condition is, e.g.,
a constant or has sufficiently small support. Based on this observation,
one can iteratively solve the Volterra equation~(\ref{eq:convol})
in order to obtain a global solution using a classical pasting argument.
In case of Volterra equations this procedure will require carefully
chosen support conditions on the kernel functions and the noise terms.
In the special case of classical rough differential equations (which
corresponds to $\phi_{1}=\phi_{2}=\ind_{[0,\infty)}$, see Subsection~\ref{subsec: rough differential equation})
such procedure was carried out in, e.g., \citep{Gubinelli2015} and~\citep{Promel2015}.
\end{rem}

\begin{rem}
The assumption~$(v)$ on the kernel $\phi_{1}$ is fairly flexible
and covers many typical applications. For $s=0$ we may replace $\ind_{(0,\infty)}$
by $\ind_{[0,\infty)}$ and we obtain a class of regular kernels $\phi_{1}=\ind_{[0,\infty)}\psi$
for some\textcolor{blue}{{} }$\psi\in\mathcal{B}_{1,\infty}^{\delta}$,
$\delta\in(1\vee(\gamma_{1}+1-\alpha-\beta_{1}),\gamma_{1}+\alpha)$,
(setting $r_{1}=0$ for simplicity). In this case the singularity
at $0$ is not more severe than a jump such that we recover many features
of ordinary rough differential equations, especially $\gamma_{1}=1$.
The condition $\psi\in\mathcal{B}_{1,\infty}^{\delta}$ is quite weak
and includes, for instance, the kernels studied in \citep{Deya2011}
where $\psi\in C^{3}$. On the one hand $\delta$ has to be larger
than $\gamma$ such that $\psi$ is more regular than $\phi_{1}$
itself and on the other hand $\delta>1$ ensures that $\psi$ is continuous.
For $s>0$ and $\psi(0)\neq0$ the kernel is singular. Note that the
degree of the singularity is constrained by the regularity assumption
$\phi_{1}\in\mathcal{B}_{1,\infty}^{\gamma_{1}}$ implying $s\le1-\gamma_{1}$.
For example, if $\xi_{1}$ is white noise, then $\alpha>1/3$ implies
$\gamma>5/6$ such that we require $s\in[0,1/6)$. For further examples
we refer to Section~\ref{sec:examples}.
\end{rem}

\begin{rem}
\label{rem:smallKernel}More generally, for singular kernels $\phi_{1}$
which do not satisfy assumption~$(v)$, a uniform bound (in~$\lambda$)
of the localised resonant term $\|\pi(\phi_{1}^{loc,\lambda}\ast\xi_{1},\xi_{1})\|_{\alpha+\beta_{1}-1-\epsilon,p/2,\infty}$
from~(\ref{eq:locRes}) could be directly assumed. Indeed, we will
see in the stochastic construction below (see the proof of Theorem~\ref{thm: contruction resonant term})
that this resonant term is typically of order $\|\phi_{1}^{loc,\lambda}\|_{\gamma,1,\infty}$,
which can be bounded by Lemma~\ref{lem:dialation} as
\begin{align*}
\|\phi_{1}^{loc,\lambda}\|_{\gamma-\epsilon,1,\infty} & =\|\chi(\lambda^{-1}\cdot)\phi_{1}\|_{\gamma-\epsilon,1,\infty}\\
 & \lesssim\lambda^{\epsilon}|\log\lambda|\|\chi\|_{\gamma-\epsilon,1,\infty}\|\phi_{1}\|_{\gamma-\epsilon,1/(1-\epsilon),\infty}\lesssim\lambda^{\varepsilon/2}\|\chi\|_{\gamma,1,\infty}\|\phi_{1}\|_{\gamma,1,\infty}.
\end{align*}
Note that the last estimate is arbitrary small for sufficiently small
$\lambda$, Theorem~\ref{thm:solve} can then be directly applied
to the localised equation~(\ref{eq:convLoc-1}) without an additional
scaling argument if $\gamma_{1},\gamma_{2}<1$.
\end{rem}

\section{The resonant term\label{sec:resonant term}}

In order to apply the existence and unique results provided in Section~\ref{sec:convolution}
to stochastic Volterra equations, it is often necessary to construct
the resonant term $\text{\ensuremath{\pi}(\ensuremath{\phi\ast\xi},\ensuremath{\xi})}$
for the driving stochastic processes. In the case of regular kernels~$\phi\in\mathcal{B}_{1,\infty}^{1}$,
the existence of the resonant term $\text{\ensuremath{\pi}(\ensuremath{\phi\ast\xi},\ensuremath{\xi})}$
is equivalent to the existence of the classical rough path, see Subsection~\ref{subsec:rough path}.
However, for singular kernels $\phi\in\mathcal{B}_{1,\infty}^{\delta}$
with $\delta<1$ this equivalence does not hold anymore and it is
necessary to include the kernel~$\phi$ in the definition of the
`'rough path'{}', see Example~\ref{ex: counter example}. Therefore,
we provide a probabilistic construction of convolutional rough paths
for a wide class of Gaussian processes in Subsection~\ref{subsec:construction resonant term}.

\subsection{Relation to rough path theory\label{subsec:rough path}}

For a regular kernel $\phi=\ind_{[0,\infty)}\psi$ and a rough signal
$\xi$ the resonant term $\pi(\phi*\xi,\xi)$ can be reduced to the
resonant term $\pi(\ind_{[0,\infty)}*\xi,\xi)=\pi(\int_{-\infty}^{t}\d\xi(s),\xi)$
between $\xi$ and its anti-derivative. The latter corresponds to
the classical rough path integral, cf. \citep{Gubinelli2015}. Considering
the Volterra equation on some bounded time interval, we may use $\pi((\ind_{[0,\infty)}\chi)*\xi,\xi)$
instead of $\pi(\ind_{[0,\infty)}*\xi,\xi)$ where $\chi$ is some
smooth compactly supported function being constant one in a neighbourhood
of the origin. Note that $\chi$ only ensures integrability of the
kernel, while the characteristic regularity properties of $\ind_{[0,\infty)}$
are preserved. In particular, the (weak) derivative of $(\ind_{[0,\infty)}\chi)*\xi$
is $\xi$ up to some additional smooth remainder.
\begin{lem}
\label{lem:roughpath} Let $\xi\in\mathcal{B}_{p,\infty}^{\beta-1}$
for $\beta>0$, $p\in[2,\infty]$ and $(\xi^{n})_{n}\subset\mathcal{S}$
be such that $\xi^{n}\to\xi$ in $\mathcal{B}_{p,\infty}^{\beta-1}$
as $n\to\infty$. Suppose that $\chi\in C^{\infty}$ is a smooth compactly
supported function with $\chi(0)=1$ and $\phi:=\psi\ind_{[0,\infty)}\in\mathcal{B}_{1,\infty}^{1}$
for some $\psi\in\mathcal{B}_{1,\infty}^{\delta}$ with $\delta\in(1\vee2(1-\beta),2)$
and $\psi(0)\neq0$. Then, $\pi(\phi\ast\xi,\xi):=\lim_{n\to\infty}\pi(\phi\ast\xi^{n},\xi^{n})$
exists in $\mathcal{B}_{p/2,\infty}^{2\beta-1}$ if and only if $\pi\big((\ind_{[0,\infty)}\chi)\ast\xi,\xi\big):=\lim_{n\to\infty}\pi\big((\ind_{[0,\infty)}\chi)\ast\xi^{n},\xi^{n}\big)$
exists in $\mathcal{B}_{p/2,\infty}^{2\beta-1}$. In this case, one
has
\begin{align*}
\pi(\phi\ast\xi,\xi)-\phi(0)\pi\big((\ind_{[0,\infty)}\chi)\ast\xi,\xi\big) & \in\mathcal{B}_{p/2,\infty}^{\delta-2(1-\beta)}.
\end{align*}
\end{lem}

\begin{proof}
Let $(\xi^{n})_{n}\subset\mathcal{S}$ be such that $\xi^{n}\to\xi$
in $\mathcal{B}_{p,\infty}^{\beta-1}$ and $\pi(\phi\ast\xi,\xi):=\lim_{n\to\infty}\pi(\phi\ast\xi^{n},\xi^{n})$
in $\mathcal{B}_{p/2,\infty}^{2\beta-1}$. We first observe that
\begin{align*}
 & \pi\big((\ind_{[0,\infty)}\chi)\ast\xi^{n},\xi^{n}\big)\\
 & \quad=\psi(0)^{-1}\pi(\phi\ast\xi^{n},\xi^{n})-\big(\psi(0)^{-1}\pi(\phi*\xi^{n},\xi^{n})-\pi\big((\ind_{[0,\infty)}\chi)\ast\xi^{n},\xi^{n}\big)\big).
\end{align*}
Since the first term converges by assumption, it is sufficient to
consider the other two. Setting $\epsilon:=\delta-2(1-\beta)>0$,
Bony's paraproduct estimates and the generalised Young inequality
yield
\begin{align*}
 & \big\|\pi(\phi\ast\xi^{n},\xi^{n})-\psi(0)\pi\big((\ind_{[0,\infty)}\chi)\ast\xi^{n},\xi^{n}\big)\big\|_{\epsilon,p/2,\infty}\\
 & \quad=\big\|\pi\big(\big((\psi-\psi(0)\chi)\ind_{[0,\infty)}\big)\ast\xi^{n},\xi^{n}\big)\big\|_{\epsilon,p/2,\infty}\\
 & \quad\lesssim\big\|\big((\psi-\psi(0)\chi)\ind_{[0,\infty)}\big)\ast\xi^{n}\big\|_{\epsilon-\beta_{1}+1,p,\infty}\|\xi^{n}\|_{\beta-1,p,\infty}\\
 & \quad\lesssim\big\|\big((\psi-\psi(0)\chi)\ind_{[0,\infty)}\big)\big\|_{\delta,1,\infty}\|\xi^{n}\|_{\beta-1,p,\infty}^{2}.
\end{align*}
Applying the estimate (\ref{eq:extension}) for the regularity $1<\delta<2$,
we obtain
\begin{align*}
\big\|\big((\psi-\psi(0)\chi)\ind_{[0,\infty)}\big)\big\|_{\delta,1,\infty} & \lesssim\|\psi-\psi(0)\chi\|_{\delta,1,\infty}\\
 & \le\|\psi\|_{\delta,1,\infty}+|\psi(0)|\|\chi\|_{\delta,1,\infty}\lesssim\big(1+\|\chi\|_{\delta,1,\infty}\big)\|\psi\|_{\delta,1,\infty}.
\end{align*}
As $\xi^{n}\to\xi$ in $\mathcal{B}_{p,\infty}^{\beta-1}$ and $\mathcal{B}_{p/2,\infty}^{\delta-2(1-\beta)}\subset\mathcal{B}_{p/2,\infty}^{2\beta-1}$
, this implies one direction of the assertion. The converse direction
follows analogously.
\end{proof}
\begin{rem}
For $\alpha+\beta_{1}<1$ the condition $\delta>2(1-\beta)=1-\alpha-\beta+\gamma>\gamma$
is in line with the regular case in Theorem~\ref{thm:localSol}.
Lemma~\ref{lem:roughpath} especially implies that for regular kernels
the results, developed in Section~\ref{sec:convolution} for convolutional
rough paths, can be applied to all stochastic processes which can
be enhanced to rough paths such as semi-martingales and various Gaussian
processes, cf. \citet{Friz2010}.
\end{rem}

While in the regular case the additional information can be reduced
to $\pi((\ind_{[0,\infty)}\chi)*\xi,\xi)$, the following example
illustrates that for singular Volterra equations it is indeed necessary
to include the kernel into the resonant term, i.e., it is not sufficient
to take only this ``classical'' resonant term into account.
\begin{example}
\label{ex: counter example} Consider the following $2$-dimensional
Volterra equation
\begin{align*}
u^{1} & =\phi*\xi^{1},\\
u^{2} & =\phi*(u^{1}\xi^{2})=\phi*\big((\phi*\xi^{1})\xi^{2}\big),
\end{align*}
with some singular kernel $\phi\in\mathcal{B}_{1,\infty}^{\gamma}$
for $\gamma\in(0,1)$ and $(\xi^{1},\xi^{2})\in\mathcal{B}_{p,\infty}^{\beta-1}$.
We notice that $\pi((\ind_{[0,\infty)}\chi)*\xi^{1},\xi^{2})\in\mathcal{B}_{p/2,\infty}^{2\beta-1}$
is well-defined if $2\beta>1$, but $\pi((\phi*\xi^{1}),\xi^{2})\in\mathcal{B}_{p/2,\infty}^{2\beta-2+\gamma}$
is not well-defined if $2\beta<2-\gamma$. Hence, for $\gamma<1$
and $1/2<\beta<1-\gamma/2$ the product $(\phi*\xi^{1})\xi^{2}$ is
not well-defined while the resonant term $\pi((\ind_{[0,\infty)}\chi)*\xi^{1},\xi^{2})$
gives no additional information.

In order to make the example more explicit, we set $\xi^{i}=\dd\theta^{i}$,
$i=1,2$, with $\theta^{i}=B_{H}^{i}\tilde{\chi}$ for fractional
Brownian motions $B_{H}^{i}$ with Hurst index $H\in(1/2,2/3)$ and
a compactly supported function $\tilde{\chi}\in C^{\infty}$. Moreover,
we choose the kernel $\phi(s)=s^{r-1}\ind_{(0,\infty)}(s)\tilde{\chi}(s)$
for $r\in(4/3-H,2-2H)$, which is associated to the fractional integration
operator of order $r$, cf. Section~\ref{sec:fractSDEs}. We then
have for any arbitrarily small $\epsilon>0$ that $(\theta^{1},\theta^{2})\in\mathcal{B}_{\infty,\infty}^{\beta}$
and $\phi\in\mathcal{B}_{1,\infty}^{\gamma}$ with $\beta=H-\epsilon$
and $\gamma=r-\epsilon$. By the choice of $r$ and $H$, we indeed
have $1/2<\beta<1-\gamma/2$, but also $\alpha:=\beta+\gamma-1>1/3$
and $2\alpha+\beta>1$ such that Theorem~\ref{thm:solve} is applicable. 
\end{example}

\subsection{Stochastic construction of the resonant term\label{subsec:construction resonant term}}

While Lemma~\ref{lem:roughpath} allows for the construction of the
resonant term $\pi(\phi\ast\xi,\xi)$ for a regular kernels~$\phi$
and a large class of noise processes~$\xi$ via rough path theory,
the aim of this section is to directly construct $\pi(\phi\ast\xi,\xi)$.
This is particularly interesting for singular kernels, but also gives
some deeper understanding on the interplay between the analytical
object $\pi(\phi\ast\xi,\xi)$ and the stochastic behaviour of $\xi$.
We investigate a class of stochastic processes admitting a series
expansion 
\begin{equation}
\xi_{t}=\sum_{n\ge1}a_{n}(t)\zeta_{n},\quad t\in\R,\label{eq:seriesExpansion}
\end{equation}
for coefficient processes $(a_{n})_{n\ge1}$ and random variables
$\zeta_{n}$, which are all defined on a joint probability space $(\Omega,\mathcal{F},\mathbb{P})$
with corresponding expectation operator $\mathbb{E}$. We will impose
the following assumptions: 
\begin{enumerate}
\item[(A)]  Let $(\zeta_{n})_{n\ge1}$ be a sequence of random variables satisfying
$\E[\zeta_{n}\zeta_{m}]=\ind_{\{n=m\}}$ and the following \emph{hypercontractivity}
property: For every $r\ge1$ there is a constant $C_{r}>0$ such that
for every polynomial $P\colon\R^{n}\to\R$ of degree $2$ we have
\[
\E\big[|P(\zeta_{1},\dots,\zeta_{n})|^{r}\big]\le C_{r}\E\big[|P(\zeta_{1},\dots,\zeta_{n})|^{2}\big]^{r/2}.
\]
\item[(B)]  Let $a_{n}\in\mathcal{B}_{p,1}^{\beta-1}$, $n\ge1$, for some $p\ge2$,
$\beta\in(0,1)$ such that $\sum_{n\ge1}\|a_{n}\|_{\beta-1,p,1}^{2}<\infty$.
\end{enumerate}
An important class of processes satisfying these assumptions are centred
Gaussian processes $\xi$ whose covariance operator can be represented
as an $L^{2}$-inner product, i.e., $\E[\xi_{s}\xi_{t}]=\langle f_{s},f_{t}\rangle$
for a class of functions $(f_{t})_{t\in\R}$. If we expand $f_{t}=\sum_{n}a_{n}(t)\psi_{n}$,
$a_{n}(t)=\langle f_{t},\psi_{n}\rangle$, with respect to some orthonormal
basis $(\psi_{n})$, we may obtain the representation (\ref{eq:seriesExpansion})
with i.i.d. standard normal $(\zeta_{n})$. Indeed, the distribution
of the finite dimensional distributions of the random series then
coincides with the original process by construction, such that in
general only tightness has additionally to be verified.
\begin{example}
~
\begin{enumerate}
\item Let $(B_{t})_{t\in[0,1]}$ be a Brownian motion. Its well-known Karhunen-Loève
expansion is given by
\[
B_{t}=\sqrt{2}\sum_{n=1}^{\infty}\frac{\sin\big((n-1/2)\pi t\big)}{(n-1/2)\pi}\zeta_{n},\qquad t\in[0,1],
\]
for i.i.d. $\zeta_{n}\sim\mathcal{N}(0,1)$. Using a periodic version
of Brownian motion, we may consider this series for all $t\in\R$.
Let $\xi=(\dd B)\chi$ be the distributional derivative multiplied
with a localising function $\chi\in L^{p}$. Then $\xi$ admits the
representation with (\ref{eq:seriesExpansion}) with $a_{n}(t)=\sqrt{2}\cos((n-1/2)\pi t)\chi(t)$.
Since $\|a_{n}\|_{\beta-1,p,1}$ is of the order $n^{\beta-1}$, Assumption
(B) is satisfied for all $\beta<1/2$.
\item \citet{Dzhaparidze2004} have proved the following series expansion
for the fractional Brownian motion $(X_{t})_{t\in[0,1]}$ with Hurst
index $H\in(0,1)$:
\[
X_{t}=\sum_{n=1}^{\infty}\frac{\sin(x_{n}t)}{x_{n}}\sigma_{n}\zeta_{n}+\sum_{n=1}^{\infty}\frac{1-\cos(y_{n}t)}{y_{n}}\tau_{n}\eta_{n},\qquad t\in[0,1],
\]
where $(\zeta_{n})_{n\ge1}$ and $(\eta_{n})_{n\ge1}$ are independent,
standard normal random variables, $x_{1}<x_{2}<\dots$ are the positive,
real zeros of the Bessel function $J_{-H}$ of the first kind of order
$-H$ and $y_{1}<y_{2}<\dots$ are the positive zeros of $J_{1-H}$.
Moreover, $\sigma_{n}^{2}=c_{H}x_{n}^{-2H}J_{1-H}^{-2}(x_{n})$ and
$\tau_{n}^{2}=c_{H}y_{n}^{-2H}J_{-H}^{-2}(y_{n})$ with some explicit
constant $c_{H}>0$ given in \citep{Dzhaparidze2004}.  $X_{t}$ can
be decomposed into a two-dimensional process with coordinates given
by the first and the second sum, respectively. As noise process $\xi$,
we again consider the localised derivative leading to (\ref{eq:seriesExpansion})
with $a_{n}(t)=(a_{n}^{(1)}(t),a_{n}^{(2)}(t))=(\sigma_{n}\cos(x_{n}t),\tau_{n}\sin(y_{n}))^{\top}\chi(t)$.
Noting the asymptotic expressions $\sigma_{n}^{2}\sim\tau_{n}^{2}\sim n^{1-2H}$
and $x_{n}\sim y_{n}\sim n$ for $n\to\infty$, cf. \citep{Dzhaparidze2004},
we obtain $\|a_{n}^{(1)}\|_{\beta-1,p,p}\sim\sigma_{n}x_{n}^{\beta-1}\sim n^{-1/2-H+\beta}$
and $\|a_{n}^{(2)}\|_{\beta-1,p,p}\sim n^{-1/2-H+\beta}$. We conclude
that Assumption (B) is fulfilled for $\beta<H$.
\end{enumerate}
\end{example}

Based on the assumption (A) and (B), we first verify the Besov regularity
of $\xi$.
\begin{lem}
\label{lem:regularity of gaussian process} Let $(\zeta_{n})_{n\ge1}$
and $(a_{n})_{n\ge1}$ fulfil Assumptions (A) and (B), respectively.
Then, there is a monotone integer valued sequence $(m_{n})\uparrow\infty$
such that the approximating sequence $\xi^{n}=\sum_{k=1}^{m_{n}}a_{k}(s)\zeta_{k}$
is almost surely a Cauchy sequence with respect to $\|\cdot\|_{\beta-1,p,\infty}$.
In particular, the almost sure and $L^{p}$-limit 
\[
\xi_{t}:=\lim_{n\to\infty}\xi_{t}^{n}=\sum_{n\ge1}a_{n}(t)\zeta_{n},\quad t\in\R,
\]
 is $\mathcal{B}_{p,\infty}^{\beta-1}$-regular. 
\end{lem}

\begin{proof}
We set $m_{0}=1$ and
\begin{equation}
m_{n}:=\inf\Big\{ K\ge m_{n-1}:\sum_{k=K+1}^{\infty}\|a_{k}\|_{\beta-1,p,p}^{2}\le n^{-6}\Big\},\qquad n\ge1.\label{eq:mn}
\end{equation}
It is sufficient to show 
\begin{equation}
\sum_{n\ge1}\P\Big(\|\xi^{n+1}-\xi^{n}\|_{\beta-1,p,\infty}>b_{n}\Big)<\infty\label{eq:borelCantelli}
\end{equation}
for some sequence $(b_{n})\in\ell^{1}$. Then, the Borel-Cantelli
Lemma yields that for almost every $\omega\in\Omega$ there is some
$n(\omega)\ge1$ such that $\|\xi^{m+1}-\xi^{m}\|_{\beta-1,p,\infty}\le b_{m}$
for all $m\ge n(\omega)$. Since $b_{m}$ is summable, $(\xi^{n})_{n\ge1}$
is almost surely a Cauchy sequence converging to $\xi\in\mathcal{B}_{p,\infty}^{\beta-1}$.
Moreover, it suffices to consider $p<\infty$ due to the embedding
$\mathcal{B}_{p,\infty}^{\beta-1}\subset\mathcal{B}_{\infty,\infty}^{\beta-1-1/p}$,
which is sufficient if $p$ is chosen large enough.

We now verify (\ref{eq:borelCantelli}). By definition we have
\begin{align*}
\|\xi^{n+1}-\xi^{n}\|_{\beta-1,p,\infty} & =\bigg\|\sum_{k=m_{n}+1}^{m_{n+1}}\zeta_{k}a_{k}\bigg\|_{\beta-1,p,\infty}\\
 & =\sup_{j\ge-1}\bigg(2^{(\beta-1)j}\bigg\|\sum_{k=m_{n}+1}^{m_{n+1}}\zeta_{k}(\Delta_{j}a_{k})\bigg\|_{L^{p}}\bigg)\\
 & =\sup_{j\ge-1}\bigg(2^{(\beta-1)jp}\int_{\R}\bigg|\sum_{k=m_{n}+1}^{m_{n+1}}\zeta_{k}(\Delta_{j}a_{k})(x)\bigg|^{p}\d x\Big)^{1/p}.
\end{align*}
Hence, using an union bound for the supremum and Markov's inequality
we have
\begin{align}
\P\Big(\|\xi^{n+1}-\xi^{n}\|_{\beta-1,p,\infty}>b_{n}\Big) & \leq b_{n}^{-p}\sum_{j\in\N}2^{(\beta-1)jp}\int_{\R}\mathbb{E}\bigg[\bigg|\sum_{k=m_{n}+1}^{m_{n+1}}\zeta_{k}(\Delta_{j}a_{k})(x)\bigg|^{p}\bigg]\d x.\label{eq:regBB-1}
\end{align}
Using the hypercontractivity and $\E[\zeta_{n}\zeta_{m}]=\ind_{\{n=m\}}$,
we obtain the upper bound 
\begin{align*}
 & b_{n}^{-p}\sum_{j\in\N}2^{(\beta-1)jp}\int_{\R}\mathbb{E}\bigg[\bigg(\sum_{k=m_{n}+1}^{m_{n+1}}\zeta_{k}(\Delta_{j}a_{k})(x)\bigg)^{2}\bigg]^{\frac{p}{2}}\d x\\
 & \qquad\quad=b_{n}^{-p}\sum_{j\in\N}2^{(\beta-1)jp}\int_{\R}\bigg(\sum_{k=m_{n}+1}^{m_{n+1}}(\Delta_{j}a_{k})^{2}(x)\bigg)^{\frac{p}{2}}\d x.
\end{align*}
We now use Hölder's inequality to obtain for any sequence $(c_{k})\in\ell^{1}$
\begin{align*}
\P\Big(\|\xi^{n+1}-\xi^{n}\|_{\beta-1,p,\infty}>b_{n}\Big)\lesssim & b_{n}^{-p}\sum_{j\in\N}2^{(\beta-1)jp}\Big(\sum_{k=m_{n}+1}^{m_{n+1}}c_{k}\Big)^{p/2-1}\sum_{k=m_{n}+1}^{m_{n+1}}c_{k}^{-(p/2-1)}\|\Delta_{j}a_{k}\|_{L^{p}}^{p}\\
= & b_{n}^{-p}\Big(\sum_{k=m_{n}+1}^{m_{n+1}}c_{k}\Big)^{p/2-1}\sum_{k=2^{n}+1}^{2^{n+1}}c_{k}^{-(p/2-1)}\|a_{k}\|_{\beta-1,p,p}^{p}\\
\le & b_{n}^{-p}\Big(\sum_{k=m_{n}+1}^{m_{n+1}}c_{k}\Big)^{p/2}\Big(\sup_{k}c_{k}^{-1/2}\|a_{k}\|_{\beta-1,p,p}\Big)^{p}.
\end{align*}
Choosing $c_{k}:=\|a_{k}\|_{\beta-1,p,1}^{2}\ge\|a_{k}\|_{\beta-1,p,p}^{2}$,
it remains to note that $d_{n}:=(\sum_{k=m_{n}+1}^{m_{n+1}}c_{k})^{1/2}\le n^{-3}$
by the choice of $m_{n}$, such that we may choose $b_{n}=n^{-3/2}$.
\end{proof}
Young's inequality (Lemma~\ref{lem:Convolution}) yields automatically
$\phi\ast\xi\in\mathcal{B}_{p,\infty}^{\gamma+\beta-1}$ for $\phi\in\mathcal{B}_{1,\infty}^{\gamma}$.
With these preperations we can verify the existence of a limit $\lim_{n\to\infty}\pi(\phi\ast\xi^{n},\xi^{n})=:\pi(\phi\ast\xi,\xi)\in\mathcal{B}_{p/2,\infty}^{2\beta+\gamma-2}$. 
\begin{thm}
\label{thm: contruction resonant term} Let $(\zeta_{n})_{n\ge1}$
and $(a_{n})_{n\ge1}$ fulfil Assumptions (A) and (B), $p\ge4$ and
$\gamma>0$. Further, suppose that $(\zeta_{n})$ are independent,
and $\phi\in\mathcal{B}_{1,\infty}^{\gamma}$. Set $\xi^{n}:=\sum_{k=1}^{m_{n}}a_{k}\zeta_{k}$
for a sufficiently fast growing integer valued sequence $(m_{n})\uparrow\infty$.
Then $(\pi(\phi\ast\xi^{n},\xi^{n}))_{n\ge1}$ is almost surely a
Cauchy sequence with respect to $\|\cdot\|_{2\beta+\gamma-2,p/2,\infty}$
with almost sure and $L^{p/2}$-limit 
\[
\pi(\phi\ast\xi,\xi):=\lim_{n\to\infty}\pi(\phi\ast\xi^{n},\xi^{n})\in\mathcal{B}_{p/2,\infty}^{2\beta+\gamma-2}.
\]
\end{thm}

\begin{proof}
Let $(m_{n})_{n\ge0}$ be as in (\ref{eq:mn}). As in the Lemma~\ref{lem:regularity of gaussian process}
thanks to the Borel-Cantelli Lemma it suffices to prove for some sequence
$(b_{n})\in\ell^{1}$ and finite $p\in[1,\infty)$:
\[
\sum_{n\ge1}\P\Big(\|\pi(\phi\ast\xi^{n+1},\xi^{n+1})-\pi(\phi\ast\xi^{n},\xi^{n})\|_{2\beta+\gamma-2,p/2,\infty}>b_{n}\Big)<\infty.
\]
 Defining $\Delta_{k}^{\phi}f:=\Delta_{k}(\phi\ast f)=\F^{-1}[\rho_{j}\F\phi]\ast f$
for distributions $f$, we have
\begin{align*}
 & \pi(\phi\ast\xi^{n+1},\xi^{n+1})-\pi(\phi\ast\xi^{n},\xi^{n})\\
 & =\sum_{j\ge1}\sum_{k=j-1}^{j+1}\Delta_{j}\big(\xi^{n+1}-\xi^{n}\big)\Delta_{k}^{\phi}\xi^{n+1}+\sum_{j\ge1}\sum_{k=j-1}^{j+1}\Delta_{j}\xi^{n}\Delta_{k}^{\phi}\big(\xi^{n+1}-\xi^{n}\big)\\
 & =:T_{n,1}+T_{n,2}.
\end{align*}
Since both terms can be estimated analogously, we focus on $T_{n,1}$,
for which we have
\begin{align*}
T_{n,1} & =\sum_{j\ge1}\sum_{k=j-1}^{j+1}\Delta_{j}\Big(\sum_{m=m_{n}+1}^{m_{n+1}}\zeta_{m}a_{m}\Big)\cdot\Delta_{k}^{\phi}\Big(\sum_{m=1}^{m_{n+1}}\zeta_{k}a_{m}\Big)\\
 & =\sum_{j\ge1}\sum_{k=j-1}^{j+1}\sum_{m=m_{n}+1}^{m_{n+1}}\sum_{m'=1}^{m_{n+1}}\zeta_{m}\zeta_{m'}(\Delta_{j}a_{m})(\Delta_{k}^{\phi}a_{m'}).
\end{align*}
Hence, we get
\begin{align*}
 & \|T_{n,1}\|_{2\beta+\gamma-2,p/2,\infty}\\
 & \qquad=\sup_{j}\big(2^{(2\beta+\gamma-2)j}\|\Delta_{j}T_{1}\|_{L^{p/2}}\big)\\
 & \qquad\le\sup_{j}\Big(2^{(2\beta+\gamma-2)j}\sum_{j'\sim j}\Big\|\sum_{k=j'-1}^{j'+1}\sum_{m=m_{n}+1}^{m_{n+1}}\sum_{m'=1}^{m_{n+1}}\zeta_{m}\zeta_{m'}(\Delta_{j}a_{m})(\Delta_{k}^{\phi}a_{m'})\Big\|_{L^{p/2}}\Big)\\
 & \qquad\lesssim\sup_{j}\bigg(2^{(2\beta+\gamma-2)j}\Big\|\sum_{k=j-1}^{j+1}\sum_{m=m_{n}+1}^{m_{n+1}}\sum_{m'=1}^{m_{n+1}}\zeta_{m}\zeta_{m'}(\Delta_{j}a_{m})(\Delta_{k}^{\phi}a_{m'})\Big\|_{L^{p/2}}\bigg).
\end{align*}
As above, Markov's inequality and the hypercontractivity yield
\begin{align*}
 & \P\big(\|T_{n,1}\|_{2\beta+\gamma-2,p/2,\infty}>b_{n}\big)\\
 & \quad\lesssim b_{n}^{-p/2}\sum_{j}2^{(2\beta+\gamma-2)jp/2}\int_{\R}\E\Big[\Big|\sum_{k=j-1}^{j+1}\,\sum_{m=m_{n}+1}^{m_{n+1}}\,\sum_{m'=1}^{m_{n+1}}\zeta_{m}\zeta_{m'}(\Delta_{j}a_{m})(x)(\Delta_{k}^{\phi}a_{m'})(x)\Big|^{p/2}\Big]\,\dd x\\
 & \quad\lesssim b_{n}^{-p/2}\sum_{j}2^{(2\beta+\gamma-2)jp/2}\int_{\R}\E\Big[\Big|\sum_{k=j-1}^{j+1}\,\sum_{m=m_{n}+1}^{m_{n+1}}\,\sum_{m'=1}^{m_{n+1}}\zeta_{m}\zeta_{m'}(\Delta_{j}a_{m})(x)(\Delta_{k}^{\phi}a_{m'})(x)\Big|^{2}\Big]^{\frac{p}{4}}\,\dd x\\
 & \quad\lesssim b_{n}^{-p/2}\sum_{j}2^{(2\beta+\gamma-2)jp/2}\int_{\R}\Big(\sum_{k_{1},k_{2}=j-1}^{j+2}\,\sum_{m_{1},m_{2}=m_{n}+1}^{m_{n+1}}\,\sum_{m_{1}',m_{2}'=1}^{m_{n+1}}\E\big[\zeta_{m_{1}}\zeta_{m_{1}'}\zeta_{m_{2}}\zeta_{m_{2}'}\big]\\
 & \qquad\qquad\times(\Delta_{j}a_{m_{1}})(x)(\Delta_{k_{1}}^{\phi}a_{m_{1}'})(x)(\Delta_{j}a_{m_{2}})(x)(\Delta_{k_{2}}^{\phi}a_{m_{2}'}(x)\Big)^{\frac{p}{4}}\,\dd x.
\end{align*}
In the previous sum it suffices the consider the terms where $\{m_{1}=m_{2},m_{1}'=m_{2}'\}$,
$\{m_{1}=m_{1}',m_{2}=m_{2}^{'}\}$ (being equivalent to $\{m_{1}=m_{2}',m_{2}=m_{1}^{'}\}$)
and $\{m_{1}=m_{2}=m_{1}'=m_{2}'\}$, because in all other cases $\E\big[\zeta_{m_{1}}\zeta_{m_{1}'}\zeta_{m_{2}}\zeta_{m_{2}'}\big]$
is zero by independence of the $(\zeta_{m})$. Since all partial sums
can be bounded similarly, we consider only $\{m_{1}=m_{2},m_{1}'=m_{2}'\}$
for brevity. This partial sum is given by
\begin{align*}
S_{n}:= & b_{n}^{-p/2}\sum_{j}\bigg(2^{(2\beta+\gamma-2)jp/2}\\
 & \qquad\sum_{j-1\le k_{1},k_{2}\le j+1}\int_{\R}\Big(\sum_{m=m_{n}+1}^{m_{n+1}}\,\sum_{m'=1}^{m_{n+1}}(\Delta_{j}a_{m})^{2}(x)(\Delta_{k_{1}}^{\phi}a_{m'})(x)(\Delta_{k_{2}}^{\phi}a_{m'})(x)\Big)^{\frac{p}{4}}\,\dd x\bigg)\\
= & b_{n}^{-p/2}\sum_{j}\bigg(2^{(2\beta+\gamma-2)jp/2}\\
 & \qquad\sum_{j-1\le k_{1},k_{2}\le j+1}\int_{\R}\Big(\sum_{m=m_{n}+1}^{m_{n+1}}(\Delta_{j}a_{m})^{2}(x)\Big)^{\frac{p}{4}}\Big(\sum_{m'=1}^{m_{n+1}}(\Delta_{k_{1}}^{\phi}a_{m'})(x)(\Delta_{k_{2}}^{\phi}a_{m'})(x)\Big)^{\frac{p}{4}}\,\dd x\bigg).
\end{align*}
Hölder's inequality yields for the $\ell^{1}$-sequence $c_{k}:=\|a_{k}\|_{\beta-1,p,1}^{2},k\ge1,$
\begin{align*}
S_{n} & \le\frac{1}{b_{n}^{p/2}}\sum_{j}2^{(2\beta+\gamma-2)jp/2}\sum_{j-1\le k_{1},k_{2}\le j+1}\int_{\R}\Big(\sum_{m=m_{n}+1}^{m_{n+1}}c_{m}\Big)^{\frac{p}{4}-1}\Big(\sum_{m=m_{n}+1}^{m_{n+1}}c_{m}^{-(\frac{p}{4}-1)}(\Delta_{j}a_{m})^{\frac{p}{2}}(x)\Big)\\
 & \quad\times\Big(\sum_{m'=1}^{m_{n+1}}c_{m'}\Big)^{\frac{p}{4}-1}\Big(\sum_{m'=1}^{m_{n+1}}c_{m'}^{-(\frac{p}{4}-1)}(\Delta_{k_{1}}^{\phi}a_{m'})^{\frac{p}{4}}(x)(\Delta_{k_{2}}^{\phi}a_{m'})^{\frac{p}{4}}(x)\Big)\,\dd x\\
 & \le\|c_{m}\|_{\ell^{1}}^{\frac{p}{4}-1}b_{n}^{-p/2}\Big(\sum_{m=2^{n}+1}^{2^{n+1}}c_{m}\Big)^{\frac{p}{4}-1}\sum_{j}2^{(2\beta+\gamma-2)jp/2}\\
 & \quad\sum_{j-1\le k_{1},k_{2}\le j+1}\sum_{m=m_{n}+1}^{m_{n+1}}c_{m}^{-\frac{p}{4}+1}\sum_{m'=1}^{m_{n+1}}c_{m'}^{-\frac{p}{4}+1}\int_{\R}(\Delta_{j}a_{m})^{\frac{p}{2}}(x)(\Delta_{k_{1}}^{\phi}a_{m'})^{\frac{p}{4}}(x)(\Delta_{k_{2}}^{\phi}a_{m'})^{\frac{p}{4}}(x)\,\dd x.
\end{align*}
Writing $d_{n}:=(\sum_{k=m_{n}+1}^{m_{n+1}}c_{k})^{1/2}\in\ell^{1}$
and applying once again Hölder's inequality, we obtain
\begin{align*}
S_{n} & \lesssim\|c_{m}\|_{\ell^{1}}^{\frac{p}{4}-1}b_{n}^{-p/2}d_{n}^{p/2-2}\sum_{j}\,\sum_{j-1\le k_{1},k_{2}\le j+1}\sum_{m=m_{n}+1}^{m_{n+1}}c_{m}^{-(p/4-1)}2^{(\beta-1)jp/2}\|\Delta_{j}a_{m}\|_{L^{p}}^{p/2}\\
 & \qquad\sum_{m'=1}^{2^{n+1}}c_{m'}^{-(p/4-1)}2^{(\beta-1+\gamma)k_{1}p/4}\|\Delta_{k_{1}}^{\phi}a_{m'}\|_{L^{p}}^{p/4}2^{(\beta-1+\gamma)k_{2}p/4}\|\Delta_{k_{2}}^{\phi}a_{m'}\|_{L^{p}}^{p/4}\\
 & \lesssim\|c_{m}\|_{\ell^{1}}^{\frac{p}{4}-1}b_{n}^{-p/2}d_{n}^{p/2-2}\sum_{m=2^{n}+1}^{2^{n+1}}c_{m}^{-(p/4-1)}\|a_{m}\|_{\beta-1,p,p/2}^{p/2}\sum_{m'=1}^{2^{n+1}}c_{m'}^{-(\frac{p}{4}-1)}\|\phi\ast a_{m'}\|_{\beta-1+\gamma,p,p/4}^{p/2}\\
 & \le\|c_{m}\|_{\ell^{1}}^{\frac{p}{4}}(d_{n}/b_{n})^{p/2}\Big(\sup_{m'}c_{m'}^{-1/2}\|\phi\ast a_{m'}\|_{\beta-1+\gamma,p,p/4}\Big)^{p/2}.
\end{align*}
With $\|\phi\ast a_{m'}\|_{\beta-1+\gamma,p,p/4}\lesssim\|\phi\|_{\gamma,1,1}\|a_{m'}\|_{\beta-1,p,p/4}$
by Young's inequality, we conclude $S_{n}\lesssim(d_{n}/b_{n})^{p/2}\|\phi\|_{\gamma,1,1}^{p/2}.$
Since $d_{n}\lesssim n^{-3}$, we deduce $\sum_{n\ge1}S_{n}<\infty$
for $b_{n}=n^{-3/2}$\textcolor{blue}{.}
\end{proof}

\begin{rem}
For the special case where $\phi\ast\xi$ is replaced by the antiderivative
of $\xi$, alternative constructions of rough path and iterated integrals
above stochastic processes defined by random Fourier or Schauder expansions
were considered in \citep{Friz2016,Gubinelli2016,Promel2015}. 
\end{rem}

\section{Application to rough and stochastic differential equations\label{sec:examples}}

The general existence and uniqueness results for solutions to Volterra
equations of the form~(\ref{eq:convol}) provided in Section~\ref{sec:convolution}
allow to recover well-known results in the paracontrolled distribution
setting but additionally contain many novel results concerning differential
equations driven by stochastic processes or convolutional rough paths.
In the following we discuss some exemplary stochastic equations and
explicitly state the particular existence and uniqueness results. 

\subsection{Stochastic and rough differential equations with possible delay\label{subsec: rough differential equation}}

Ordinary stochastic differential equations and their pathwise counterparts
given by rough differential equations constitute fundamental and well
studied objects in stochastic analysis. These differential equations
can typically written in their integral form 
\begin{equation}
u(t)=u_{0}+\int_{0}^{t}\sigma_{1}(u(s-r_{1}))\d\theta(s)+\int_{0}^{t}\sigma_{2}(u(s-r_{2}))\d s,\quad t\in[0,T],\label{eq:rde}
\end{equation}
where $\theta$ is a suitable driving signal, e.g., a (fractional)
Brownian motion or a rough path, and $r_{1},r_{2}\geq0$ are constant
delay parameters. Thanks to the general regularity assumptions required
on the kernel functions in Section~\ref{sec:convolution}, the differential
equation~(\ref{eq:rde}) can be viewed as a special case of the Volterra
equation~(\ref{eq:convol}), and we can recover for instance the
following results. For this purpose, we denote by $\dot{\theta}$
the distributional derivative of $\theta\in\mathcal{B}_{p,\infty}^{\beta}$
and introduce a kernel function $\phi_{T}$ which is assumed to be
compactly supported on $[0,2T]$, smooth on $\R\setminus\{0\}$ and
satisfying $\phi_{T}(t)=\ind_{[0,\infty)}(t)$ for all $t\in[-T,T]$.
\begin{cor}
\label{cor:sde classical} Let $u_{0}\in\R$, $\sigma_{1}\in C^{3}$,
$\sigma_{2}\in C^{2}$ with $\sigma_{1}(0)=\sigma_{2}(0)=0$ and $r_{1},r_{2}\geq0$.
\begin{enumerate}
\item If $\theta$ is an $n$-dimensional fractional Brownian motion with
Hurst index $H>1/2$, then there exists a unique solution to the stochastic
differential equation~(\ref{eq:rde}) up to a random time~$T$.
\item If $(\dot{\theta},\pi(\phi_{T}*(\dot{\theta}\ind_{[0,T]}),\dot{\theta}(\cdot+r_{1})\ind_{[0,T]}(\cdot+r_{1})))\in\mathcal{B}_{p}^{\beta,1}(\phi_{T})$
for $\beta>1/3$ and $p\in[3,\infty)$, then there exists a unique
solution to the rough differential equation~(\ref{eq:rde}) up to
a random time~$T$. 
\end{enumerate}
\end{cor}

\begin{proof}
Let $\chi$ be a smooth and compactly support function with $\chi(t)=1$
for $t\in[0,T]$. Equation~(\ref{eq:rde}) coincides on the interval
$[0,T]$ with
\begin{align*}
u(t) & =u_{0}\chi(t)+\int_{\R}\phi_{T}(t-s-r_{1})\sigma_{1}(u(s))\dot{\Xi}(s)\d s+\int_{\R}\phi_{T}(t-s-r_{2})\sigma_{2}(u(s))\ind_{[0,T]}(s+r_{2})\d s
\end{align*}
for $t\in\R$ and driving signal $\dot{\Xi}(\cdot):=\dot{\theta}(\cdot+r_{1})\ind_{[0,T]}(\cdot+r_{1})$.
Note that $\phi_{T}(\cdot-r_{1})\in\mathcal{B}_{1,\infty}^{1}$ and
$(\cdot-r_{1})\phi_{T}(\cdot-r_{1})\in\mathcal{B}_{1,\infty}^{2}$.
Hence, (i) and (ii) follow by applying Theorem~\ref{thm:localSol}
and recalling that a fractional Brownian motion with Hurst index $H>1/2$
has almost surely $(H-\epsilon)$-Hölder continuous sample paths for
every $\epsilon>0$.
\end{proof}
Existence and uniqueness results for stochastic delay equations like~(\ref{eq:rde})
driven by a fractional Brownian motion with Hurst index $H>1/2$ were
first obtained by \citet{Ferrante2006}. Differential equations driven
by $\alpha$-Hölder continuous rough paths with $\alpha\in(1/3,1/2)$
and constant delay were first treated by \citet{Neuenkirch2008}.
Rough differential equations without delay but in the paracontrolled
distribution setting were considered in \citep{Gubinelli2015} and
\citep{Promel2015}. Furthermore, we would like to point out that
Corollary~\ref{cor:sde classical}~(ii) can be applied to a fractional
Brownian motion with Hurst index $H\in(1/3,1/2)$ due to Theorem~\ref{thm: contruction resonant term}.
\begin{rem}
For stochastic and rough differential equations like~(\ref{eq:rde}),
it is straightforward to obtain a solution on any arbitrary large
interval~$[0,T]$ applying iteratively Corollary~\ref{cor:sde classical}
on small intervals and glueing the so obtained local solutions together.
\end{rem}

\subsection{Stochastic and rough Volterra equations }

Stochastic integral equations of Volterra type appear in various areas
of mathematical modelling such as in physics or mathematical finance
and the treatment of such Volterra equations involving stochastic
integration goes back to the pioneering works of \citet{Berger1980a,Berger1980b}.
The pathwise counterparts of stochastic Volterra equations, namely,
Volterra equations driven by rough paths were first considered by
\citet{Deya2009,Deya2011}. More precisely, we consider Volterra equations
of convolution type 
\begin{equation}
u(t)=u_{0}(t)+\int_{0}^{t}\psi_{1}(t-s)\sigma_{1}\big(u(s)\big)\dot{\theta}(s)\d s+\int_{0}^{t}\psi_{2}(t-s)\sigma_{2}\big(u(s)\big)\d s,\label{eq:Volterra equation}
\end{equation}
for $t\in[0,T]$ and $\dot{\theta}$ denotes again the distributional
derivative of the path~$\theta\in\mathcal{B}_{p,\infty}^{\beta}$.
\begin{cor}
\label{cor:volterra sde} Let $p\in[3,\infty]$ and $\beta\in(1/3,1/2)$.
Suppose that $\psi_{1},\psi_{2}\in\mathcal{B}_{\infty,\infty}^{1}$,
$u_{0}\in\mathcal{B}_{p,\infty}^{2\beta}$ and $\sigma_{1}\in C^{3}$,
$\sigma_{2}\in C^{2}$ with $\sigma_{1}(0)=\sigma_{2}(0)=0$. If $(\dot{\theta},\pi((\phi_{T}\psi_{1})*(\dot{\theta}),\dot{\theta})\in\mathcal{B}_{p}^{\beta,1}(\phi_{T}\psi_{1})$
and $T>0$ is sufficiently small, then there exists $u\in\mathcal{B}_{p,\infty}^{\beta}$
which is the unique solution to the rough Volterra equation~(\ref{eq:Volterra equation})
on $[0,T]$.
\end{cor}

\begin{proof}
We first observe that $\psi_{i}\phi_{T}\in\mathcal{B}_{1,\infty}^{1}$
because $\psi_{i}\in\mathcal{B}_{\infty,\infty}^{1}$ and $\phi_{T}\in\mathcal{B}_{p,\infty}^{1/p}$,
for $i=1,2$. Moreover, the rough Volterra equation~(\ref{eq:Volterra equation})
coincides on the interval $[0,T]$ with
\begin{align*}
u(t) & =u_{0}(t)+\int_{\R}\phi_{T}(t-s)\psi_{1}(t-s)\sigma_{1}(u(s))\dot{\theta}(s)\d s\\
 & \qquad+\int_{\R}\phi_{T}(t-s)\psi_{1}(t-s)\sigma_{2}(u(s))\d s,\quad t\in\R.
\end{align*}
Therefore, Theorem~\ref{thm:localSol} and Remark~\ref{rem:local soluton}
imply the assertion. 
\end{proof}
\begin{rem}
Assuming $\psi_{1}(0)\neq0$, it is not necessary to include the kernel
function $\psi_{1}$ in the definition of the driving rough path.
Indeed, one can take a generic rough path, i.e., independent of $\psi_{1}$,
thanks to Lemma~\ref{lem:roughpath}. Furthermore, notice that the
kernel $\psi_{1}$ has only to be Lipschitz continuous. This Lipschitz
assumption is a significant relaxation compared to the $C^{3}$-regularity
of the kernel functions so far required for Volterra equations of
convolutional type driven by rough paths, see \citet{Deya2009,Deya2011}.
\end{rem}

The previous pathwise existence and uniqueness result for Volterra
equations can immediately be applied to a wide class of stochastic
processes thanks to Theorem~\ref{thm: contruction resonant term}.
\begin{cor}
Let $\theta$ be a stochastic process such that $\dot{\theta}$ is
of the form~(\ref{eq:seriesExpansion}) satisfying Assumption (A)
and (B) for $\beta\in(1/3,1)$ and $p\geq3$. Suppose that $\psi_{1},\psi_{2}\in\mathcal{B}_{\infty,\infty}^{1}$,
$u_{0}\in\mathcal{B}_{p,\infty}^{2\beta}$ and $\sigma_{1}\in C^{3}$,
$\sigma_{2}\in C^{2}$ with $\sigma_{1}(0)=\sigma_{2}(0)=0$. Then,
there exists $u\in\mathcal{B}_{p,\infty}^{\beta}$ which is the unique
solution of the stochastic Volterra equation~(\ref{eq:Volterra equation})
up to a random time~$T$.
\end{cor}

\subsection{SDEs with fractional derivatives\label{sec:fractSDEs}}

Stochastic Volterra equations with singular kernels are of particular
interest because of their applications to stochastic partial differential
equations (e.g. \citep{Zhang2010}) and stochastic differential equations
with fractional derivatives (e.g. \citep{Wang2008}), but also because
of recent developments in mathematical finance showing that Volterra
equations with singular kernels serve as very suitable models for
the probabilistic and irregular behaviour of volatility in financial
markets, see e.g. \citep{ElEuch2016}. 

In order to consider SDEs allowing for fractional derivatives, let
us recall the definition of the Riemann-Liouville fractional integral
operator (with base point 0), which is given by 
\[
I^{r}(f)(t):=\frac{1}{\Gamma(r)}\big((s^{r-1}\ind_{(0,\infty)}(s))*f\big)(t)=\frac{1}{\Gamma(r)}\int_{0}^{t}(t-s)^{r-1}f(s)\d s
\]
for $r\in(0,1)$, $f$ a suitable function and the Gamma function
\[
\Gamma(r):=\int_{0}^{\infty}t^{r-1}e^{-t}\d t,\qquad r>0.
\]
The corresponding fractional derivative operator is defined by $D^{r}f:=\frac{\dd}{\dd t}I^{1-r}(f)$.
While there are many different fractional derivative operators, the
Riemann-Liouville derivative can be considered as a natural extension
of the classical derivative to fractional order. A (fairly simple)
stochastic differential equation of fractional order $r\in(0,1)$
driven by a Brownian motion is
\[
D^{r}u(t)=\sigma(u(t))\d W(t),\quad u(0)=u_{0},
\]
or equivalently expressed as a Volterra integral equation with singular
kernel
\begin{align}
u(t) & =u_{0}+\frac{1}{\Gamma(r)}\int_{0}^{t}(t-s)^{r-1}\sigma(u(s))\dot{W}(s)\d s,\quad t\in[0,T],\label{eq:sde fractional derivatives}
\end{align}
where $\dot{W}$ is the distributional derivative of a Brownian motion~$W$.
For a more general treatment of fractional stochastic differential
equations driven by Brownian motion, we refer for instance to \citet{Lototsky2018}.
Based on the results provided in Section~\ref{sec:convolution},
we obtain the following existence and uniqueness statement. 
\begin{cor}
Let $W$ be an $n$-dimensional Brownian motion and $r>5/6$. Suppose
that $u_{0}\in\R$ and $\sigma\in C^{3}$ with $\sigma(0)=0$. Then,
there exists $u\in\mathcal{B}_{3,\infty}^{\alpha}$ for any $\alpha<r-1/2$,
which is the unique solution to the stochastic Volterra equation~(\ref{eq:sde fractional derivatives})
up to a random time~$T$.
\end{cor}

\begin{proof}
The proof works as the proof of Corollary~\ref{cor:volterra sde}
combined with the observations that the localised kernel function
$\phi(x):=x^{r-1}\phi_{T}$ satisfies $\phi\in\mathcal{B}_{1,\infty}^{\gamma}$
for every $\gamma<r$ and that the sample paths of a Brownian motion
can be considered as convolutional rough paths with regularity $\beta<1/2$
due to Theorem~\ref{thm: contruction resonant term}. 
\end{proof}

\subsection{SDEs with additive Lévy noise}

Stochastic differential equations with an additive Lévy noise constitute
appropriate models for dynamical systems which are subject to external
shocks. Examples of such systems naturally appear in insurance mathematics,
where for instance SDEs with long term memory and additive Lévy noise
are used to model the general reserve process of an insurance company,
cf. \citet{Rolski1999}. More precisely, we consider the stochastic
differential equation
\begin{equation}
u(t)=u_{0}+\int_{0}^{t}\sigma_{1}(u(s))\d s+\int_{0}^{t}\sigma(u(s))\d\theta(s)+L(t),\quad t\in[0,T],\label{eq:SDE with Levy noise}
\end{equation}
where $\theta$ is a fractional Brownian motion and $L$ is a Lévy
process. This type of stochastic differential equations were recently
investigated, e.g., in \citet{Bai2015}.
\begin{cor}
Let $L$ be an $n$-dimensional Lévy process and $\theta$ be a fractional
Brownian motion with Hurst index $H>1/2$. Let $p\in[2,\infty]$ and
$\beta\in(1/2,1)$. Suppose that $p>2$, $u_{0}\in\R$ and $\sigma_{1},\sigma_{2}\in C^{2}$
with $\sigma_{1}(0)=\sigma_{2}(0)=0$. Then, there exists $u\in\mathcal{B}_{p,\infty}^{1/p}\cap L^{\infty}$
which is the unique solution of the stochastic differential equation~(\ref{eq:SDE with Levy noise})
up to a random time~$T$.
\end{cor}

\begin{proof}
As in the proof of Corollary~\ref{cor:sde classical} one can reformulate
the SDE~(\ref{eq:SDE with Levy noise}) as a Volterra equation which
coincides with (\ref{eq:SDE with Levy noise}) on the interval~$[0,T]$.
Furthermore, let us recall that the sample paths of a fractional Brownian
motion and of a Lévy process are almost surely in $\mathcal{B}_{p,\infty}^{\beta}$
and $\mathcal{B}_{p,\infty}^{1/p}$ for every $\beta<H$ and $p>2$,
respectively, see for instance \citep[Proposition~2]{Rosenbaum2009}
and \citep[Proposition~5.31]{Bottcher2013}. Hence, we deduce the
assertion from Proposition~\ref{prop:YoungLip with jumps} in combination
with a scaling argument analogously the proof of Theorem~\ref{thm:localSol}
and Remark~\ref{rem:local soluton}.
\end{proof}

\subsection{Stochastic moving average processes driven by Lévy processes}

Moving average processes driven by Lévy processes and in particular
shot noise processes provide a modern toolbox for mathematical modelling
of, e.g., turbulence, signal processing or shot prices on energy markets,
see \citep{Barndorff-Nielsen2013,Barndorff-Nielsen2015} and the references
therein. Allowing these types of models to possess a state dependent
volatility, we consider the stochastic convolution equation
\begin{equation}
u(t)=u_{0}+\int_{\R}\psi(t-s)\sigma(u(s))\d L(s),\quad t\in[0,T],\label{eq:moving average equation}
\end{equation}
where $L$ is a general Lévy process. Because of the desired averaging
property generated by the kernel function, it is naturally to postulate
the assumption of $\psi\in\mathcal{B}_{1,\infty}^{\gamma}$ for $\gamma>1$.
In this case we arrive at the following existence and uniqueness result. 
\begin{cor}
Let $L$ be an $n$-dimensional Lévy process, $p\in(2,\infty]$, $\gamma>1$
and $\alpha=1/p+\gamma-1$. Suppose that $\psi\in\mathcal{B}_{1,\infty}^{\gamma}$
has compact support, $u_{0}\in\R$ and $\sigma\in C^{2}$ with $\sigma(0)=0$.
Then, there exists $u\in\mathcal{B}_{p,\infty}^{\alpha}$ which is
the unique solution of the stochastic convolution equations~(\ref{eq:moving average equation})
up to a random time~$T$.
\end{cor}

\begin{proof}
Due to the compact support assumption of $\psi$, we can localise
the equation~(\ref{eq:moving average equation}) such that we obtain
a (localised) Volterra equation which coincides with (\ref{eq:moving average equation})
on the interval~$[0,T]$. Since the sample paths of a Lévy process
are almost surely in $\mathcal{B}_{p,\infty}^{1/p}$ for every $p>2$,
see again \citep[Proposition~5.31]{Bottcher2013}, we conclude the
assertion from Proposition~\ref{prop:YoungLip} in combination with
a scaling argument analogously the proof of Theorem~\ref{thm:localSol}
and Remark~\ref{rem:local soluton}.
\end{proof}

\subsection{Relation to stochastic PDEs}

In general, stochastic Volterra equations are known to have many links
to stochastic partial differential equations. Here we would like to
discuss this link in the case of (a slightly modified version of)
a stochastic evolution equation studied by \citet{Mytnik2015}. We
consider the differential operator $\Delta_{\theta}:=\partial_{x}x^{\theta}\partial_{x}$
(in one space dimension) for a parameter $\theta<2$ and the associated
evolution equation 
\begin{align}
\partial_{t}u(t,x) & =\Delta_{\theta}u(t,x)+\sigma\big(u(t,x)\big)\,\xi(\dd t,\dd x),\label{eq:specialSPDE}\\
u(0,x) & =g(x),\nonumber 
\end{align}
with multiplicative noise, where $\xi$ is the space-time derivative
of $\theta(t,x)=W_{t}\mathbf{1}_{[\eta,\infty)}(x)$ for some $\eta\in\mathbb{R}$,
that is
\[
\xi(\dd t,\dd x)=\dot{W}(\dd t)\,\delta_{\eta}(\dd x).
\]
with Dirac measure $\delta_{\eta}$ in $\eta\in\R$. Note that we
recover the stochastic heat equation with multiplicative noise in
the case $\theta=0$ and the fundamental solution of (\ref{eq:specialSPDE})
with $\xi=0$ is 
\[
p_{t}(x)=\frac{c_{\theta}}{t^{1/(2-\theta)}}\exp\Big(-\frac{x^{2-\theta}}{(2-\theta)^{2}t}\Big)
\]
with normalising constant $c_{\theta}$ such that a mild solution
of (\ref{eq:specialSPDE}) is given by the formula
\begin{align*}
u(t,x) & =\int_{\mathbb{R}}p(t,x-y)g(y)\d y+\int_{0}^{t}\int_{\mathbb{R}}p(t-s,x-y)\sigma\big(u(t,y)\big)\,\xi(\dd s,\dd y)\\
 & =\int_{\mathbb{R}}p(t,x-y)g(y)\d y+\int_{0}^{t}p(t-s,x-\eta)\sigma\big(u(t,\eta)\big)\,\dot{W}(\dd s).
\end{align*}
In particular, the solution process $v(t):=u(t,\eta)$ along the edge
$\{(t,\eta):t\in\R_{+}\}$ solves the singular stochastic Volterra
equation
\begin{align*}
v(t) & =\int_{\mathbb{R}}p(t,\eta-y)g(y)\d y+\int_{0}^{t}p(t-s,0)\sigma\big(v(t)\big)\,\dot{W}(\dd s)\\
 & =\int_{\mathbb{R}}p(t,\eta-y)g(y)\d y+\int_{0}^{t}\frac{c_{\theta}}{(t-s)^{1/(2-\theta)}}\sigma\big(v(t)\big)\,\dot{W}(\dd s).
\end{align*}
For $\theta<-4$ Theorem~\ref{thm:solve} provides the existence
of the pathwise solution process $v(t)$. In the case of the Laplace
operator, i.e. $\theta=0$, the singularity in the kernel is too severe
to directly apply Theorem~\ref{thm:solve} and would require a further
extension of the above theory.  

\appendix

\section{Auxiliary Besov estimates\label{sec:appendix}}

The appendix provides (in the previous sections) frequently used,
but fairly elementary lemmas concerning Besov spaces. The first one
states the invariance of Besov norms under linear shifts. 
\begin{lem}
\label{lem:shiftBesovNorm} Let $\alpha\in\R$, $p\in[1,\infty]$
and $y\in\R^{d}$. If $f\in\mathcal{B}_{p,\infty}^{\alpha}$, then
$f(\cdot+y)\in\mathcal{B}_{p,\infty}^{\alpha}$ with 
\[
\|f\|_{\alpha,p,\infty}=\|f(\cdot+y)\|_{\alpha,p,\infty}.
\]
\end{lem}

\begin{proof}
For $y\in\R^{d}$ and $f\in\mathcal{B}_{p,\infty}^{\alpha}$, note
that 
\[
\mathcal{F}f(\cdot+y)(z)=\int_{\mathbb{R}^{d}}e^{i\langle z,x-y\rangle}f(x)\d x=\mathcal{F}f(z)e^{i\langle z,y\rangle},\quad z\in\R^{d},
\]
from which we deduce that 
\[
\Delta_{j}f(\cdot+y)(z)=\mathcal{F}^{-1}(\rho_{j}e^{-i\langle\cdot,y\rangle}\mathcal{F}f)(z)=\mathcal{F}^{-1}(\rho_{j}\mathcal{F}f)(z+y).
\]
Therefore, $\|\Delta_{j}f(\cdot+y)\|_{L^{p}}=\|\Delta_{j}f\|_{L^{p}}$
for each $j\geq-1$ and thus $\|f\|_{\alpha,p,\infty}=\|f(\cdot+y)\|_{\alpha,p,\infty}$.
\end{proof}
For sufficiently regular distributions/functions the Besov norm of
a product can be directly estimated and in particular the product
is then a well-defined operation. 
\begin{lem}
\label{lem:product}~

\begin{enumerate}
\item Let $p\in[2,\infty]$, $\alpha\in(1/p,1)$ and $\beta\in(1-\alpha,1)$.
If $f\in\mathcal{B}_{p,\infty}^{\alpha}$ and $g\in\mathcal{B}_{p,\infty}^{\beta-1}$,
then 
\[
\|fg\|_{\beta-1,p,\infty}\lesssim\|f\|_{\alpha,p,\infty}\|g\|_{\beta-1,p,\infty}.
\]
\item Let $p\in[2,\infty]$ and $\beta\in[0,1)$ be such that $\frac{1}{p}+\beta>1$.
If $f\in\mathcal{B}_{p,\infty}^{\frac{1}{p}}\cap L^{\infty}$ and
$g\in\mathcal{B}_{p,\infty}^{\beta-1}$, then 
\[
\|fg\|_{\beta-1,p,\infty}\lesssim\big(\|f\|_{\frac{1}{p},p,\infty}+\|f\|_{\infty}\big)\|g\|_{\beta-1,p,\infty}.
\]
\item Let $p\in[3,\infty]$, $\alpha\in(1/p,1)$ and $\beta>0$ such that
$\alpha+\beta<1$ and $2\alpha+\beta>1$. If $f\in L^{\infty}\cup\mathcal{B}_{p,\infty}^{\alpha}$
and $g\in\mathcal{B}_{p/2,\infty}^{\alpha+\beta-1}$, then 
\[
\|fg\|_{\alpha+\beta-1,p/2,\infty}\lesssim\big(\|f\|_{\infty}\|g\|_{2\alpha+\beta-1,p/3,\infty}\big)\wedge\big(\|f\|_{\alpha,p,\infty}\|g\|_{\alpha+\beta-1,p/2,\infty}\big).
\]
\item Let $p\in[2,\infty]$ and $\alpha\in(1/p,1)$. If $f\in\mathcal{B}_{p,\infty}^{\alpha}$
and $g\in\mathcal{B}_{p,\infty}^{\alpha}$, then 
\[
\|fg\|_{\alpha,p,\infty}\lesssim\|f\|_{\alpha,p,\infty}\|g\|_{\alpha,p,\infty}.
\]
\item If $f\in\mathcal{B}_{p,\infty}^{\frac{1}{p}}\cap L^{\infty}$ and
$g\in\mathcal{B}_{p,\infty}^{\frac{1}{p}}\cap L^{\infty}$ with $p\in[2,\infty]$,
then 
\[
\|fg\|_{\frac{1}{p},p,\infty}\lesssim\big(\|f\|_{\frac{1}{p},p,\infty}+\|f\|_{\infty}\big)\big(\|g\|_{\frac{1}{p},p,\infty}+\|g\|_{\infty}\big).
\]
\end{enumerate}
\end{lem}

\begin{proof}
Applying Besov embedding ($\alpha>1/p$) and Bony's estimates (Lemma~\ref{lem:paraproduct})
lead to:
\begin{align*}
(i) &  & \|fg\|_{\beta-1,p,\infty} & \lesssim\|T_{f}g\|_{\beta-1,p,\infty}+\|\pi(f,g)\|_{\alpha+\beta-1,p/2,\infty}+\|T_{g}f\|_{\alpha+\beta-1,p/2,\infty}\\
 &  &  & \lesssim\|f\|_{\alpha,p,\infty}\|g\|_{\beta-1,p,\infty},\\
(ii) &  & \|fg\|_{\beta-1,p,\infty} & \lesssim\|T_{f}g\|_{\beta-1,p,\infty}+\|\pi(f,g)\|_{\frac{1}{p}+\beta-1,p/2,\infty}+\|T_{g}f\|_{\frac{1}{p}+\beta-1,p/2,\infty}\\
 &  &  & \lesssim\big(\|f\|_{\frac{1}{p},p,\infty}+\|f\|_{\infty}\big)\|g\|_{\beta-1,p,\infty},\\
(iii) &  & \|fg\|_{\alpha+\beta-1,p/2,\infty} & \lesssim\|T_{f}g\|_{\alpha+\beta-1,p/2,\infty}+\|\pi(f,g)\|_{2\alpha+\beta-1,p/3,\infty}+\|T_{g}f\|_{\alpha+\beta-1,p/2,\infty}\\
 &  &  & \lesssim\|f\|_{\infty}\|g\|_{\alpha+\beta-1,p/2,\infty}+\|g\|_{\alpha+\beta-1,p/2,\infty}\|f\|_{0,\infty,\infty}\\
 &  &  & \qquad+\big(\|f\|_{0,\infty,\infty}\|g\|_{2\alpha+\beta-1,p/3,\infty}\wedge\|f\|_{\alpha,p,\infty}\|g\|_{\alpha+\beta-1,p/2,\infty}\big)\\
 &  &  & \lesssim\big(\|f\|_{\infty}\|g\|_{2\alpha+\beta-1,p/3,\infty}\big)\wedge\big(\|f\|_{\alpha,p,\infty}\|g\|_{\alpha+\beta-1,p/2,\infty}\big),\\
(vi) &  & \|fg\|_{\alpha,p,\infty} & \lesssim\|T_{f}g\|_{\alpha,p,\infty}+\|\pi(f,g)\|_{2\alpha,p/2,\infty}+\|T_{g}f\|_{a,p,\infty}\lesssim\|f\|_{\alpha,p,\infty}\|g\|_{a,p,\infty},\\
(v) &  & \|fg\|_{\frac{1}{p},p,\infty} & \lesssim\|T_{f}g\|_{\frac{1}{p},p,\infty}+\|\pi(f,g)\|_{\frac{2}{p},p/2,\infty}+\|T_{g}f\|_{\frac{1}{p},p,\infty}\\
 &  &  & \lesssim\big(\|f\|_{\frac{1}{p},p,\infty}+\|f\|_{\infty}\big)\big(\|g\|_{\frac{1}{p},p,\infty}+\|g\|_{\infty}\big).\tag*{{\qedhere}}
\end{align*}
\end{proof}
The following estimates are crucial to obtain the existence of a solution
to the Volterra equation~(\ref{eq:convol}) and the local Lipschitz
continuity of the corresponding Itô-Lyons map~(\ref{eq:itomap}).
\begin{lem}
\label{lem:differenceBesovNorm}~

\begin{enumerate}
\item Let $\alpha>0$ and $p\in[1,\infty]$. If $f\in\mathcal{B}_{p,\infty}^{\alpha}\cap L^{\infty}$
and $F\in C^{\lceil\alpha\rceil}\mbox{ with }F(0)=0$, then
\[
\|F(f)\|_{\alpha,p,\infty}\lesssim\|F\|_{C^{\lceil\alpha\rceil}}\|f\|_{\alpha,p,\infty}.
\]
\item Let $\alpha\in(1/p,1]$ and $p\in[2,\infty]$. If $f,g\in\mathcal{B}_{p,\infty}^{\alpha}$
and $F\in C^{2}$, then
\[
\|F(f)-F(g)\|_{\alpha,p,\infty}\lesssim\|F\|_{C^{2}}\big(1+\|f\|_{\alpha,p,\infty}+\|g\|_{\alpha,p,\infty}\big)\|f-g\|_{\alpha,p,\infty}.
\]
\item If $f,g\in\mathcal{B}_{p,\infty}^{\frac{1}{p}}\cap L^{\infty}$ and
$F\in C^{2}$ with $p\in[2,\infty]$, then
\[
\|F(f)-F(g)\|_{\frac{1}{p},p,\infty}\lesssim\|F\|_{C^{2}}\big(1+\|f\|_{\frac{1}{p},p,\infty}+\|f\|_{\infty}+\|g\|_{\frac{1}{p},p,\infty}+\|g\|_{\infty}\big)\big(\|f-g\|_{\frac{1}{p},p,\infty}+\|f-g\|_{\infty}\big).
\]
\end{enumerate}
\end{lem}

\begin{proof}
(i) can be deduced from \citep[Theorem~2.87]{Bahouri2011}. 

For (ii) we apply Lemma~\ref{lem:product} (iv) and the first part
of this lemma to obtain

\begin{align*}
\|F(f)-F(g)\|_{\alpha,p,\infty} & \leq\int_{0}^{1}\|F^{\prime}(f+s(g-f))(f-g)\|_{\alpha,p,\infty}\d s\\
 & \lesssim\|f-g\|_{\alpha,p,\infty}\int_{0}^{1}\|F^{\prime}(f+s(g-f))\|_{\alpha,p,\infty}\d s\\
 & \lesssim\|F\|_{C^{2}}\big(1+\|f\|_{\alpha,p,\infty}+\|g\|_{\alpha,p,\infty}\big)\|f-g\|_{\alpha,p,\infty}.
\end{align*}
For (iii) we apply an analogous estimate, but use Lemma~\ref{lem:product}
(v) instead of (iv).
\end{proof}
We also need this linearization lemma:
\begin{lem}
\label{lem:linearization} Let $\sigma\in C^{2}$, $p\ge1$ and $\alpha>1/p$.
Supposing $u=T_{u^{(1)}}w_{1}+T_{u^{(2)}}w_{2}+u^{\#}\in\mathcal{B}_{p,\infty}^{\alpha}$
with $u^{(1)},u^{(2)},w_{1},w_{2}\in\mathcal{B}_{p,\infty}^{\alpha}$
and $u^{\#}\in\mathcal{B}_{p/2,\infty}^{2\alpha}$, we have
\[
\sigma(u)=\sigma(0)+T_{\sigma'(u)}u+S_{\sigma}(u)
\]
for a function $S_{\sigma}(u)\in\mathcal{B}_{p/2,\infty}^{2\alpha}$
satisfying 
\[
\|S_{\sigma}(u)\|_{2\alpha,p/2,\infty}\lesssim\|\sigma\|_{C^{2}}\Big(1+\sum_{j=1,2}\|u^{(j)}\|_{\infty}\|w_{j}\|_{\alpha,p,\infty}\Big)\big(\|u\|_{\alpha,p,\infty}+\|u^{\#}\|_{2\alpha,p/2,\infty}\big).
\]
\end{lem}

\begin{proof}
The proof follows from Step 1 in the proof of \citep[Proposition~5.6]{Promel2015}
with $\tilde{u}=u$ and $v_{u}=T_{u^{(1)}}w_{1}+T_{u^{(2)}}w_{2}$.
\end{proof}
A refinement of \citep[Lemma 2.3]{Promel2015} is given by the following
result:
\begin{lem}
\label{lem:dialation} Let $\lambda,\gamma>0$, $p\ge1$ and $f\in\mathcal{B}_{p,\infty}^{\gamma}$.
We have for any $\gamma'\in[0,\gamma)\cap[0,1/p]$:
\begin{enumerate}
\item If $\chi\in\mathcal{B}_{1/\gamma',\infty}^{\gamma}$, then
\[
\|\chi\Lambda_{\lambda}f\|_{\gamma,p,\infty}\lesssim\lambda^{\gamma'-1/p}|\log\lambda|\|f\|_{\gamma,p,\infty}\|\chi\|_{\gamma,1/\gamma',\infty}.
\]
\item If additionally $xf(x)\in\mathcal{B}_{p,\infty}^{\gamma+1+\epsilon}$
for some $\epsilon>0$, then we have for any functions $\chi_{1},\chi_{2}$
such that $C_{\chi}:=\|\chi_{1}\|_{\gamma+1,\infty,\infty}(\|\chi_{2}\|_{\gamma+1,p,\infty}+\|x\chi_{2}(x)\|_{L^{1/\gamma'}})$
is finite and for any $\lambda\in(0,1)$
\[
\big\|\chi_{1}(x)\chi_{2}(x)\Lambda_{\lambda}\big(xf(x)\big)\big\|_{\gamma+1,p,\infty}\lesssim\lambda^{1+\gamma'-1/p}|\log\lambda|C_{\chi}\big(\|xf(x)\|_{\gamma+1+\epsilon,p,\infty}+\|f\|_{\gamma,p,\infty}\big).
\]
\end{enumerate}
\end{lem}

\begin{proof}
We decompose $\chi\Lambda_{\lambda}f$ into small and larger Littlewood-Paley
blocks. Arguing as in \citep[Lemma 2.3]{Promel2015} for the $\Delta_{-1}$
block, we have for the small blocks
\begin{align}
\begin{split}\bigg\|\sum_{j\lesssim1}\Delta_{j}(\chi\Lambda_{\lambda}f)\bigg\|_{\gamma,p,\infty} & =\bigg\|\sum_{j\lesssim1}\Delta_{j}\Lambda_{\lambda}(\chi(\lambda^{-1}\cdot)f)\bigg\|_{\gamma,p,\infty}\\
 & \lesssim\sum_{j:2^{j}\lesssim\lambda^{-1}\vee1}\lambda^{-1/p}\big\|\Delta_{j}\big(\chi(\lambda^{-1}\cdot)f\big)\big\|_{L^{p}}\\
 & \lesssim\lambda^{-1/p}|\log\lambda|\|\chi(\lambda^{-1}\cdot)f\|_{0,p,\infty}\lesssim\lambda^{-1/p}|\log\lambda|\|\chi(\lambda^{-1}\cdot)f\|_{L^{p}}.
\end{split}
\label{eq:dialationSmallBlocks}
\end{align}
For any $\gamma'\in[0,\gamma)\cap[0,1/p]$ and $q\ge p$ satisfying
$\frac{1}{p}=\gamma'+\frac{1}{q}$ Hölder's inequality yields (with
convention $1/0=:\infty$)
\[
\|\chi(\lambda^{-1}\cdot)f\|_{L^{p}}\le\|\chi(\lambda^{-1}\cdot)\|_{L^{1/\gamma'}}\|f\|_{L^{q}}\lesssim\lambda^{\gamma'}\|\chi\|_{L^{1/\gamma'}}\|f\|_{\gamma,p,\infty},
\]
which gives the asserted bound for blocks $\Delta_{j}$ with $j$
smaller than a fixed constant. 

Hence, we are left to bound the higher Littlewood-Paley blocks. Using
Bony's decomposition, we get

\begin{align}
\begin{split}\bigg\|\sum_{j\gtrsim1}\Delta_{j}(\chi\Lambda_{\lambda}f)\bigg\|_{\gamma,p,\infty} & \le\bigg\|\sum_{j\gtrsim1}\Delta_{j}T_{\chi}(\Lambda_{\lambda}f)\bigg\|_{\gamma,p,\infty}\\
 & \quad+\bigg\|\sum_{j\gtrsim1}\Delta_{j}T_{\Lambda_{\lambda}f}\chi\bigg\|_{\gamma,p,\infty}+\bigg\|\sum_{j\gtrsim1}\Delta_{j}\pi(\chi,\Lambda_{\lambda}f)\bigg\|_{\gamma,p,\infty}.
\end{split}
\label{eq:dialationDecomp}
\end{align}
We will estimate these three terms separately. By the support properties
of the Littlewood-Paley blocks in the Fourier domain we have $\Delta_{j}T_{\chi}(\Lambda_{\lambda}f)=\Delta_{j}\sum_{k\sim j}S_{k-1}\chi\Delta_{k}(\Lambda_{\lambda}f).$
Therefore,
\[
2^{j\gamma}\|\Delta_{j}T_{\chi}(\Lambda_{\lambda}f)\|_{L^{p}}\lesssim2^{j\gamma}\sum_{k\sim j}\|S_{k-1}\chi\|_{L^{\infty}}\|\Delta_{k}(\Lambda_{\lambda}f)\|_{L^{p}}\lesssim\|\chi\|_{\infty}\big\|\big(2^{k\gamma}\|\Delta_{k}(\Lambda_{\lambda}f)\|_{L^{p}}\big)_{k\ge0}\big\|_{\ell^{\infty}}.
\]
The last norm in the previous display can be estimated as in \citep[Lem. 2.3]{Promel2015},
which yields
\[
\sup_{j\gtrsim1}2^{j\gamma}\|\Delta_{j}T_{\chi}(\Lambda_{\lambda}f)\|_{L^{p}}\lesssim\lambda^{\gamma-1/p}|\log\lambda|\|\chi\|_{\infty}\|f\|_{\gamma,p,\infty}.
\]
For the second term in (\ref{eq:dialationDecomp}) we note with $\gamma'$
and $q$ as above that
\[
2^{j\gamma}\|\Delta_{j}T_{\Lambda_{\lambda}f}\chi\|_{L^{p}}\lesssim2^{j\gamma}\sum_{k\sim j}\|S_{k-1}\Lambda_{\lambda}f\|_{L^{q}}\|\Delta_{k}\chi\|_{L^{1/\gamma'}}\lesssim\|f(\lambda\cdot)\|_{L^{q}}\|\chi\|_{\gamma,1/\gamma',\infty},
\]
where $\|f(\lambda\cdot)\|_{L^{q}}=\lambda^{-1/q}\|f\|_{L^{q}}\lesssim\lambda^{\gamma'-1/p}\|f\|_{\gamma,p,\infty}$.
Finally, the third term in (\ref{eq:dialationDecomp}) is bounded
by 
\begin{align*}
2^{j\gamma}\|\Delta_{j}\pi(\chi,\Lambda_{\lambda}f)\|_{L^{p}} & \lesssim2^{j\gamma}\sum_{k\gtrsim j}\Big\|\sum_{|l|\le1}\Delta_{k-l}\chi\Delta_{k}\Lambda_{\lambda}f\Big\|_{L^{p}}\\
 & \lesssim\sum_{k\gtrsim j}2^{-(k-j)\gamma}\sum_{|l|\le1}\|\Delta_{k-l}\chi\|_{\infty}2^{k\gamma}\|\Delta_{k}\Lambda_{\lambda}f\|_{L^{p}}\\
 & \lesssim\|\chi\|_{\infty}\big\|\big(2^{k\gamma}\|\Delta_{k}(\Lambda_{\lambda}f)\|_{L^{p}}\big)_{k\ge0}\big\|_{\ell^{\infty}}\lesssim\lambda^{\gamma-1/p}|\log\lambda|\|\chi\|_{\infty}\|f\|_{\gamma,p,\infty}.
\end{align*}
For part (i) it remains to note that $\|\chi\|_{L^{1/\gamma'}}\le\|\chi\|_{\gamma,1/\gamma',\infty}$
and $\|\chi\|_{\infty}\lesssim\|\chi\|_{\gamma-\gamma',\infty,\infty}\lesssim\|\chi\|_{\gamma,1/\gamma',\infty}$
due to Besov embeddings.

For (ii) we first note for the small blocks as in (\ref{eq:dialationSmallBlocks})
\begin{align*}
\Big\|\sum_{j\lesssim1}\Delta_{j}\big(\chi_{1}(x)\chi_{2}(x)\Lambda_{\lambda}(xf(x))\big)\Big\|_{\gamma+1,p,\infty} & \lesssim\sum_{j:\lambda2^{j}\lesssim1}\lambda^{-1/p}\|\Delta_{j}\big(\chi_{1}(x/\lambda)\chi_{2}(x/\lambda)xf(x)\big)\big\|_{L^{p}}\\
 & \lesssim\lambda^{-1/p}|\log\lambda|\|\chi_{1}(x/\lambda)\chi_{2}(x/\lambda)xf(x)\|_{L^{p}}\\
 & \lesssim\lambda^{-1/p}|\log\lambda|\|x\chi_{1}(x/\lambda)\chi_{2}(x/\lambda)\|_{L^{1/\gamma'}}\|f\|_{L^{q}}\\
 & \lesssim\lambda^{\gamma'+1-1/p}|\log\lambda|\|x\chi_{1}(x)\|_{L^{1/\gamma'}}\|\chi_{2}\|_{\infty}\|f\|_{\gamma,p,\infty}.
\end{align*}
For the large blocks we obtain as in (i) 
\begin{align*}
 & \Big\|\sum_{j\gtrsim1}\Delta_{j}\big(\chi_{1}(x)\chi_{2}(x)\Lambda_{\lambda}(xf(x))\big)\Big\|_{\gamma+1,p,\infty}\\
 & \quad\lesssim\lambda^{\gamma+1-1/p}|\log\lambda|\|\chi_{1}\|_{\infty}\|\chi_{2}(\lambda^{-1}x)xf(x)\|_{\gamma+1,p,\infty}+\|\chi_{2}(x)\Lambda_{\lambda}(xf(x))\|_{L^{p}}\|\chi_{1}\|_{\gamma+1,\infty,\infty}.
\end{align*}
Since
\begin{align*}
\|\chi_{2}(x)\Lambda_{\lambda}(xf(x))\|_{L^{p}} & =\lambda\|\chi_{2}(x)xf(\lambda x)\|_{L^{p}}\lesssim\lambda\|x\chi_{2}(x)\|_{L^{1/\gamma'}}\|f(\lambda x)\|_{L^{q}}\\
 & \lesssim\lambda^{\gamma'+1-1/p}\|x\chi_{2}(x)\|_{L^{1/\gamma'}}\|f\|_{\gamma,p,\infty},
\end{align*}
we only need a uniform bound for $\|\chi_{2}(\lambda^{-1}x)xf(x)\|_{\gamma+1,p,\infty}$
for which we apply (i) with $\gamma'=p^{-1}-\epsilon<1$:
\begin{align*}
\|\chi_{2}(\lambda^{-1}x)xf(x)\|_{\gamma+1,p,\infty} & =\|xf(x)\Lambda_{\lambda^{-1}}\chi_{2}\|_{\gamma+1,p,\infty}\\
 & \lesssim\lambda^{\epsilon}|\log\lambda|\|\chi_{2}\|_{\gamma+1,p,\infty}\|xf(x)\|_{\gamma+1,1/\gamma',\infty}\\
 & \lesssim\|\chi_{2}\|_{\gamma+1,p,\infty}\|xf(x)\|_{\gamma+1+\epsilon,p,\infty},
\end{align*}
where the last estimate follows from the embedding $\mathcal{B}_{p,\infty}^{1+\gamma+\epsilon}\subset\mathcal{B}_{\gamma',\infty}^{1+\gamma}$.
\end{proof}
Finally, we estimate the Besov norms of the scaled resonant term.
\begin{lem}
\label{lem:scaling resonant term} For $\alpha,\beta\in\R$, $p\ge2$,
$f,g\in\mathcal{S}$ we have uniformly in $\lambda\in(0,1]$ that
\[
\big\|\Lambda_{\lambda}\pi(f,g)-\pi(\Lambda_{\lambda}f,\Lambda_{\lambda}g)\big\|_{\alpha+\beta,p/2,\infty}\lesssim\lambda^{-|\alpha+\beta|-p/2}\|f\|_{\alpha,p,\infty}\|g\|_{\beta,p,\infty}+\|\Lambda_{\lambda}f\|_{\alpha,p,\infty}\|\Lambda_{\lambda}g\|_{\beta,p,\infty}.
\]
\end{lem}

\begin{proof}
We proceed by generalising the proofs of \citep[Lem. B.1]{Gubinelli2015}
and of \citep[Theorem~2.1]{Bony1981}. Let us choose $K=K(\lambda)\in\N$
such that $\lambda':=\lambda2^{K}\in(1/2,1]$ and decompose
\begin{align}
\begin{split}\Lambda_{\lambda}\pi(f,g) & =\sum_{j,k<K:|k-j|\le1}\Lambda_{\lambda}\Delta_{i}f\Delta_{k}g\\
 & \qquad+\sum_{j,k\ge K:|k-j|\le1}\big(\F^{-1}[\rho(2^{-j+K}\lambda'^{-1}\cdot)]\ast\Lambda_{\lambda}f\big)\big(\F^{-1}[\rho(2^{-k+K}\lambda'^{-1}\cdot)]\ast\Lambda_{\lambda}g\big).
\end{split}
\label{eq:lambdaRes}
\end{align}
The Fourier transform of the first term is spectrally supported in
a ball with radius of order $2^{K}\sim\lambda^{-1}$ such that
\begin{align*}
\Big\|\sum_{j,k\le K:|k-j|\le1}\Lambda_{\lambda}\Delta_{i}f\Delta_{k}g\Big\|_{\alpha+\beta,p/2,\infty} & \lesssim(2^{K(\alpha+\beta)}\vee1)\sum_{j,k\le K:|k-j|\le1}\big\|\Lambda_{\lambda}\Delta_{j}f\Delta_{k}g\big\|_{L^{p/2}}\\
 & \lesssim(2^{K(\alpha+\beta)}\vee1)\lambda^{-2/p}\sum_{j,k\le K:|k-j|\le1}\|\Delta_{j}f\|_{L^{p}}\|\Delta_{k}g\|_{L^{p}}\\
 & \lesssim(2^{K(\alpha+\beta)}\vee1)\lambda^{-2/p}\sum_{j,k\le K:|k-j|\le1}2^{-j\alpha-k\beta}\|f\|_{\alpha,p,\infty}\|g\|_{\beta,p,\infty}\\
 & \lesssim(\lambda^{-(\alpha+\beta)}\vee1)(\lambda^{(\alpha+\beta)}\vee1)\lambda^{-2/p}\|f\|_{\alpha,p,\infty}\|g\|_{\beta,p,\infty}.
\end{align*}
The second term in (\ref{eq:lambdaRes}) equals $\pi'(\Lambda_{\lambda}f,\Lambda_{\lambda}g)$
where $\pi'$ is the resonant term corresponding to the modified partition
of unity $(\chi(\cdot/\lambda'),\rho(\cdot/\lambda'))$. Note that
the scaling parameter $\lambda'\in(1/2,1]$ is uniformly bounded from
above and below. It remains to show
\[
\big\|\pi'(f,g)-\pi(f,g)\big\|_{\alpha+\beta,p/2,\infty}\lesssim\|f\|_{\alpha,p,\infty}\|g\|_{\beta,p,\infty}.
\]
Owing to $fg=T'_{g}f+T'_{f}g+\pi'(f,g)$ for the paraproduct operators
$T'_{g}f$ associated to $(\chi(\cdot/\lambda'),\rho(\cdot/\lambda'))$,
we have
\[
\big\|\pi'(f,g)-\pi(f,g)\big\|_{\alpha+\beta,p/2,\infty}\le\big\| T'_{f}g-T_{f}g\big\|_{\alpha+\beta,p/2,\infty}+\big\| T'_{g}f-T_{g}f\big\|_{\alpha+\beta,p/2,\infty}.
\]
Generalising \citep[Thm. 2.1]{Bony1981}, we will now prove
\begin{equation}
\big\| T_{g}f-T_{g}^{*}f\big\|_{\alpha+\beta,p/2,\infty}\lesssim\|f\|_{\alpha,p,\infty}\|g\|_{\beta,p,\infty}\label{eq:BonyThm21}
\end{equation}
for the operator
\[
T_{g}^{*}f:=\F^{-1}\Big[\int_{\R}\chi(u-v,v)\F g(u-v)\F f(v)\d v\Big]
\]
where $\chi\colon\R^{2}\setminus\{0\}\to[0,1]$ is a $C^{\infty}$-function
such that for sufficiently small constants $0<\epsilon_{1}<\epsilon_{2}$:
\[
\chi(u,v)=\begin{cases}
1, & |u|\le\epsilon_{1}|v|\\
0, & |u|\ge\epsilon_{2}|v|
\end{cases}.
\]
The estimate (\ref{eq:BonyThm21}) especially implies that $T_{g}f$
and thus $\pi(f,g)$ does not depend on the choice of the partition
of unity up to a regular remainder, which concludes the proof. 

To verify (\ref{eq:BonyThm21}), we decompose
\[
\F\big[T_{g}^{*}f]=\sum_{j,k}\int_{\R}\chi(u-v,v)\F[\Delta_{k}g](u-v)\F[\Delta_{j}f](v)\d v.
\]
Due to the support assumption on $\chi$, the terms with $2^{k}\gtrsim2^{j-1}$
are zero and the integrands with $2^{k}\lesssim2^{j-1}$ coincide
with $\F[\Delta_{k}g]\ast\F[\Delta_{j}f](u)$. Therefore, for integers
$N_{1}<N_{2}$ depending only on $\epsilon_{1}$, $\epsilon_{1}$
and $(\chi,\rho)$, respectively, we have
\[
T_{g}^{*}f=\sum_{j}\sum_{k<j-N_{1}}\text{\ensuremath{\Delta_{k}g}\ensuremath{\Delta_{j}f}}+R(g,f)
\]
with
\[
R(g,f)=\sum_{j}R_{j}(g,f),\quad R_{j}(g,f):=\sum_{k=j-N_{1}}^{j-N_{2}}\F^{-1}\Big[\int_{\R}\chi(u-v,v)\F[\Delta_{k}g](u-v)\F[\Delta_{j}f](v)\d v\Big].
\]
Fubini's theorem yields 
\begin{align*}
R_{j}(g,f)(x) & =\sum_{k=j-N_{1}}^{j-N_{2}}\frac{1}{2\pi}\int_{\R}\int_{\R}e^{ix(u+v)}\chi(u,v)\F[\Delta_{k}g](u)\F[\Delta_{j}f](v)\d v\d u\\
 & =\sum_{k=j-N_{1}}^{j-N_{2}}\int_{\R}\int_{\R}\F^{-1}[\chi](s,t)\Delta_{k}g(x-s)\Delta_{j}f(x-t)\d s\d t.
\end{align*}
Since $\F^{-1}\chi\in L^{1}(\R^{2})$ due to the regularity of $\chi$,
Young's inequality implies
\[
\|R_{j}(g,f)\|_{L^{p/2}}\lesssim\sum_{k=j-N_{1}}^{j-N_{2}}\|\F^{-1}\chi\|_{L^{1}}\|\Delta_{k}g\|_{L^{p}}\|\Delta_{j}f\|_{L^{p}}.
\]
Noting that $R_{j}(g,f)$ is spectrally supported in an annulus with
radius of order $2^{j}C$ for some $C>0$, we obtain 
\begin{align*}
\|R(g,f)\|_{\alpha+\beta,p/2,\infty} & \lesssim\sup_{m}2^{m(\alpha+\beta)}\sum_{2^{m}\sim2^{j}C}\sum_{k\sim j}\|\Delta_{k}g\|_{L^{p}}\|\Delta_{j}f\|_{L^{p}}\lesssim\|f\|_{\alpha,p,\infty}\|g\|_{\beta,p,\infty}
\end{align*}
and thus (\ref{eq:BonyThm21}).
\end{proof}
\begin{singlespace}
{\footnotesize{}\bibliographystyle{chicago}
\bibliography{quellen}
}{\footnotesize\par}
\end{singlespace}

\end{document}